\documentclass[10pt]{amsart}

\usepackage{amsmath}
\usepackage{amsthm}
\usepackage{amssymb}
\usepackage[shortlabels]{enumitem}
\usepackage[all]{xy}
\usepackage[hidelinks]{hyperref}
\usepackage{url}

\usepackage{graphicx}
\usepackage{pdflscape}

\usepackage{tikz}
\usetikzlibrary{babel,calc,matrix,graphs,positioning,arrows.meta, quotes,shapes.geometric}

\usepackage{color}

\usepackage{bm}

\usepackage[Symbol]{upgreek}

\DeclareFontFamily{OMS}{smallo}{}
\DeclareFontShape{OMS}{smallo}{m}{n}{<->s*[.65]cmsy10}{}
\DeclareSymbolFont{smallo@m}{OMS}{smallo}{m}{n}
\DeclareMathSymbol{\smallo}{\mathord}{smallo@m}{79}

\newtheorem{thm}{Theorem}[section]
\newtheorem*{thmunnumbered}{Theorem}

\newtheorem{lemma}[thm]{Lemma}
\newtheorem{cor}[thm]{Corollary}
\newtheorem{prop}[thm]{Proposition}

\newtheorem{G}[thm]{G}

\theoremstyle{definition}

\newtheorem{definition}[thm]{Definition}

\theoremstyle{remark}

\newtheorem{remark}[thm]{Remark}
\newtheorem{remark*}[thm]{Remark}
\newtheorem*{remarks*}{Remarks}

\newtheorem*{notationandterminology}{Notation and terminology}

\newcommand\N{\mathbb{N}}
\newcommand\Z{\mathbb{Z}}
\newcommand\Q{\mathbb{Q}}
\newcommand\R{\mathbb{R}}

\newcommand\dd{\mathfrak{d}}
\DeclareMathOperator\id{id}
\newcommand\T{\mathbb{T}}

\newcommand\dominatedby{\preccurlyeq}

\newcommand\dominates{\succcurlyeq}

\DeclareMathOperator\dom{dom}

\newcommand\K{\boldsymbol{k}}

\DeclareMathOperator\ndeg{ndeg}

\DeclareMathOperator\ddeg{ddeg}
\DeclareMathOperator\dmul{dmul}

\DeclareMathOperator\qddeg{qddeg}
\DeclareMathOperator\dwt{dwt}
\DeclareMathOperator\dwm{dwm}
\DeclareMathOperator\mul{mul}
\DeclareMathOperator\wt{wt}
\DeclareMathOperator\wm{wm}
\DeclareMathOperator\swt{swt}

\DeclareMathOperator\nwt{nwt}
\DeclareMathOperator\order{order}

\newcommand\I {\operatorname{I}}
\def \d{\operatorname{d}}

\DeclareMathOperator\Ri{Ri}
\DeclareMathOperator\ex{e}

\def \bomega{{\boldsymbol{\omega}}}
\def \btau{{\boldsymbol{\tau}}}
\newcommand{\dabs}[1]{\lVert#1\rVert}
\def \ev{\operatorname{e}}

\DeclareMathOperator\sdeg{sdeg}

\DeclareFontFamily{U}{fsy}{}
\DeclareFontShape{U}{fsy}{m}{n}{<->s*[.9]psyr}{}
\DeclareSymbolFont{der@m}{U}{fsy}{m}{n}
\DeclareMathSymbol{\der}{\mathord}{der@m}{182}

\newcommand \upg{\upgamma}

\newcommand \upl{\uplambda}
\newcommand \Upl{\Uplambda}
\newcommand \upo{\upomega}
\newcommand \Upo{\Upomega}
\newcommand \upd{\updelta}

\renewcommand\epsilon{\varepsilon}

\DeclareFontFamily{U}{fsy}{}
\DeclareFontShape{U}{fsy}{m}{n}{<->s*[.9]psyr}{}
\DeclareSymbolFont{der@m}{U}{fsy}{m}{n}
\DeclareMathSymbol{\der}{\mathord}{der@m}{182}

\renewcommand\L{\mathcal{L}}

\author{Julian Ziegler Hunts}
\title{Definable Eventual Equalizers}
\email{julian.james.ziegler.hunts@univie.ac.at}
\address{Kurt G\"odel Research Center for Mathematical Logic\\
Universit\"at Wien\\
1090 Wien\\ Austria}
\date{\today}

\begin{document}
\begin{abstract}
The solutions of algebraic differential equations in certain valued differential fields, including the differential field of transseries, can be analyzed using a Newton diagram method.  In this paper, we show that (eventual) equalizers, a crucial part of this process, can be obtained uniformly and definably from the coefficients of the input differential polynomials.  We also obtain similar definability results for a certain compositional conjugation which is used repeatedly as an intermediate simplification step.
\end{abstract}
\maketitle


\section*{Introduction}
\noindent
The field of Laurent series with real coefficients comes with a natural derivation but is too small to be closed under integration and exponentiation.  These defects are cured by passing to a certain canonical extension, the   differential field $\T$ of transseries, introduced by \'Ecalle~\cite{Ecalle}.  Transseries are formal (Hahn) series in an indeterminate $x>\R$ with real coefficients, such as
\[
\begin{split}-3\ex^{\ex^x}&+\ex^{\frac{\ex^x}{\log x}+\frac{\ex^x}{\log^2 x}+\frac{\ex^x}{\log^3x}+\cdots}-x^{11}+7\\
&+\frac{\pi}{x}+\frac{1}{x\log x}+\frac{1}{x\log^2x}+\frac{1}{x\log^3x}+\cdots\\
&+\frac{2}{x^2}+\frac{6}{x^3}+\frac{24}{x^4}+\frac{120}{x^5}+\frac{720}{x^6}+\cdots\\
&+\ex^{-x}+2\ex^{-x^2}+3\ex^{-x^3}+4\ex^{-x^4}+\cdots,
\end{split}
\]
where $\log^2x:=(\log x)^2$, etc. The {\it transmonomials} appearing in the support of such a transseries may be real powers of $x$, like $x^{11}$ or $\frac{1}{x}$; but  they may also  be obtained as the exponential or logarithm of ``simpler'' transseries, like~$\ex^{\ex^x}$, $\ex^{\frac{\ex^x}{\log x}+\frac{\ex^x}{\log^2 x}+\frac{\ex^x}{\log^3x}+\cdots}$, $\frac{1}{x\log x}$, or $\ex^{-x}$. The construction of $\T$ proceeds inductively, starting from $\R$ and $x$, iterating infinite (Hahn) summation, exponentiation, and logarithm.
Transseries are also called  \emph{logarithmic-exponential series} (LE-series) in some places, for example, in~\cite{LEseriespaper}; we refer to that paper, or to Appendix A of~\cite{ADAMTT}, for a detailed construction of~$\T$.
The derivation $\frac{d}{dx}$ of $\T$ has constant field $\R$. 
The field $\T$ also comes equipped with an ordering making it
an ordered field, and a (Krull)  valuation~${v\colon\T^\times=\T\setminus\{0\}\to\Gamma_{\T}}$ whose valuation ring consists of the $f\in\T$ such that~$|f|\leq c$  for some constant~$c\in\R$. (For example, the transseries displayed above is $<0$, since its leading coefficient $-3$ is negative, and it has the same valuation as its leading transmonomial $\ex^{\ex^x}$.)

The primary result of \cite{ADAMTT} is an elimination theory for $\T$.  This relies on an extensive analysis of the solvability of algebraic differential equations over certain valued differential fields (such as $\T$), via a Newton diagram method.  A key ingredient in this process is the existence of \emph{equalizers}, and subsequently \emph{eventual equalizers}, for homogeneous differential polynomials of different degrees.  A more precise analysis of these quantities is a step toward a better understanding of definable sets in $\T$, in particular in light of the issue of uniform bounds for the number of steps required in a co-analysis of the solution set of a given
algebraic differential equation in $\T$ raised in \cite{dim}, and also with an eye towards a possible fully algorithmic treatment of algebraic differential equations over $\T$ and related differential fields.
In this paper we show that (eventual) equalizers can be obtained uniformly and definably from the coefficients of the input differential polynomials. Before we can state our main results, we need to recall some definitions from \cite{ADAMTT}:

\begin{notationandterminology}
 The ring $\T\{Y\}$ of differential polynomials over $\T$ in the single differential indeterminate $Y$ is given the gaussian valuation extending the valuation~$v$ of~$\T$, which we also denote by $v$.
The \emph{degree} of a differential monomial
$$Y^{i_0}(Y')^{i_1}\cdots(Y^{(r)})^{i_r}\qquad(i_0,\dots,i_r\in\N)$$ is~$i_0+i_1+\cdots+i_r$,   its \emph{weight} is $i_1+2i_2+\cdots+ri_r$, and provided that~ $i_r\ne 0$, its \emph{order}\/ is $r$.  The degree (weight, order) of a nonzero~$P\in\T\{Y\}$ is the maximum degree (weight or order, respectively) of a monomial which appears in $P$ with nonzero coefficient.  A nonzero differential polynomial in which all monomials which appear with nonzero coefficient have the same degree is called \emph{homogeneous}. 
For~$P\in\T\{Y\}$, its \emph{multiplicative conjugate by $a\in\T$} is $P_{\times a}:=P(aY)\in\T\{Y\}$.
\end{notationandterminology}

\noindent
The \emph{Equalizer Theorem} \cite[Theorem 6.0.1]{ADAMTT} (applied to $\T$) states that if $P,Q\in\T\{Y\}$ are homogeneous  of distinct degrees $d$ and $e$, respectively, then there is an~${a\in\T^\times}$ with $v(P_{\times a})=v(Q_{\times a})$; moreover, the valuation $v(a)$ of such an~$a$ is unique.  The proof of this theorem in loc.~cit.~is non-constructive. With a little additional work, we can apply Herbrand's Theorem to a very similar result, with the consequence that (for fixed $d$ and $e$ and bounded order), the transseries $a^{d-e}$ can be taken to be one of finitely many differential rational functions in the coefficients of~$P$ and~$Q$.  However, in the present paper we   give an explicit construction of such differential rational functions, leading to the following:

\begin{thm}\label{equalizertheoremT}
    Fix $d,e,w\in\N$ with $d>e$ and $w\geq2$.  Then there exist   differential rational functions~$H_1,\ldots,H_N$ with coefficients in $\Q$, where $N\leq 2^{4w^2+2w}$, such that if $P,Q\in\T\{Y\}$ are homogeneous     of degrees $d$ and $e$, respectively, and weight at most $w$, then for some $i\leq N$ we have, for any $a\in \T^\times$:
    $$
    v(P_{\times a})=v(Q_{\times a})\quad\Longleftrightarrow\quad v(a^{d-e})= v(H_i(P,Q)),$$ where the transseries $H_i(P,Q)\in\T$ is obtained by applying $H_i$ to the coefficients of $P$ and $Q$.  Moreover, these $H_i$ may additionally be chosen to be computable from~$d$,~$e$,~$w$, with numerator and denominator of each $H_i$ having degree and weight at most~${(w+2)\cdot(2w+1)!}$.
\end{thm}

\noindent
The \emph{upward shift} ${f\!\uparrow}$ of  $f\in\T$ is given by replacing each appearance of~$x$ in~$f$ with~$\ex^x$. 
By construction of $\T$, for each $f\in \T$ there is an $n$ such that ${f\!\uparrow^n}$ does
not contain logarithmic terms. (For the transseries displayed above,
shifting upwards once is enough.) As a rule, such {\it exponential}\/ transseries are easier to handle.

Now the map $f\mapsto {f\!\uparrow}$ is an automorphism of the ordered valued field~$\T$, but it is not a differential
field automorphism. To remedy this, we consider
the \emph{compositional conjugate $\T^\phi$ of $\T$ by $\phi\in\T$}, $\phi>0$: the valued field $\T$ equipped with the derivation~$\upd=\phi^{-1}\der$. 
For each $P\in\T\{Y\}$ there is a differential polynomial~${P^\phi\in\T^\phi\{Y\}}$ with~$P^\phi(y)=P(y)$ for any~$y\in \T$,
the \emph{compositional conjugate of~$P$ by $\phi$}.
Then $f\mapsto {f\!\uparrow}$ is an isomorphism $\T^{1/x}\to\T$
of differential fields. More generally, put~$\ell_0:=x$, recursively define $\ell_{n+1}:=\log \ell_n$, and set~$\upgamma_n:=\frac{1}{\ell_0\cdots\ell_n}$, as well as~$\upgamma_{-1}:=1$.  Then    
 $f\mapsto {f\!\uparrow}$ is
a differential field isomorphism  $\T^{\upgamma_n}\to\T^{\upgamma_{n-1}}$.  

For a differential polynomial $P\in\T\{Y\}$, the upward shift ${P\!\uparrow}$ of $P$ is obtained by applying the upward shift to the coefficients of $P^{1/x}=P^{\upgamma_0}\in\T^{\upgamma_0}\{Y\}$.  Then~${P(y)\!\uparrow} = {P\!\uparrow}({y\!\uparrow})$ for each $y\in\T$. For $a\in \T$ we have~$({P\!\uparrow})_{\times a\uparrow}=(P_{\times a}){\uparrow}$. Just as repeated application of the upward shift operation turns transseries into exponential transseries,
it also helps to simplify differential polynomials: this is one of the prevalent themes in~\cite{ADAMTT,van2006transseries}. 
For example, for each nonzero $P\in\T\{Y\}$ there is a 
differential polynomial  $N_P(Y)=Q(Y)\cdot(Y')^d$, where~$0\ne Q\in\R[Y]$ and
$d\in\N$, such that for all sufficiently large $n$ we have~${P\!\uparrow^n}=\mathfrak d\cdot N_P+\text{smaller terms}$, for some transmonomial $\mathfrak d=\mathfrak d_n$. (See \cite[Corollary~13.3.17]{ADAMTT}.)

This motivates 
the \emph{Eventual Equalizer Theorem} \cite[Theorem~13.0.3]{ADAMTT}, which (applied to $\T$) implies that if~$P,Q\in\T\{Y\}$ are homogeneous  of distinct degrees $d$ and~$e$, respectively, then there exists $n_0$ such that for any $a\in\T^\times$ and $n\geq n_0$:
$$v\big((P\!\uparrow^n)_{\times a}\big)=v\big((Q\!\uparrow^n)_{\times a}\big) \quad\Longleftrightarrow\quad v\big((P\!\uparrow^{n_0})_{\times a}\big)=v\big((Q\!\uparrow^{n_0})_{\times a}\big).$$  
One can show that for $P$ and $Q$ of bounded order, whether a given $n_0$ works is a definable property of the coefficients of $P$, $Q$, but we can do a little better.  (Her\-brand's Theorem is not applicable here, as the Eventual Equalizer Theorem is not known to hold in a setting with a reasonable universal axiomatization. Also note that $\mathbb T$ does not have definable Skolem functions~\cite[Corollary~6.5]{dim}.)
To state this,  
let $\L=\{0,1,{+},{\,\cdot\,},\der,{\dominatedby}\}$ be the   language of valued differential fields (see \cite[p.~678]{ADAMTT})   and let $d$, $e$, $w$ 
be as in Theorem~\ref{equalizertheoremT}.

\begin{thm}\label{eventualequalizertheoremT}
    There exists a map $(P,Q)\mapsto\gamma(P,Q)\in\Gamma_{\T}$, where~$P$ and~$Q$ range over homogeneous differential polynomials in $\T\{Y\}$ of degrees $d$ and $e$, respectively, and weight at most~$w$, such that 
    \begin{enumerate}
        \item[\textup{(i)}] the relation 
    $v(g)=\gamma(P,Q)$ on $g\in\T,P,Q$ is definable in  $\T$ by both an existential $\mathcal L$-formula and a universal $\mathcal L$-formula;
    \item[\textup{(ii)}]  there is an $n$ such that $v(\upgamma_{n})\ge \gamma(P,Q)$; and
    \item[\textup{(iii)}] any such $n$ will function as $n_0$ in the Eventual Equalizer Theorem.
    \end{enumerate}
\end{thm}

\noindent
The proof of \cite[Proposition 8.14]{van2006transseries} can be adapted to give a similar result, with a bound that depends on first applying the upward shift sufficiently often to remove all logarithms from the coefficients of $P$ and $Q$; but this bound is not definable in the natural languages on $\T$.  This argument is given in  Appendix~\ref{sec:upshift lower bound}, including an explicit lower bound on $n_0$. However,
the full version of our main result (Theorem~\ref{eventualequalizerthm}) also applies to more general valued differential fields considered in
\cite{ADAMTT} instead of $\T$. Note that the $\L$-formulas mentioned in (i)
may be replaced by quantifier-free formulas in the expansion  
of the language $\L$ by a unary function symbol $\iota$ interpreted by multiplicative inversion and three unary predicate symbols $\I$, $\Upl$, $\Upo$,   interpreted by certain convex subsets of $\T$, which was introduced in \cite[Chapter~16]{ADAMTT}, and in which~$\T$ admits quantifier elimination \cite[Theorem~16.0.1]{ADAMTT}.

In Section~\ref{sec:prelims} we set notation and conventions, introduce these more general settings in which the above theorems are naturally formulated, and recall or prove some basic results.  Section~\ref{equalizersection} is dedicated to the proof of the Equalizer Theorem.  Section~\ref{sec:clean} covers a theorem on compositional conjugation, which is applied in Section~\ref{sec:evequ}   to obtain the Eventual Equalizer Theorem and some related results.

This paper constitutes part of the PhD thesis of the author, written under the guidance of Matthias Aschenbrenner at the University of Vienna.

\textbf{Acknowledgements.}  The author wishes to thank Matthias Aschenbrenner and Allen Gehret for many helpful comments and suggestions, Joris van der Hoeven for suggesting Remark~\ref{multivariableequalizers}, and Nigel Pynne-Coates for feedback on this note.

\section{Preliminaries}\label{sec:prelims}

\noindent
In the rest of this paper we assume familiarity with the basic setup of \cite{ADAMTT}. Nevertheless, for convenience of the reader, 
in this section we  recall various notions which are studied in asymptotic differential algebra and results that we will use.  We also define a few new concepts, and state or prove some simple lemmas.

Subsection~\ref{subsec:conventions} contains definitions for most of the standard notions from asymptotic differential algebra that we need and several of their properties, and defines \emph{super-weight} and \emph{super-isobaric}, which we use in formulating complexity bounds.  Section~\ref{equalizersection} relies only on the material in this section through Lemma~\ref{superweightbounds}.
Subsection~\ref{subsec:asymptotic, s, d} deals with asymptotic fields and the maps $s$, $d_{\upl}$, and $d_{\upo}$, of which the latter two are newly defined in this paper, establishing some basic properties. Finally,
subsection~\ref{subsec:definable on value group} deals with the precise notions of definability used later.

\subsection{Conventions and notations}\label{subsec:conventions}
In this paper,  $d$,  $i$, $j$,  $n$, $r$, and $w$ will range over the set $\N=\{0,1,2,\dots\}$ of natural numbers.  Boldface indices $
\bm{i}$, $\bm{j}$, $\bm{k}$, $\bm{l}$ (sometimes with decorations) will denote tuples $\bm{i}=(i_0,\ldots,i_r)\in\N^{1+r}$ for an appropriate~$r$; here, the inequality $\bm{i}\leq\bm{j}$ means $i_k\leq j_k$ for   $k=0,\dots,r$. For~$\bm{i}_1,\dots,\bm{i}_m\le\bm{j}$ we also set $\binom{\bm{j}}{\bm{i}_1,\ldots,\bm{i}_m}:=\prod_k\binom{j_k}{i_{1,k},\ldots,i_{m,k}}$.

For subsets~$A$,~$B$ of a (totally) ordered set $S$, $A<B$ means $a<b$ for any $a\in A$, $b\in B$, and $A^\downarrow:=\{s\in S:\text{$s\le a$ for some $a\in A$}\}$.
A subset of a set described by a simple condition is often denoted with a superscript, and when the condition is a comparison to an identity element, the identity element is often omitted, e.g., $\Q^<:=\Q^{<0}:=\{q\in\Q:q<0\}$. For a field $K$ we put $K^\times:=K\setminus\{0\}$.

\subsubsection{Valued differential fields}
Throughout, $K$ denotes a valued differential field, that is, a valued field of equicharacteristic zero equipped with a derivation, and~$a$,~$b$ range over $K$.  The valuation of $K$ is denoted by $v\colon K^\times\to\Gamma$, where $\Gamma=v(K^\times)$  is an additively written ordered abelian group (with ordering  $\le$), and the derivation of $K$ by $\der$.  We let $\alpha$, $\beta$, $\gamma$ range over $\Gamma$ and $\theta$     over $K^\times$.  We extend $v$ to a map~$K\to\Gamma_{\infty}=\Gamma\cup\{\infty\}$  by setting $v(0):=\infty$, where $\infty>\Gamma$.  The strict and non-strict dominance relations induced by $v$ are denoted by $\prec$ and $\dominatedby$, respectively: 
$$a\dominatedby b\iff va\geq vb\quad\text{and}\quad a\prec b\iff va>vb.$$  If $va=vb$, we write $a\asymp b$, and if $v(a-b)>va$, then we write $a\sim b$.  The constant field of $K$ is $C:=\{a:a'=0\}$.  We assume that the valuation and derivation are both nontrivial, i.e., $\Gamma\neq\{0\}$ and $C\neq K$.

When we wish to display the dependence on $K$, we write $v_K$, $\Gamma_K$, $\der_K$, etc.  When the relevant derivation is clear from context, we usually write $a',a'',a^{(3)},\ldots$ instead of $\der a,\der^2a,\der^3a,\ldots$, and we write $a^\dagger:=a'/a$ for the logarithmic derivative of $a\neq0$. 

We set $a^{\bm{i}}:=a^{i_0}(a')^{i_1}\cdots(a^{(r)})^{i_r}$, and define $$|\bm{i}|:=i_0+i_1+\cdots+i_r,\quad \|\bm{i}\|:=i_1+2i_2+\cdots+ri_r,\quad \|\bm{i}\|'=\|\bm{i}\|+|\bm{i}|.$$  If $|\alpha|\leq n|\beta|$ for some $n$, then we write $\alpha=O(\beta)$, and we say $\alpha=o(\beta)$ if $n|\alpha|<|\beta|$ for all $n$.

\subsubsection{The $\nabla$-map}
The derivation of $K$ is said to be \emph{small} if $a'\prec1$ whenever~$a\prec1$.
In this case, the map $\nabla\colon\Gamma\to\Gamma_\infty$ on $\Gamma$ is defined in~\cite[Section 6.4]{ADAMTT} by $\nabla(0):=\infty$ and $\nabla(\alpha):=\min(va^\dagger,0)$  for~$0\neq\alpha=va$; this is independent of the choice of $a$.  By~\cite[Lem\-ma~6.5.1]{ADAMTT}, this makes $(\Gamma,\nabla)$ into an \emph{asymptotic couple}, meaning that $\nabla$ is a valuation on $\Gamma$ and if $\alpha>0$ then $\alpha+\nabla(\alpha)>\nabla(\Gamma^{\neq})$.

\begin{lemma}[properties of $\nabla$]\label{nablalemma}  Assume $K$ has small derivation.
\begin{enumerate}[label=\textup{(\theenumi)}, ref=\theenumi]

    \item\label{nablacontract} If $\alpha,\beta,\alpha-\beta\ne0$, then $\nabla(\alpha)-\nabla(\beta)=o(\alpha-\beta)$ \textup{\cite[Lem\-ma~6.5.4(ii)]{ADAMTT}}.
    \item\label{nablashrink} If $\alpha\ne0$, then $\nabla(\alpha)=o(\alpha)$ \textup{\cite[Lemma 6.4.1(iii)]{ADAMTT}}.
    \item\label{nablaslow} In particular, if $\beta=O(\alpha)$, $\beta\ne0$, then $\nabla(\beta)=o(\alpha)$, and $\nabla\big(\alpha+
    \nabla(\beta)\big)=\nabla(\alpha)$ by \textup{\cite[\S 6.5 (AC1)]{ADAMTT}}.
    \item\label{derivativeval} Let $\alpha=va$.
    \begin{enumerate}
        \item[\textup{(a)}] If $\nabla\alpha<0$, then $v(a^{\boldsymbol{j}})=|\boldsymbol{j}|\alpha+\|\boldsymbol{j}\|\cdot\nabla\alpha$ \textup{\cite[Lemma~6.4.1(iv)]{ADAMTT}}.
        \item[\textup{(b)}] If $\nabla\alpha\ge0$, then $v(a^{\boldsymbol{j}})\ge|\boldsymbol{j}|\alpha$ \textup{\cite[Lemma 6.5.4(iii) + induction]{ADAMTT}}.
    \end{enumerate}
\end{enumerate}    
\end{lemma}

\subsubsection{Differential polynomials}
The letters $Y$ and $Z$ will denote differential indeterminates, and $K\{Y\}:=K[Y,Y',Y'',\ldots]$ is the differential ring  of differential polynomials (or \emph{$\d$-polynomials}) in $Y$ over~$K$. 
We also let
$K\langle Y\rangle$ be the field of differential rational functions in $Y$ over $K$ (the
differential fraction field of~$K\{Y\}$).
We let $P$, $Q$ range over $K\{Y\}$.  The following discussion mostly follows~\cite[Sections~4.1--4.2,~4.5]{ADAMTT}.
As with field elements, we set~$P^{\bm{i}}:=P^{i_0}(P')^{i_1}\cdots (P^{(r)})^{i_r}$.  So if~$P$ has order at most~$r$, then there are unique $P_{\bm{i}}\in K$, all but finitely many zero, so that
\[
P\ =\ \sum_{\bm{i}}P_{\bm{i}}Y^{\bm{i}}.
\]
Then the \emph{degree of $P$}\/ is $\deg P=\max\{|\bm{i}|:P_{\bm{i}}\neq0\}$, the \emph{multiplicity of $P$}\/ is $\mul P=\min\{|\bm{i}|:P_{\bm{i}}\neq0\}$, the \emph{weight of $P$} is $\wt P=\max\{\|\bm{i}\|:P_{\bm{i}}\neq0\}$, and the \emph{weighted multiplicity of $P$}\/ is $\wm P=\min\{\|\bm{i}\|:P_{\bm{i}}\neq0\}$.  We also define the \emph{super-weight} of $P$, $\swt P:=\max\{\|\bm{i}\|':P_{\bm{i}}\neq0\}$.
Here $\max\emptyset:=-\infty$ by convention. Note that~$|\{\bm{i}:\|\bm{i}\|'\leq s\}|\leq 2^s$ and $\{\bm{i}:|\bm{i}|\leq d,\|\bm{i}\|\leq w\}$ is finite.

The \emph{homogeneous component of $P$ of degree $d$}\/ is denoted by $P_d:=\sum_{|\bm{i}|=d}P_{\bm{i}}Y^{\bm{i}}$, the  \emph{isobaric component of $P$ of weight $w$}\/ is   $P_{[w]}:=\sum_{\|\bm{i}\|=w}P_{\bm{i}}Y^{\bm{i}}$, and the \emph{super-isobaric component of $P$ of super-weight $s$}\/ is $P_{[s]'}:=\sum_{\|\bm{i}\|'=s}P_{\bm{i}}Y^{\bm{i}}$; we say that~$P$ is \emph{homogeneous} of degree $d$ (\emph{isobaric} of weight $w$, \emph{superisobaric} of super-weight~$s$) if $P=P_d$ ($P=P_{[w]}$, $P=P_{[s]'}$, respectively).  We additionally define $P_{[\leq w]}:=\sum_{i\leq w}P_{[i]}$, and $P_{[>w]}:=\sum_{i>w}P_{[i]}$.  Similarly, $K\{Y\}_{d}$ is the subspace of the $K$-linear space $K\{Y\}$ consisting of all homogeneous differential polynomials $P$ of degree $d$; likewise we define~$K\{Y\}_{[w]}$, $K\{Y\}_{\leq d}$,  $K\{Y\}_{[\leq w]}$,  and $K\{Y\}_{[\leq s]'}$.

For a $\d$-polynomial  $R\in K\{Y_1,\ldots,Y_n\}=(K\{Y_1,\dots,Y_{n-1}\})\{Y_n\}$ in several distinct differential indeterminates $Y_1,\dots,Y_n$, we   define degree, weight, homogeneity, etc., in $Y_n$ by viewing~$R$ as a $\d$-polynomial in $Y_n$ over $K\{Y_1,\ldots,Y_{n-1}\}$ (strictly speaking, over its differential fraction field, as we introduced   these notions here only over a differential field).  We also define the \emph{total degree} of $R$ by writing
\[
R\ =\ \sum_{\bm{i}_1,\ldots,\bm{i}_n}R_{\bm{i}_1,\ldots,\bm{i}_n}Y_1^{\bm{i}_1}\cdots Y_n^{\bm{i}_n} \qquad (R_{\bm{i}_1,\ldots,\bm{i}_n}\in K),
\]
and setting 
$$\deg R:=\max\big\{|\bm{i}_1|+\cdots+|\bm{i}_n|:R_{\bm{i}_1,\ldots,\bm{i}_n}\neq0\big\}, \quad\text{with $\max\emptyset:=-\infty$.}$$  \emph{Homogeneous} is defined for $\d$-polynomials in several indeterminates analogously to the case of a single indeterminate, but using $|\bm{i}_1|+\cdots+|\bm{i}_n|$ in place of $|\bm{i}|$.  \emph{Total weight}, \emph{isobaric}, \emph{total super-weight}, and \emph{super-isobaric} are defined similarly.
We also let  $K\langle Y_1,\ldots,Y_n\rangle$ denote
the differential fraction field of $K\{Y_1,\dots,Y_n\}$, so~$K\langle Y_1,\ldots,Y_n\rangle=K\langle Y_1\rangle\cdots\langle Y_n\rangle$.
 
\subsubsection{Dominant degrees and weights}
We extend the valuation $v$ of $K$ first to the gaussian valuation on $K\{Y\}$,  given by $v(P):=\min_{\bm{i}}v(P_{\bm{i}})$, and then uniquely to a valuation on the
differential fraction field $K\langle Y\rangle$ of $K\{Y\}$, also denoted by $v$.  The \emph{dominant degree} of $P$ is~$\ddeg(P):={\max\{d:P_d\asymp P\}}$, and the \emph{dominant multiplicity}~$\dmul(P)$, \emph{dominant weight} $\dwt(P)$, and \emph{dominant weighted multiplicity}~$\dwm(P)$ are defined similarly.  These are equivalent to the definitions in~\cite[Sections~4.5,6.6]{ADAMTT}.

\subsubsection{Composition}
The  differential ring $K\{Y\}$  has a natural composition operation $$(P,Q)\mapsto P(Q)\colon K\{ Y\} \times K\{  Y\}\to K\{ Y\}.$$  This extends to differential rational functions in several indeterminates: temporarily allowing   $Y=(Y_1,\dots,Y_k)$ and $Z=(Z_1,\dots,Z_n)$ to denote tuples of distinct differential indeterminates, there is a natural composition map $$(R,S)\mapsto R(S)\colon K\{   Y \} \times K\langle  Z\rangle^k\to K\langle Z\rangle,$$ which extends to a partial map $K\langle Y\rangle \times K\langle  Z\rangle^k\rightharpoonup K\langle Z\rangle$, with $R(S)=R_1(S)/R_2(S)$ for $R=R_1/R_2$ with $R_1,R_2\in K\{Y\}$ coprime $\d$-poly\-nomials and $S\in K\langle Z\rangle^k$, defined whenever $R_2(S)\neq0$.  This composition is associative when everything involved is defined; in fact:

\begin{lemma}\label{lemma: drational composition}
Let $R\in K\langle Y\rangle$, $S=(S_1,\ldots,S_k)\in K\langle Z\rangle^k$, and $a\in K^n$.    If each~$S_i(a)$  and $R(S(a))$ are defined, then so are $R(S)$ and $R(S)(a)$, and $R(S)(a)=R(S(a))$.
\end{lemma}
\begin{proof}
    By direct calculation, for fixed $a$, the set of $F\in K\langle Z\rangle$ such that $F(a)$ is defined forms a differential subring of $K\langle Z\rangle$, and $F\mapsto F(a)$ is a differential ring morphism from this differential ring to $K$.  Consequently, writing $R=R_1/R_2$ with coprime $R_1,R_2\in K\{Y\}$, we have~$R_1(S)(a)=R_1(S(a))$ and $R_2(S)(a)=R_2(S(a))$.  Thus if $R_2(S(a))\neq0$,~then $R_2(S)\neq0$, $R_2(S)(a)\neq0$, and $R(S)(a)=R(S(a))$.
\end{proof}

\noindent
Reverting back to our earlier convention about $Y$, $Z$,
we next  fix an algebraic closure~$K\langle Y\rangle^{\operatorname{a}}$ of the differential field $K\langle Y\rangle$ and   an embedding $q\mapsto Y^q$ of the additive group $\Q$ into the multiplicative group of $K\langle Y\rangle^{\operatorname{a}}$  with $Y^1=Y$.
For $q\in\Q$ we   have~$(Y^q)'=qY'Y^{q-1}$, and we obtain a $K$-algebra morphism $$P\mapsto P(Y^q)\colon K\{Y\}\to K\langle Y\rangle[Y^q].$$
See \cite[Sections~4.3,~5.7]{ADAMTT} for more about this.  If~$P$ is homogeneous of degree~$d$ and isobaric of weight~$w$, then by \cite[Corollary~4.3.17]{ADAMTT}, $Y^{w-dq}P(Y^q)$ is in $K\{Y\}$ and is homogeneous of degree $w$ and isobaric of weight $w$.  We have~$vP(Y^q)=vP$.

\subsubsection{Additive and multiplicative conjugation}
The \emph{additive} (resp.~\emph{multiplicative}) \emph{conjugate} of $P$ by $a$ is $P_{+a}(Y):=P(Y+a)$ (resp.~$P_{\times a}(Y):=P(aY)$).  If $P$ is homogeneous and $a\ne 0$, then $P_{\times a}$ is homogeneous of the same degree.

\begin{lemma}\label{additiveconjval}
Assume $K$ has small derivation.  
\begin{enumerate}[label=\textup{(\theenumi)}, ref=\theenumi]
\item If $a\dominatedby1$, then $\ddeg P_{+a}=\ddeg P$ and $vP_{+a}=vP$.   
\item If $a\succ1$, then $|vP_{+a}-vP|\le (\deg P)|\alpha|+(\wt P)|\nabla\alpha|$, where $\alpha:=va$.
\end{enumerate}
\end{lemma}
\begin{proof}
    The first part follows from~\cite[Lemma~4.5.1(i) \& Lemma~6.6.5(i)]{ADAMTT}.  For the second part, it suffices to show that $vP_{+a}\geq vP-d|\alpha|-(\wt P)|\nabla\alpha|$; we may then replace $P$, $a$ by $P_{+a}$, $-a$, respectively, to obtain the other inequality.  From \cite[Lem\-ma~4.3.1]{ADAMTT} recall that for any $\bm{i}\in\N^{1+r}$ where $r=\order(P)$  we have~$P_{+a,\bm{i}}=\sum_{\bm{j}\geq\bm{i}}\binom{\bm{j}}{\bm{i}}P_{\bm{j}}a^{\bm{j}-\bm{i}}$.  By Lemma~\ref{nablalemma}(\ref{derivativeval}), we have
    \[
    \begin{split}
    vP_{+a}\ &=\ \min_{\bm{i}}vP_{+a,\bm{i}}\\
    &\geq\ \min_{\bm{i}}\min_{\bm{j}\geq\bm{i}}v\big(P_{\bm{j}}a^{\bm{j}-\bm{i}}\big)\\
    &\geq\ \min_{\substack{\bm{i},\bm{j}\\\bm{i}\leq\bm{j}}}\big[vP_{\bm{j}}+|\bm{j}-\bm{i}|\alpha-\|\bm{j}-\bm{i}\|\cdot|\nabla\alpha|\big]\\
    &\geq\ vP-(\deg P)|\alpha|-(\wt P)|\nabla\alpha|.\qedhere
    \end{split}
    \]
\end{proof}

\subsubsection{The Riccati transform}
As in~\cite[Section 5.8]{ADAMTT} we define the differential polynomials
$$R_0(Z):=1,\qquad R_{n+1}(Z):=ZR_n(Z)+R_n(Z)'$$ in $\Q\{Z\}$; then~$R_n(Y^\dagger)=Y^{(n)}/Y$ (in $K\langle Y\rangle$).  The \emph{Riccati transform} is the $K$-al\-ge\-bra morphism 
$$\Ri\colon K\{Y\}\to K\{Z\}\quad\text{with}\quad Y^{(n)}\mapsto R_n(Z)\text{ for each~$n$.}$$  
Then $\Ri(P)(Y^\dagger)=Y^{-d}P(Y)$ if $P$ is homogeneous of degree $d$.  For $a\ne 0$ we have 
$$\Ri(P_{\times a})=a^d\Ri(P)_{+a^\dagger}.$$
Lemmas~\ref{riccatisuperweight} and~\ref{superweightbounds} below are used for the complexity bounds in Theorem~\ref{equalizertheoremT}.

\begin{lemma}\label{riccatisuperweight}
    For all $n$, $R_n$ is superisobaric of super-weight $n$.  If $P$ is homogeneous and $\wt P=w$, then $\swt \Ri(P)=w$.
\end{lemma}

\noindent
The first statement is easy to show by induction on $n$, and immediately implies the second statement.

\subsubsection{Complexity bounds}
Recursively define  $S_n\in\Q\{Y\}$  by 
$$S_0:=Y',\qquad S_{n+1}:=YS_n'-(n+1)Y'S_n.$$  
Then $(Y^\dagger)^{(n)}=Y^{-n-1}S_n$.  For each $\bm{i}$ we also set $S_{\bm{i}}:=\prod_{k=0}^rS_k^{i_k}$, so $(Y^\dagger)^{\bm{i}}=Y^{-\|\bm{i}\|'}S_{\bm{i}}$.  An easy induction on $n$ shows that $S_n$ is homogeneous of degree $n+1$ and isobaric of weight $n+1$, thus $S_{\bm{i}}$ is homogeneous of degree $\|\bm{i}\|'$ and isobaric of weight~$\|\bm{i}\|'$.  If $P$ is super-isobaric of super-weight $s$, then $Y^sP(Y^\dagger)=\sum_{\bm{i}}P_{\bm{i}}Y^{s-\|\bm{i}\|'}S_{\bm{i}}\in K\{Y\}$ is homogeneous and isobaric of degree and weight $s$, and~$\Ri\big(Y^sP(Y^\dagger)\big)=P$.

\begin{lemma}\label{superweightbounds} 
    Let $R\in K\{Y_1,\dots,Y_n\}$ be a differential polynomial in several variables, $S\in K\{Z\}$ a differential polynomial in one variable.  If $R$ is homogeneous of total degree $d$ and isobaric of total weight $w$, then $R^{\bm{i}}$ is homogeneous of total degree $d|\bm{i}|$ and isobaric of total weight $w|\bm{i}|+\|\bm{i}\|$.  If $S$ is homogeneous and isobaric of degree~$d'$ and weight $w'$, then $S\circ R:=S(R)$ is homogeneous of total degree~$d'd$ and isobaric of total weight $d'w+w'$.  If $R$ has total super-weight $s$ and $S$ has degree $d$ and weight $w$, then $S\circ R$ has super-weight at most $ds+w$.  In particular, $\swt S_{\bm{i}}(R)\leq \|\bm{i}\|'(\swt R+1)$.
\end{lemma}

\subsubsection{The derivations  $\nabla_1$, $\nabla_2$}
These are as defined in \cite[Section~12.7]{ADAMTT}: they are the $K$-derivations on $K\{Y\}$ satisfying $\nabla_1Y^{(n)}=-\binom{n}{2}Y^{(n-1)}$ and $\nabla_2Y^{(n)}=\frac12\binom{n}{3}Y^{(n-2)}$.  The following are immediate:

\begin{lemma}\label{nablawtdecrement}
    If $P$ is isobaric of weight $w$, then $\nabla_1P$ and $\nabla_2P$ are isobaric of weights $w-1$ and $w-2$, respectively
\end{lemma}

\begin{lemma}\label{nablabounded}
    For any $P$, $\nabla_1P\dominatedby P$ and $\nabla_2P\dominatedby P$.
\end{lemma}

\subsubsection{Compositional conjugation}
The operation of compositional conjugation is explored in~\cite[Section 5.7]{ADAMTT}.  For each $\theta$, the \emph{compositional conjugate} $K^\theta$ of~$K$ by $\theta$ is the valued field $K$ equipped with the derivation~$\updelta:=\theta^{-1}\der$.  
Note that~$(K^{\theta_1})^{\theta_2}=K^{\theta_1\theta_2}$. 
The \emph{compositional conjugate} of $P$ by $\theta$ is a differential polynomial~$P^\theta\in K^\theta\{Y\}$ such that~$P^\theta(y)=P(y)$ for any $y\in K$, so that, e.g., $(Y'')^\theta=(\theta\upd)^2Y=\theta' Y'+\theta^2 Y''$.  Formally, we set $F^0_0:=1$, $F^n_k:=0$ for $k\leq 0<n$ or~$k>n$,   recursively define~$F^{n+1}_k:={F^n_k}'+YF^n_{k-1}$, and then take $P\mapsto P^\theta$ to be the $K$-algebra morphism~$K\{Y\}\to K^\theta\{Y\}$ satisfying $$Y^{(n)}\mapsto F^n_n(\theta)Y^{(n)}+F^n_{n-1}(\theta)Y^{(n-1)}+\cdots+F^n_0(\theta)Y.$$  Compositional conjugation of $\d$-polynomials preserves order, degree, and homogeneity, and $(P^{\theta_1})^{\theta_2}=P^{\theta_1\theta_2}$.  We also have 
$$(P^\theta)_{+a}=(P_{+a})^\theta,\quad (P^\theta)_{\times a}=(P_{\times a})^\theta, \quad \Ri(P^\theta)=\Ri(P)^\theta_{\times\theta}.$$
Since $F^n_k\in\Q\{Y\}\subseteq K^{\dominatedby1}\{Y\}$ for all $n$, $k$, by examining monomials we have that~$P^\theta\dominatedby P$ if $\theta\dominatedby1$ and $K$ has small derivation.

Define $\omega\colon K\to K$ by $\omega(z):=-(2z'+z^2)$. This allows us to connect compositional conjugation with~$\nabla_1$ and~$\nabla_2$ above:

\begin{lemma}[{cf.~\cite[Corollary 12.8.6]{ADAMTT}}]\label{magicformulas}
Suppose $P\neq0$, and set $Q:=P^\theta$.  Then with $\lambda:=-\theta^\dagger$, $\omega=\omega(-\theta^\dagger)$, and  $w:=\wt P$:
\[
\begin{aligned}
Q_{[w]}\ &=\ \theta^wP_{[w]},\\
Q_{[w-1]}\ &=\ \theta^{w-1}\big[P_{[w-1]}+\lambda\nabla_1(P_{[w]})\big],\\
Q_{[w-2]}\ &=\ \theta^{w-2}\big[P_{[w-2]}+\lambda\nabla_1(P_{[w-1]})+(\omega\nabla_2+\textstyle\frac12\lambda^2\nabla_1^2)P_{[w]}\big].
\end{aligned}
\]
\end{lemma}

\begin{prop}[{cf.~\cite[Corollary 12.7.18]{ADAMTT}}]\label{nablamagic}
    If $P$ is isobaric of weight $w$, then
    \[
    \nabla_1P=\nabla_2P=0\ \iff\ P\in K[Y](Y')^w,
    \]
    and in this  case $P^\theta=\theta^w P$ for all $\theta$.
\end{prop}

\noindent
The first part of the above proposition was proved in a more general context, which  only required $K$ to be a field and $P\in K[Y,Y',\dots]$ where $Y,Y',\dots$ are viewed as ordinary indeterminates over $K$.  Thus it also holds with $K$ replaced by its residue field~$\K:=K^{\dominatedby}/K^{\prec}$, which we use in the proof of the next lemma.

\begin{lemma}\label{nabladecomposition}
    Suppose $P$ is isobaric of weight $w$.  Set
    \[
    Q\ :=\ \sum_{\bm{i}:i_1=w}P_{\bm{i}}Y^{\bm{i}} \quad\text{and}\quad R\ :=\ \sum_{\bm{i}:i_1\neq w}P_{\bm{i}}Y^{\bm{i}}.
    \]
    Then $Q\in K[Y](Y')^w$ and $vR=\min\{v\nabla_1P,v\nabla_2P\}$.  If $K$ has small derivation, then $v(P^\theta-\theta^wP)\geq\min\{v\nabla_1P,v\nabla_2P\}$ for any $\theta\dominatedby1$.  
\end{lemma}
\begin{proof}
    By Proposition~\ref{nablamagic} applied to $Q\in K\{Y\}$,  we have $\nabla_1Q=\nabla_2Q=0$.  We are done if~$R=0$,  so assume $R\neq0$.  Let $a\in K$, $a\asymp R$, and let $b\mapsto\overline{b}$ be the residue morphism~$K^{\dominatedby}\to \K$, extended to $K^{\dominatedby}\{Y\}\to \K[Y,Y',\dots]$ as usual.  By Proposition~\ref{nablamagic} applied to~$\overline{a^{-1}R}$,   we have $\nabla_1 \overline{a^{-1}R}\neq0$ or $\nabla_2 \overline{a^{-1}R}\neq0$.  It is clear from the definition of $\nabla_i$ that $\nabla_i\overline{a^{-1}R}=\overline{\nabla_ia^{-1}R}=\overline{a^{-1}\nabla_iR}$ for $i=1,2$, so~$\nabla_1R\asymp a\asymp R$ or~$\nabla_2R\asymp a\asymp R$.  
    Since $\nabla_iP=\nabla_i(Q+R)=\nabla_iR$ for $i=1,2$, we have $\nabla_1P\asymp R$ or $\nabla_2P\asymp R$.  Together with Lemma~\ref{nablabounded} applied to $R$, this implies~$vR=\min\{v\nabla_1P,v\nabla_2P\}$.
    The last statement follows from
    \[
    P^\theta-\theta^wP\ =\ Q^\theta-\theta^wQ+R^\theta-\theta^wR\ =\ R^\theta-\theta^w R\ \dominatedby\ R. \qedhere
    \]

\end{proof}

\subsubsection{Properties of valued differential fields}
If $a\prec b\iff a'\prec b'$ for $a,b\prec1$, then $K$ is an \emph{asymptotic field}; if in addition $0\neq a\prec b\prec1\implies a^{\dagger}\dominates b^{\dagger}$, then~$K$ is said to be \emph{of $H$-type}, or \emph{$H$-asymptotic}.  Let~$K$ be an asymptotic field. Then for~$0\neq a\not\asymp1$, the element~$v(a')$ of $\Gamma$ depends only on $va$, and we define the induced maps~$\der,\psi\colon\Gamma^{\neq}\to\Gamma$ by $\der(va)=v(a')$, $\psi(va)=v(a^\dagger)$. Sometimes, the elements $\der(\alpha)$, $\psi(\alpha)$ of $\Gamma$ are also   denoted by $\alpha'$, $\alpha^\dagger$, respectively. The maps $\der$, $\psi$ are extended to maps $\Gamma\to\Gamma_\infty$ by $\der(0):=\psi(0):=\infty$.  We also extend these to~$\Q\Gamma$ by $\psi(q\alpha)=\psi(\alpha)$ and~$(q\alpha)'=q\alpha+\psi(q\alpha)$ for each $\alpha$ and~$q\in\Q^{\neq}$.  The   inverse  of~$\der$ is denoted by~$\int\colon(\Gamma^{\ne})'\to\Gamma^{\neq}$.  We call $K$ \emph{differential-valued}, or \emph{$\d$-valued}, if for any~$a\asymp1$ there is a $c\in C$ with $a\sim c$.
In any asymptotic field, 
\[
\Psi:=\psi(\Gamma^{\neq})<(\Gamma^>)',\qquad (\Gamma^<)'\subseteq\Psi^{\downarrow},\qquad \text{and}\qquad |\Gamma\setminus(\Gamma^{\neq})'|\leq1.
\]
An asymptotic field is \emph{ungrounded} if $\Psi$ has no maximal element; in this case we have $(\Gamma^<)'=\Psi^{\downarrow}$.  If $K$ is an asymptotic field such that  $(\Gamma^{\neq})'=\Gamma$, then $K$ (or $\Gamma$) is said to have \emph{asymptotic integration}.  An ungrounded $H$-asymptotic field $K$   is said to be \emph{$\upl$-free} if for all $a$ there exists a $b\succ1$ such that $a-b^{\dagger\dagger}\dominates b^\dagger$.  It is called \emph{$\upo$-free} if for each~$a$ there exists $b\succ1$ such that $a-\omega(b^{\dagger\dagger})\dominates (b^{\dagger})^2$.  In general, $\upo$-free implies $\upl$-free, which implies asymptotic integration, which implies ungroundedness.  Each of these properties is preserved by compositional conjugation.

\subsubsection{Active elements}
In this subsection $K$ is assumed to be asymptotic.
 We say that $\theta$ is \emph{active} (in $K$) if $a^\dagger\dominatedby\theta$ for some $a\not\asymp1$. We let $\phi$ (potentially with subscripts) range over the active elements of $K$.  A statement which depends on $\phi$ is said to hold \emph{eventually \textup{(}with respect to $\phi$\textup{)}} if there exists an active $\phi_0$ such that it holds for any active~$\phi\dominatedby\phi_0$.
Here are some observations about $P^\phi$ used later:

\begin{lemma}\label{compconjvalbound}
    Suppose $K$ has small derivation and $\phi\dominatedby1$.  Then
    \[
     vP+ (\wm P)\min\{v\phi,\psi(v\phi)\}\ \leq\ vP^\phi\ \leq\ vP+(\dwm P)v\phi.
    \]
\end{lemma}
\begin{proof}
    The second inequality is part of~\cite[Proposition 11.1.4]{ADAMTT}.  For the first inequality, it suffices to handle the case where $P$ is isobaric, say of weight $w$.  If~$w=0$, then $P\in K[Y]$, $P^\phi=P$, and $vP^\phi=vP$, so assume $w>0$.
    The coefficients of $P^\phi$ are given in~\cite[Lemma 5.7.4]{ADAMTT} as linear combinations of the coefficients of $P$, with coefficients given by certain differential polynomials $F^{\bm{\tau}}_{\bm{\sigma}}$ in~$\phi$.  If~$\phi^\dagger\dominatedby\phi$, then the~$v(F^{\bm{\tau}}_{\bm{\sigma}}(\phi))$ are bounded by \cite[Corollary 11.1.3]{ADAMTT} and it follows that $vP^\phi \geq vP+wv\phi$.  If $\phi^\dagger\succ\phi$, then the $v(F^{\bm{\tau}}_{\bm{\sigma}}(\phi))$ are given by \cite[Lemma 11.1.8]{ADAMTT} and it follows that $vP^\phi \geq vP+w\cdot\psi(v\phi)$.
\end{proof}

\begin{lemma}\label{lem:13.1.4}
    Suppose $K$ has small derivation, $\phi\dominatedby1$, and $\dwt P^\phi=\dwm P=w$.  Then $P^{\phi}\sim\phi^wP$ and hence $D_{P^\phi}=D_{P}$.
\end{lemma}
\begin{proof}
    For $\phi\prec1$, this follows from \cite[Corollary 11.1.11(iii)]{ADAMTT}.
    Suppose $\phi\asymp1$.  By~\cite[Lem\-ma~11.1.2(ii)]{ADAMTT} and the definition $F^{\bm{\tau}}_{\bm{\sigma}}=F^{\tau_1}_{\sigma_1}\cdots F^{\tau_n}_{\sigma_n}$, if $\bm{\tau}\geq\bm{\sigma}$ but~${\bm{\tau}\neq\bm{\sigma}}$, then~$F^{\bm{\tau}}_{\bm{\sigma}}(\phi)\prec1$.  Since $F^{\bm{\sigma}}_{\bm{\sigma}}=X^{\|\bm{\sigma}\|}$, the result follows from \cite[Lem\-ma~5.7.4]{ADAMTT}.
\end{proof}

\subsubsection{Asymptotic couples} 
\emph{Assume $K$ is an asymptotic field.}  The pair $(\Gamma,\psi)$ consisting of the
value group $\Gamma$ of $K$ and the map $\psi$ forms the asymptotic couple of $K$ (see~\cite[Sections~6.5, 9.1, \& 9.2]{ADAMTT}).  Small derivation in $K$ is equivalent to $(\Gamma^>)'\subseteq\Gamma^>$.  Since $\Psi^\phi=\Psi-v\phi$, this implies that $K^\phi$ has small derivation for any active $\phi$.  If~$K$ has asymptotic integration, then small derivation is equivalent to $\Psi^{>0}\neq\varnothing$.

\subsubsection{Flattening}
Suppose $K$ is $H$-asymptotic. 
We then have a convex subgroup  
\[
\Gamma^\flat\ :=\ \{\gamma: \gamma^\dagger>0\}
\]
of $\Gamma$,
and we define $v^\flat\colon K^\times\to\Gamma/\Gamma^\flat$ to be the coarsening of $v$ by $\Gamma^\flat$ and $\dominatedby^\flat$ the corresponding dominance relation.  The corresponding notions in a compositional conjugate are then denoted with a subscript: $\Gamma^\flat_\theta=\{\gamma:\gamma^\dagger>v\theta\}$, $v^\flat_\theta$, $\dominatedby^\flat_\theta$.

\begin{lemma}\label{flatderivative}
    If $K$ has small derivation and $\alpha\neq0$, then $\alpha\in\Gamma^{\flat}$ iff $\alpha'\in\Gamma^{\flat}$.
\end{lemma}
\begin{proof}
    If $\alpha'=0\in\Psi^{\downarrow}\subseteq(\Gamma^<)'$, then $\psi(\alpha)=-\alpha>0$, while if $\alpha'\neq0$, then
    \[
    \alpha\in\Gamma^{\flat}\ \iff\ \alpha'>\alpha\ \iff\ \alpha''>\alpha'\ \iff\ \alpha'\in\Gamma^{\flat}. \qedhere
    \]
\end{proof}

\noindent
The following variants of \cite[Lemma~13.1.15]{ADAMTT} will be needed in Section~\ref{sec:evequ}.

\begin{lemma}\label{vPylemma}
Suppose that $P=D\cdot(Y')^w+R$ with $D\in K[Y]$ and $R\in K\{Y\}$, $R\prec^{\flat} P$.  Then
\[
\begin{aligned}
\dmul D=\mul D\text{ and } 1\succ y\asymp^{\flat}1\quad&\implies\quad vP(y)=vP+\dmul(P)vy+w\psi(y),\\
\ddeg D=\deg D\text{ and } 1\prec y\asymp^{\flat}1\quad&\implies\quad vP(y)=vP+\ddeg(P)vy+w\psi(y).
\end{aligned}
\]
\end{lemma}
\begin{proof}
    If $K$ does not have small derivation, then $\Gamma^{\flat}=\{0\}$ and the result is vacuous, so assume that $K$ has small derivation.  Assume $y\asymp^{\flat}1,y\not\asymp1$.  We have by Lemma~\ref{flatderivative} that~$v(y^{(n)})\in\Gamma^{\flat}$ for all $n$, so $R(y)\dominatedby^\flat R\prec^{\flat}P$.  Since $R\prec^{\flat}D$, we have $\dmul(P)=\dmul D+w$, $\ddeg P=\ddeg D+w$, and $\dwt(P)=w$.
    
    Suppose $\dmul D=\mul D=:m$ and $y\prec 1$.  Then we can write $D=Y^{m+1}\cdot E+aY^m$ with~$E\in K[Y]$, $E\dominatedby a\asymp P$.  Then~$E(y)\dominatedby a$ and $D(y)=ay^m+E(y)y^{m+1}\sim ay^m$ since~$y\prec 1$.  From~$a\asymp P$, $y\asymp^{\flat}1$, and $y'\asymp^{\flat}1$, we have~$D(y)\cdot(y')^w\asymp^{\flat}P$, so
    \[
    \begin{split}
    vP(y)\ &=\ vD(y)+w\cdot vy'\\
    &= \ va+m\cdot vy+w(vy+\psi(vy))\\
    &=\ vP+\dmul(P)vy+w\psi(vy).
    \end{split}
    \]

    Now suppose $\ddeg D=\deg D=:d$ and $y\succ 1$.  We can write $D=aY^d+E$ with~$E\in K[Y]$, $\deg E\leq d-1$, and $E\dominatedby a\asymp P$.  Then~$E(y)\dominatedby ay^{d-1}$ and $D(y)=ay^d+E(y)\sim ay^d$.  From~$a\asymp P$, $y\asymp^{\flat}1$, and $y'\asymp^{\flat}1$, we have~$D(y)\cdot(y')^w\asymp^{\flat}P$, so
    \[
    \begin{split}
    vP(y)\ &=\ vD(y)+w\cdot vy'\\
    &= \ va+d\cdot vy+w(vy+\psi(vy))\\
    &=\ vP+\ddeg(P)vy+w\psi(vy). \qedhere
    \end{split}
    \]
\end{proof}

\begin{lemma}\label{vPphilemma}
Suppose that $Q=P(Y')$ and $Q^{\phi_0}=D\cdot(Y')^w+R$ with $D\in K[Y]$ and $R\in K\{Y\}$, $R\prec^{\flat}_{\phi_0} Q^{\phi_0}$.  Then for any $\phi\dominatedby\phi_0$,
\[
vP(\phi)\ =\ vP^{\phi_0}_{\times\phi_0}+wv(\phi/\phi_0).
\]
\end{lemma}
\begin{proof}
Since $Q\in K\{Y'\}$, we have $D\sim^{\flat}_{\phi_0}D_0\in K$, so we may replace $R$ by $R+D-D_0$ and assume $D\in K$ and $R=S(Y')$ for some $S\in K\{Y\}$.  Since $Q^{\phi_0}=P^{\phi_0}_{\times\phi_0}(Y')$, we have $P^{\phi_0}_{\times\phi_0}=DY^w+S$.  Since $\phi\dominatedby\phi_0$ is active, $\phi/\phi_0\asymp^{\flat}_{\phi_0}1$.  As in Lemma~\ref{vPylemma}, we then have

\[
\begin{split}
P(\phi)\ &=\ P^{\phi_0}_{\times\phi_0}(\phi/\phi_0)\\
&=\ D\cdot (\phi/\phi_0)^w+S(\phi/\phi_0)\\
&\sim^{\flat}_{\phi_0}\ D\cdot(\phi/\phi_0)^w\\
&\asymp P^{\phi_0}_{\times\phi_0}\cdot(\phi/\phi_0)^w.\qedhere
\end{split}
\]
\end{proof}

\noindent
Flattening is discussed in~\cite[Sections 9.4 \& 13.1]{ADAMTT}.

\subsubsection{Dominant part}
If $K$ is $\d$-valued, then   we can decompose any nonzero $P$ as~$P=\dd_PD_P+R_P$ with $\dd_P\in K^\times$, $\dd_P\asymp P$, $D_P\in C\{Y\}$, and $R_P\prec P$; by insisting that the lexicographically minimal monomial 
in~$D_P$ has coefficient $1$ in~$D_P$ and that $\dd_P$ is its coefficient in $P$, this decomposition becomes unique.  Note that then~$\ddeg P=\deg D_P$, $\dwt P=\wt D_P$, and $\dwm P=\wm D_P$.  This definition of~$\dd_P$ makes sense even when $K$ is not $\d$-valued, although $D_P$ does not.  For $a\neq0$, we have~$D_{aP}=D_P$, $\dd_{aP}=a\dd_P$, and $R_{aP}=aR_P$, and if $P\sim Q$, then $D_P=D_Q$.

We note that in~\cite[Section 13.1]{ADAMTT}, the dominant part is defined using a set~${\mathfrak{M}\subseteq K^\times}$ of representatives of $\Gamma$ under $v$, where in connection with the Newton polynomial, $\mathfrak{M}$ is additionally taken to be a multiplicative subgroup of $K^\times$.  For our purposes, the properties established above suffice.

\subsubsection{Newton polynomial}  
If $K$ is a $\d$-valued asymptotic field of $H$-type with asymptotic integration and small derivation, then for any $P\ne 0$, there exists a differential polynomial~$N_P\in C\{Y\}^{\ne}$ such that, eventually, $D_{P^\phi}=N_P$.  This is the \emph{Newton polynomial} of~$P$, and its degree and weight are the \emph{Newton degree}~$\ndeg P$ and \emph{Newton weight}~$\nwt P$ of $P$.  Newton weight and degree are defined in~\cite[Section~11.1]{ADAMTT}, while Newton polynomials are defined in~\cite[Section 13.1]{ADAMTT}.  There, the definition of Newton polynomial uses a monomial group, but all that we need is a canonical choice of dominant part such that $D_Q=D_P$ whenever $Q\sim aP$ for some~$a\neq0$.  
Lemma~\ref{lem:13.1.4}, which corresponds to~\cite[Lemma~13.1.4]{ADAMTT}, implies the well-definedness of the Newton polynomial.

\subsubsection{Equalizers}
If $K$ has small derivation, then   $v(P_{\times a})$ depends only on $P$ and~$va$.  If $a\ne 0$ is such that $vP_{\times a}=vQ_{\times a}$, then $va$ is called an \emph{equalizer} for~$P$ and~$Q$.  If $a\ne 0$ is such that $vP^\phi_{\times a}=vQ^\phi_{\times a}$, eventually, then $va$ is called an \emph{eventual equalizer} for $P$ and $Q$.

\subsection{The maps $s$, $d_{\upl}$, and $d_{\upo}$}\label{subsec:asymptotic, s, d}
{\it In this subsection we assume that $K$ is $H$-asymp\-totic.}\/
We introduce three maps taking values in $\Gamma$ (or a certain ordered abelian group extension thereof)
which will play an important role in our work. The first one was already considered in~\cite{GammaTLog}. The maps $d_{\upl}$ and $d_{\upo}$ are new, and quantify
$\upl$-freeness and $\upo$-freeness of $K$, respectively.

\subsubsection{The $s$ map}
 We define $$s\colon(\Gamma^{\neq})'\to\Psi,\qquad s(\alpha')=\psi(\alpha).$$ Thus if~$\beta'=\alpha$, then~$s(\alpha)=\alpha-\beta$.  If $(\Gamma,\psi)$ has asymptotic integration, then the domain of $s$ is~${(\Gamma^{\ne})'=\Gamma}$.  If not, then $\Gamma\setminus(\Gamma^{\neq})'=\{\alpha\}$ for some $\alpha$.  In this case, if~$\alpha\not\in\Psi$, then we set $s(\alpha):=\alpha$, $\Gamma_{\upg}:=\Gamma$, and $\upg:=\alpha$, while if $\alpha\in\Psi$ (so~$\alpha=\max\Psi$), we take an element $\upg$ in an ordered abelian group extension of the divisible hull $\Q\Gamma$ of $\Gamma$ satisfying $\Psi<\upg<(\Q\Gamma^>)'$ and set $\Gamma_{\upg}:=\Gamma+\Z\upg$, $s(\alpha):=\upg$. 

The map $s$ extends naturally to the asymptotic couple $(\Q\Gamma,\psi)$, defined in the same way; by \cite[Corollary~9.2.8]{ADAMTT}, the special case of $\alpha\in\Gamma\setminus(\Gamma^{\neq})'$ does not cause issues.

\begin{lemma}[basic properties of $s$]\label{sbasics} 
\mbox{}
\begin{enumerate}[label=\textup{(\theenumi)}, ref=\theenumi]
\item \label{ssqueezing} $\psi(s(\alpha)-\alpha)=s(\alpha)$ for $\alpha\in(\Gamma^{\neq})'$, and if $\alpha\in(\Gamma^<)'$, then $s(\alpha)>\alpha$.
    \item \label{Psibound} For any $\alpha\in\Psi^{\downarrow}$, $\Psi<s(\alpha)+s(\alpha)-\alpha$, and in particular $\Psi<s(0)+s(0)$.
\item \label{sshape} $s$ is increasing on $(\Gamma^<)'$ and decreasing on $(\Gamma^>)'$.
\item \label{sflatcharacterization} $s(\alpha)>0$ iff $\alpha\in\Gamma^{\flat}$.
\end{enumerate}
\end{lemma}
\begin{proof}
    Suppose $\beta'=\alpha$.  Then $\psi(s(\alpha)-\alpha)=\psi(\psi\beta-\beta')=\psi(\beta)=s(\alpha)$.  The second part of (\ref{ssqueezing}) follows from $s(\alpha)=\alpha-\beta$.
Part
    (\ref{Psibound}) follows from (\ref{ssqueezing}), since~$s(\alpha)+s(\alpha)-\alpha=(s(\alpha)-\alpha)'$,
and    
    (\ref{sshape}) follows from $s(\alpha)=\psi(\beta)$ and the fact that $\psi$ is increasing on $\Gamma^<$ and decreasing on $\Gamma^>$ while the derivative is increasing on~$\Gamma^{\neq}$.
    By definition, $s(\alpha)=\psi(\beta)>0$ iff $\beta\in\Gamma^{\flat}$, so (\ref{sflatcharacterization}) follows from Lemma~\ref{flatderivative}.
\end{proof}

\noindent
Note that if $K$ has small derivation, then by (\ref{ssqueezing}) either $\psi(s0)=s0>0$ or $0\not\in(\Gamma^{\neq})'$, so (\ref{Psibound}) implies that $\Psi^{>0}\subseteq\Gamma^\flat$.
Here is an application of the preceding lemma:

\begin{lemma}\label{flatpreservation}
Suppose that  $K$ is   ungrounded and $P^{\phi_0}\prec^{\flat}_{\phi_0}Q^{\phi_0}$.  Then $P^\phi\prec^{\flat}_{\phi_0}Q^\phi$ and $P^\phi\prec^{\flat}_{\phi}Q^\phi$ for all $\phi\dominatedby\phi_0$.
\end{lemma}
\begin{proof}
    Suppose $\phi\dominatedby\phi_0$.  Since $\phi$ is active, by Lemma~\ref{sbasics}(\ref{Psibound}), $v\phi<2s(v\phi_0)-v\phi_0$.  Thus $\psi(v\phi-v\phi_0)\geq\psi\big(2sv(v\phi_0)-2v\phi_0\big)=s(v\phi_0)$, and so $\phi/\phi_0\asymp^{\flat}_{\phi_0}1$. Now
    by \cite[Lem\-ma~11.1.1]{ADAMTT} with $K^{\phi_0}$, $\phi/\phi_0$, and $v^{\flat}_{\phi_0}$ in place of $K,\phi$, and $v$,
    \[
    P^{\phi}\ \asymp^{\flat}_{\phi_0}P^{\phi_0}\ \prec^{\flat}_{\phi_0}\ Q^{\phi_0}\ \asymp^{\flat}_{\phi_0}\ Q^{\phi}. \qedhere
    \]
\end{proof}

\subsubsection{The maps $d_{\upl}$, $d_{\upo}$}
\emph{Assume that $K$ is of $H$-type.}  For each~$a$ such that~${v(a+\phi^\dagger)}$ is eventually constant, we let $d_{\upl}(a)$ be this eventual value.  Similarly, $d_{\upo}(a)$ is the eventual value of $v\big(a-\omega(-\phi^\dagger)\big)$ when it exists.  Equivalently, if $L$ is an asymptotic field extension of $K$ containing an element $\upl$ such that for any~$b\succ1$ we have~$\upl-b^{\dagger\dagger}\prec b^\dagger$, then $d_{\upl}(a)$ is defined exactly when $a$ is not such a $\upl$, and if $d_{\upl}(a)$ is defined, then $d_{\upl}(a)=v_L(a-\upl)$ for any such $\upl$.  The same holds with $\upo,d_{\upo}$ in place of $\upl,d_{\upl}$ and the condition $\upo-\omega(-b^{\dagger\dagger})\prec (b^{\dagger})^2$ in place of $\upl-b^{\dagger\dagger}\prec b^\dagger$.

If $K$ is $\upl$-free, then~$d_{\upl}$ is defined on all of $K$; and if $K$ is $\upo$-free, then $d_{\upo}$ is defined on all of~$K$.  If $K$ is not $\upl$-free, then with $\upg$ and $\Gamma_{\upg}$ as in the extended definition of~$s$, we extend $d_{\upl}$ to all of $K$ by setting $d_{\upl}(a):=\upg$ when ${v(a+\phi^\dagger)}$ is not eventually constant.  Similarly, if $K$ is not $\upo$-free, then $d_{\upo}(a):=2\upg$ when $v\big(a-\omega(-\phi^\dagger)\big)$ is not eventually constant.

The next lemmas establish basic properties of $d_{\upl}$ and $d_{\upo}$, and relate them to $s$.

\begin{lemma}\label{dbasics}
\mbox{}
    \begin{enumerate}[label=\textup{(\theenumi)}, ref=\theenumi]
        \item\label{dranges} $s(\alpha)\in\Psi\cup\{\upg\}$, $d_{\upl}(a)\in\Psi^{\downarrow}\cup\{\upg\}$, and $d_{\upo}(a)\in (2\Psi)^{\downarrow}\cup\{2\upg\}$.
        \item\label{dvaluation} $v(a-b)\geq\min\{d_{\upl}(a),d_{\upl}(b)\}$ and $d_{\upl}(a)\geq\min\{d_{\upl}(b),v(a-b)\}$.  If either of these minima is unique, then the corresponding inequality is an equality.  The same is true with $d_{\upo}$ in place of $d_{\upl}$.
        \item\label{dinvariance} If $v(a-b)>\Psi$, or $a-b\prec^{\flat}1$, or $a\sim^\flat b$, then $d_{\upl}(a)=d_{\upl}(b)$.  If $v(a-b)>2\Psi$, or $a-b\prec^{\flat}1$, or $a\sim^\flat b$, then $d_{\upo}(a)=d_{\upo}(b)$.
        \item\label{dconjugation} 
        With  $s^\theta$, $d_{\upl}^\theta$,  $d_{\upo}^\theta$   the maps $s$, $d_{\upl}$,  $d_{\upo}$   defined in~$K^\theta$, and $\upg^\theta=\upg-v\theta$:
        \begin{align*} s^\theta(\alpha)&=s(\alpha+v\theta)-v\theta,\\ d_{\upl}^\theta(a)&=d_{\upl}(\theta a-\theta^\dagger)-v\theta,\\ d_{\upo}^\theta(a)&=d_{\upo}({\theta^2a+\omega(-\theta^\dagger)})-2v\theta.\end{align*}
    \end{enumerate}
\end{lemma}
\begin{proof}
    Part~(\ref{dranges}) follows from~\cite[Lemma 11.5.2 \& Corollary 11.7.2]{ADAMTT} and the definition of $s$.  Part~(\ref{dvaluation}) is immediate from the definitions of $d_{\upl}$ and $d_{\upo}$ and   $v$ being a valuation.

    If $a\succ1$, then $va<d_{\upl}(0)=d_{\upo}(0)$, so $d_{\upl}(a)=d_{\upo}(a)=va$ by (\ref{dvaluation}).  This implies the case $a\sim^\flat b\succ^\flat1$  of (\ref{dinvariance}); the other cases are immediate from (\ref{dvaluation}) and (\ref{dranges}) plus~$\Psi^{>0}\subseteq\Gamma^{\flat}$.

    Using the formulas from the end of Section~11.7 in \cite{ADAMTT}, we see that $d_{\upl}^\theta(a)$ is the eventual (with respect to $\phi$) valuation of
     \[
    a+\theta^{-1}(\phi^\dagger-\theta^\dagger)\ =\ \theta^{-1}[(\theta a-\theta^\dagger)+\phi^{\dagger}],
    \]
    and $d_{\upo}^\theta(a)$ is the eventual valuation of
    \[
    a-\omega^\theta(-\phi^\dagger/\theta+\theta^\dagger/\theta)\ =\ \theta^{-2}[(\theta^2a+\omega(-\theta^\dagger))-\omega(-\phi^\dagger)].
    \]
    Applying $v$ yields the $d_{\upl}$ and $d_{\upo}$ parts of (\ref{dconjugation}); the $s$ part follows easily from the definition of $s$. 
\end{proof}

\begin{lemma}\label{drelations}
\mbox{}
\begin{enumerate}[label=\textup{(\theenumi)}, ref=\theenumi]
    \item\label{sfromdl} For any $n\geq1$, $d_{\upl}\big({-\frac1na^\dagger}\big)=s\big(\frac1nva\big)$, where $s(\frac1nva)$ is defined in $\Q\Gamma$.
    \item\label{dlfromdo} $d_{\upo}(\omega(a))=2d_{\upl}(a)$.
    \item\label{dlrepresentatives} If $s(v\theta)>v(a+\theta^\dagger)$ or $s(v\theta)>d_{\upl}(a)$, then $d_{\upl}(a)=v(a+\theta^\dagger)$.
    \item\label{dorepresentatives} If $s(v\theta)>\frac12v(a-\omega(-\theta^\dagger))$ or $s(v\theta)>\frac12d_{\upo}(a)$, then $d_{\upo}(a)=v(a-\omega(-\theta^\dagger))$.
    \item\label{ddeffromA} Suppose $\theta$ is active.  Then
    \begin{enumerate}[label=\textup{(\theenumii)}, ref=\theenumii]
        \item  $v\theta\geq d_{\upl}(a)$ iff $\theta\dominatedby a+\theta^\dagger$;
        \item $v\theta=d_{\upl}(a)$ iff $\theta\asymp a+\theta^\dagger$;
        \item $v\theta\geq\frac12 d_{\upo}(a)$ iff $\theta^2\dominatedby a-\omega(-\theta^\dagger)$; and
        \item $v\theta=\frac12 d_{\upo}(a)$ iff $\theta^2\asymp a-\omega(-\theta^\dagger)$.
    \end{enumerate}
    \item\label{ddeffromA2} Suppose $v\theta\in(2\Psi)^\downarrow$.  Then
    \begin{enumerate}[label=\textup{(\theenumii)}, ref=\theenumii]
        \item $v\theta\geq d_{\upo}(a)$ iff $\theta\dominatedby a-\omega(-\frac12\theta^\dagger)$ and
        \item $v\theta=d_{\upo}(a)$ iff $\theta\asymp a-\omega(-\frac12\theta^\dagger)$.
    \end{enumerate}
\end{enumerate}
\end{lemma}
\begin{proof}
We first prove (\ref{sfromdl}) in the case $n=1$; the proof for $n\ge 2$ will make use of (\ref{dlrepresentatives}).  We are interested in the eventual value of $v(-a^\dagger+\phi^\dagger)=\psi(v\phi-va)$.  Let $\alpha=va$ and suppose $\beta$ is such that $\beta'=\alpha$.  If $\beta<0$, then $\alpha-\beta=\psi(\beta)\in\Psi$ and $\alpha-2\beta=(-\beta)'\in(\Gamma^>)'$, so eventually $\alpha-\beta<v\phi<\alpha-2\beta$ and $\psi(v\phi-\alpha)=\psi(\beta)=s(\alpha)$.

If $\beta>0$, let $\gamma<0$ be such that $\gamma'=\psi(\beta)$.  Then $\psi(\gamma)=\psi(\beta)-\gamma>\psi(\beta)$.  We have $\psi(\beta)=\alpha-\beta\in\Psi$ and $\alpha-\beta-2\gamma=(-\gamma)'\in(\Gamma^>)'$, so eventually~$\alpha-\beta<v\phi<\alpha-\beta-2\gamma$ and $\psi(v\phi-\alpha)=\psi(\beta)=s(\alpha)$.

If no such $\beta$ exists, then $(\Gamma^<)'<\alpha<(\Gamma^>)'$.  Then for any $\phi$, $s(v\phi)<\alpha<s(v\phi)+s(v\phi)-v\phi$, so $v(-a^\dagger+\phi^\dagger)=\psi(\alpha-v\phi)=\psi(s(v\phi)-v\phi)=s(v\phi)$.  This is not eventually constant, so $d_{\upl}(-a^\dagger)=s(\alpha)=\upg$.

To obtain (\ref{dlrepresentatives})
take $b:=-\theta^\dagger$ in Lemma~\ref{dbasics}(\ref{dvaluation}) and use $d_{\upl}(b)=s(\theta)$.  Part~(\ref{dorepresentatives}) will follow in the same way once we have shown (\ref{dlfromdo}), which we do next. For this, suppose first that
 $d_{\upl}(a)\neq\upg$ and $v\phi\geq s(d_{\upl}(a))$, hence $s(v\phi)>s(d_{\upl}(a))>d_{\upl}(a)$, and let~$\theta:=a+\phi^\dagger$.  By~\cite[Lemma 5.2.1]{ADAMTT}, $$\omega(-\phi^\dagger)-\omega(a)=\theta(2(\theta^\dagger-\phi^\dagger)+\theta).$$ By (\ref{dlrepresentatives}), we have $v\theta=d_{\upl}(a)$, and since $v\phi\geq s(v\theta)$,   also  $v(\phi^\dagger-\theta^\dagger)=d_{\upl}(-\theta^\dagger)=s(v\theta)$.  Thus $\theta^\dagger-\phi^\dagger\prec\theta$ and $\omega(a)-\omega(-\phi^\dagger)\sim\theta^2$.  Applying $v$ yields the case~$d_{\upl}(a)\neq\upg$   of~(\ref{dlfromdo}).
If $d_{\upl}(a)=\upg$, then for any active $\phi$, we have $v(a+\phi^\dagger)=d_{\upl}(-\phi^\dagger)=s(v\phi)$ by Lemma~\ref{dbasics}(\ref{dvaluation}).  Set $\theta:=a+\phi^\dagger$, so   $v(\theta^\dagger-\phi^\dagger)=d_{\upl}(-\phi^\dagger)=s(v\phi)$ by~(\ref{dlrepresentatives}), and~$v(a+\theta^\dagger)=s(v\theta)>s(v\phi)$.  Then we obtain 
\[\omega(-\phi^\dagger)-\omega(a)\ =\ \theta(2(\theta^\dagger-\phi^\dagger)+\theta)\ =\ \theta((\theta^\dagger-\phi^\dagger)+(a+\theta^\dagger))\ \asymp\ \theta^2,\]
and (\ref{dlfromdo}) follows.

We now return to the case $n\ge 2$ of (\ref{sfromdl}).  Suppose that $d_{\upl}(-\frac1na^\dagger)\neq\upg$, and let $\theta$ be such that $v\theta=d_{\upl}({-\frac1na^\dagger})$.  Then $s(v\theta)>v\theta$, so by~(\ref{dlrepresentatives}),
\[
d_{\upl}\big({-\textstyle\frac1na^\dagger}\big)\ =\ v\big(\theta^\dagger-\textstyle\frac1na^\dagger\big)\ =\ v(n\theta^\dagger-a^\dagger)\ =\ \psi(nv\theta-\alpha).
\]
Thus in $\Q\Gamma$, $\psi(v\theta-\frac1n\alpha)=v\theta$ and $(\frac1n\alpha-v\theta)'=\frac1n\alpha$.

Now suppose that $d_{\upl}(-\frac1na^\dagger)=\upg$, $v\theta\neq n\upg$, and  $v\phi\geq d_{\upl}(-\frac1n\theta^\dagger)$.  Then by Lemma~\ref{dbasics}(\ref{dvaluation}), $\psi(nv\phi-\alpha)=v(-\frac1na^\dagger+\phi^\dagger)=s(v\phi)$ and  $\psi(nv\phi-v\theta)=d_{\upl}(-\frac1n\theta^\dagger)\leq v\phi$.  Thus $\alpha\neq v\theta$, so $\alpha=n\upg$ and $s(\frac1n\alpha)=\upg=d_{\upl}(-\frac1na^\dagger)$.

It remains to show (\ref{ddeffromA}) and (\ref{ddeffromA2}). For this note that $s(v\theta)>v\theta$
for active $\theta$.  Thus if any one of the conditions in (\ref{ddeffromA}) holds, we may apply either (\ref{dlrepresentatives}) or (\ref{dorepresentatives}) to show that the corresponding condition also holds.  For~(\ref{ddeffromA2}), we need a slight modification of~(\ref{dorepresentatives}), obtained by taking $b:=\omega(-\frac12\theta^\dagger)$ in Lemma~\ref{dbasics}(\ref{dvaluation}).
\end{proof}

\begin{remark}
If $K$ is a Liouville closed $\upo$-free $H$-field, then by \cite[Corollaries~11.8.13 and~11.8.16]{ADAMTT}, $a\in\Lambda(K)$ iff $a<\upl$ in some/any immediate extension of $K$ that contains $\upl$.  Since $a<\upl$ iff $a<b$ for any $b$ such that $b-\upl\prec a-\upl$, we have~$a\in\Lambda(K)\iff a<-\phi^\dagger$ whenever $v\phi\geq d_{\upl}(a)$, by Lemma~\ref{drelations}(\ref{sfromdl}).  Similarly, Lemma~\ref{drelations}(\ref{dlfromdo}) implies that $a<\upo$ in some/any immediate extension containing~$\upo$ iff $a<\omega(-\phi^\dagger)$ whenever $v\phi\geq\frac12d_{\upo}(a)$, or equivalently $a<\omega(-\frac12\theta^\dagger)$ whenever~$\theta\in(2\Psi)^{\downarrow}$ and $v\theta\geq d_{\upo}(a)$.
\end{remark}

\subsection{Definable functions involving the value group}\label{subsec:definable on value group}
Unless otherwise specified, all definability will be \emph{uniform}, i.e., with definition independent of the structure.  For purposes of definability, we will identify elements of a given finite-dimensional subspace of $K\{Y\}$ with the tuples of their coefficients in some arbitrary but fixed way. 

The one-sorted language of valued differential rings is $\L_{\textrm{r}}:=\{0,1,{+},{-},{\,\cdot\,},\der,{\dominatedby}\}$; the one-sorted language of valued differential fields is $\L_{\textrm{f}}:=\L_{\textrm{r}}\cup\{\iota\}$, where~$\iota$ is a unary function symbol to represent the multiplicative inverse, interpreted by~${\iota(a)=a^{-1}}$ if~$a\neq0$, $\iota(0)=0$; and the two-sorted language of asymptotic fields (with an additional sort for the value group) is $\L_{\textrm{a}}:=\L_{\textrm{f}}\cup\{v,0_\Gamma,{+_\Gamma},{-_\Gamma},\psi\}$. Valued differential fields are viewed as   $\L_{\textrm{f}}$-structures in      the natural way; likewise with ``asymptotic'' and
``$\L_{\textrm{a}}$'' in place of ``valued differential'', `` $\L_{\textrm{f}}$'', respectively.
 Unless otherwise specified, ``definable'' will mean $\L_{\textrm{f}}$-definable, without parameters.  With this convention, we call a relation~$R\subseteq K^m\times \Gamma^n$ \emph{quantifier-free definable} in~$K$ if its preimage~$(\id_K^m\times v^n)^{-1}(R)\subseteq K^m\times (K^\times)^n$ is quantifier-free definable in $K$.

 We note that any quantifier-free $\L_{\textrm{f}}$-formula is equivalent to a quantifier-free $\L_{\textrm{r}}$-formula under the axioms of valued differential fields: if two valued differential fields have differential subrings which are isomorphic as valued differential rings, then this isomorphism extends to their (valued differential) fraction fields by the universal property thereof, so by~\cite[Corollary B.11.5]{ADAMTT} every quantifier-free $\L_{\textrm{f}}$-formula is equivalent to a quantifier-free $\L_{\textrm{r}}$-formula. (See  \cite[Corollary~16.5.2]{ADAMTT} for the analogous
 statement for valued ordered differential fields.)

While we discuss compositional and multiplicative conjugation by nonzero elements of~$K$, properties of $P^\phi_{\times a}$ are often more naturally associated with $v\phi$ and~$va$ than with the specific elements $\phi$ and $a$.  For example, there is a natural map which sends a pair $(P,Q)$ of homogeneous differential polynomials of different degrees to their eventual equalizer.  However, we will also need to have access to representatives $\phi$ and $a$, e.g., when composing such maps.

We will mostly deal with quantifier-free definable partial maps.  Each of the following are quantifier-free definable in $\L_\textrm{f}$, in an $\upo$-free $H$-asymptotic field $K$ satisfying~$2\Gamma=\Gamma$:
\begin{itemize}
    \item Any relation which is quantifier-free definable in~$\L_{\textrm{a}}$.
    \item In particular, the maps $\der_\Gamma\colon\Gamma^{\neq}\to\Gamma$, $\int_\Gamma\colon\Gamma\to\Gamma^{\neq}$, and $s\colon\Gamma\to\Gamma$.
    \item The maps $a\mapsto\int d_{\upl}(a)$ and $a\mapsto\int \frac12d_{\upo}(a)$, by Lemma~\ref{drelations}(\ref{ddeffromA}) and the fact that $\theta$ is active iff $\int v\theta<0$.
    \item For any $f$ and injective $g$ given by (tuples of) $\L_{\textrm{a}}$-terms and any quantifier-free definable map~$F$, the map $g^{-1}\circ F\circ f$; in particular, the map ${\int}\circ F\circ\der$.
    \item Given quantifier-free definable partial maps $F\colon K^{m_1}\times\Gamma^{n_1}\rightharpoonup K^{m_2}\times \Gamma^{n_2}$ and $G\colon K^{m_2}\times\Gamma^{n_2}\rightharpoonup K^{m_3}\times \Gamma^{n_3}$,
    the partial map  $(F,G\circ F)$. 
\end{itemize}
We need $(F,G\circ F)$ rather than just $G\circ F$ in the last item because a composition of quantifier-free definable partial maps may not be quantifier-free definable.  For example, the map $a\mapsto\int va$ and the sign map $\Gamma\to\{-1,0,1\}$ are quantifier-free definable in an asymptotic field with asymptotic integration, but by~\cite[Proposition~16.5.1]{ADAMTT}, their composition is not quantifier-free definable in $\T$.

If $(f_1,\ldots,f_{m_2+n_2})\colon K^{m_1}\times \Gamma^{n_1}\rightharpoonup K^{m_2}\times\Gamma^{n_2}$ is a quantifier-free definable partial map, then each of the $f_i$ is definable by both a universal formula and an existential formula.

Since we work with compositional conjugates, we will need to ensure that the values of certain maps involving $P^\theta$ depend only on $v\theta$, and not the choice of representative~$\theta$. Towards this, we define:

\begin{definition}\label{conjinvariantdef}
A partial map $\eta\colon K\{Y\}\times K^\times\rightharpoonup\Gamma$  is  \emph{conjugation-invariant} if for any $(P,\theta_1)\in\dom\eta$ and  $\theta_2\asymp\theta_1$ we have $(P,\theta_2)\in\dom\eta$ and $\eta(P,\theta_2)=\eta(P,\theta_1)$.
\end{definition}

\noindent
We think of such an $\eta$ as operating on $P^\theta\in K^\theta\{Y\}$.  As such, the composition of two conjugation-invariant functions $\eta_1,\eta_2$ is defined by
\[
(\eta_2\circ\eta_1)(P,\theta)\ :=\ \eta_2(P,\theta_1)\qquad\text{for any $\theta_1$ such that $v\theta_1=\eta_1(P,\theta)$}.
\]
Here, since $\eta_2$ is conjugation-invariant, $\eta_2(P,\theta_1)$ is independent of the choice of $\theta_1$.  The following facts are straightforward:

\begin{lemma}\label{conjinvbasics}
    Let $\eta_1,\eta_2\colon K\{Y\}\times K^\times\rightharpoonup\Gamma$ be conjugation-invariant.  Then
    \begin{enumerate}[label=\textup{(\theenumi)}, ref=\theenumi]
        \item $\eta_2\circ\eta_1$ is conjugation-invariant.
        \item If $A\subseteq K\{Y\}$, $B\subseteq K^\times$, $A\times B\subseteq \dom\eta_1$, and $A\times v^{-1}(\operatorname{range}\eta_1)\subseteq\dom\eta_2$, then $A\times B\subseteq\dom(\eta_2\circ\eta_1)$.
        \item If $\{\eta_i\}_{i\in I}$ is a collection of conjugation-invariant partial functions with disjoint domains, then $\bigcup_{i\in I}\eta_i$ is conjugation-invariant.
    \end{enumerate}
\end{lemma}

\noindent
Let now $T$ be a theory extending the $\L_{\textrm{a}}$-theory of asymptotic fields with small derivation, such that if $K\models T$ and $\phi$ is active then $K^\phi\models T$.
Suppose that a partial function $\eta\colon(K\{Y\}_{\leq d,[\leq w]})^m\rightharpoonup\Gamma$ is defined by an $\L_{\textrm{a}}$-formula $\delta$ (without parameters), and that $T\models\text{``}\eta\text{ is a partial function''}$.  For active $\phi$ we can then define~$\eta_\phi\colon (K^\phi\{Y\}_{\leq d,[\leq w]})^m\rightharpoonup\Gamma^\phi$ by interpreting $\delta$ in the $\L_{\textrm{a}}$-structure $K^\phi$.  Equivalently, recalling that the underlying sets of $K^\phi$ and its asymptotic couple are simply~$K$,~$\Gamma$, respectively, treating $\phi$ as a parameter we may obtain $\eta_\phi$ as the partial map defined by  the formula obtained from $\delta$ by replacing each occurrence of the function symbols~$\der$, $\psi$   by $\phi^{-1}\der$, $\psi-v\phi$, respectively.  We now have a natural extension of~$\eta$ to a partial map~$(K\{Y\}_{\leq d,[\leq w]})^m\times\mathcal{A}\rightharpoonup\Gamma$ defined as follows,
with~$\mathcal{A}:=v^{-1}\Psi^{\downarrow}$ denoting the set of active elements of $K$:
\[
\widehat{\eta}(\vec{P},\phi)\ :=\ \eta_\phi(\vec{P}^\phi)+v\phi,
\]
thus $\widehat{\eta}(\vec{P},1)=\eta(\vec{P})$. Here $\vec{P}^\phi:=(P_1^\phi,\dots,P_m^\phi)$ for $\vec{P}=(P_1,\dots,P_m)$.
This gives us a more convenient way to obtain definable conjugation-invariant functions:

\begin{lemma}\label{definableconjinvariantlemma}
    Let $\eta$ be as above.  
    \begin{enumerate}[label=\textup{(\theenumi)}, ref=\theenumi]
        \item\label{conjinvariantextension} Suppose
        \[
        T\ \models\ (\vec{P}\in\dom\eta\wedge\theta\asymp1)\to\big[\vec{P}^\theta\in\dom\eta_\theta\wedge \eta_\theta(\vec{P}^\theta)=\eta(\vec{P})\big].
        \]
        Then $\widehat{\eta}(\vec{P},\phi)$ depends only on $\vec{P}$, $v\phi$, and thus induces a partial function~$(K\{Y\}_{\leq d,[\leq w]})^m\times\Psi^{\downarrow}\rightharpoonup\Gamma$.
        \item\label{conjinvariantfromcomp} Suppose that $T\models \text{``$\eta=h\circ f$''}$ where $h$ is one of the partial maps $\psi\circ v$, $s\circ v$, $d_{\upl}$, or $\frac12d_{\upo}$, and  $f$ is a  differential rational function over $\Q$.  Then there is a differential rational function $\widehat{f}$ over $\Q$ with $T\models\widehat{\eta}=h\circ\widehat{f}$.
    \end{enumerate}
\end{lemma}
\begin{proof}
    If $\eta$ is as in (\ref{conjinvariantextension}) and $\phi_1\asymp\phi_2$ are active, then $K^{\phi_1}\models T$ and~$\phi_2/\phi_1\asymp1$ yield
    \[
    \begin{split}
        \widehat{\eta}(\vec{P},\phi_1)\ &=\ \eta_{\phi_1}(\vec{P}^{\phi_1})+v\phi_1\\
        &=\ (\eta_{\phi_1})_{\phi_2/\phi_1}\big((\vec{P}^{\phi_1})^{\phi_2/\phi_1}\big)+v\phi_1\\
        &=\ \eta_{\phi_2}(\vec{P}^{\phi_2})+v\phi_2\\
        &=\ \widehat{\eta}(\vec{P},\phi_2).
    \end{split}
    \]
    For (\ref{conjinvariantfromcomp}), replacing each $\der$ in the definition of $f$ by $\phi^{-1}\der$ clearly yields a differential rational function $f_\phi$ over $\Q$ in $\phi$ and the arguments of $f$.  From $\psi_\phi=\psi-v\phi$ and Lemma~\ref{dbasics}(\ref{dconjugation}), we see that~$h_\phi=(h\circ g)-v\phi$ for a differential rational function $g$ over $\Q$, and hence
    \[
    \widehat{\eta}(\vec{P},\phi)\ =\ h_\phi(f_\phi(\vec{P}^\phi))+v\phi\ =\ h(g(f_\phi(\vec{P}^\phi)))\ =:\ h(\widehat{f}(\vec{P},\phi)),
    \]
    where $\widehat{f}(\vec{P},\phi)=g(f_\phi(\vec{P}^\phi))$ is again a differential rational function over $\Q$.
\end{proof}

\section{Equalizers}\label{equalizersection}

\noindent
\emph{In this section, $K$ is a valued differential field with small derivation}. 
The Equalizer Theorem is stated in \cite{ADAMTT} as follows:

\begin{thmunnumbered}[Equalizer Theorem, ADH version]
    Suppose $P,Q\ne 0$ are homogeneous of degree $d$ and $e$, respectively, with $d>e$, and   $(d-e)\Gamma=\Gamma$.  Then~$P$ and~$Q$ have a \textup{(}necessarily unique\textup{)} equalizer in $\Gamma$.
\end{thmunnumbered}

\noindent
We first give a short argument which shows that an equalizer of $P$, $Q$ as in this theorem can be taken to be one
of  one of finitely many $\L_{\textrm{f}}$-terms in the coefficients of the $\d$-polynomials~$P$,~$Q$; but this argument does not yield an explicit bound on the complexity of these terms.  We use Herbrand's Theorem (see \cite[Proposition~B.8.6]{ADAMTT}): if~$\mathcal L$ is a language, $\Sigma$ is a set of universal $\mathcal L$-sentences, and~$\varphi(x,y)$ a quantifier-free $\mathcal L$-formula, where $x$, $y$ are finite disjoint multivariables, such that $\Sigma\models\forall x\exists y\ \varphi(x,y)$, then there are finitely many tuples $t_1(x),\ldots,t_n(x)$ of $\mathcal L$-terms such that
\[
\Sigma\models\forall x\bigvee_{i=1}^n\varphi(x,t_i(x)).
\]
To apply this model theoretic result, we need our assumptions on $K$ to be universally axiomatizable, and we need the conclusion to be expressible by a $\forall\exists$-formula.
Here, the condition ``$(d-e)\Gamma=\Gamma$'' presents a minor obstacle for a universal axiomatization, which we circumvent via the following observation:
with $m:=d-e$ and $n:=\max\{\wt P,\wt Q\}$, we have
$$\widetilde{P}:=Y^{n+1-d/m}P(Y^{1/m}), \widetilde{Q}:=Y^{n-e/m}Q(Y^{1/m})\in K\{Y\},$$ and 
these $\d$-polynomials are homogeneous of degree $n+1$ and $n$, respectively.  It is also easy to verify that for each $c$  in a valued differential field extension of~$K$,  $$v\widetilde{P}_{\times c^m}-v\widetilde{Q}_{\times c^m}=vP_{\times c}-vQ_{\times c}.$$
Thus, applying the Equalizer Theorem to $\widetilde{P}$ and $\widetilde{Q}$ does not need any assumption on the divisibility of $\Gamma$, and yields $a\in K^{\times}$ such that if $P_{\times c}\asymp Q_{\times c}$ for some $c$ in an extension of $K$, then $c^{d-e}\asymp a$.  Since $\widetilde{P}_{\times a}\asymp \widetilde{Q}_{\times a}$ is described by a quantifier-free $\L_{\textrm{f}}$-formula, the conditions of Herbrand's Theorem are fulfilled.  Our $a$ can therefore be taken to be one of finitely many $\L_{\textrm{f}}$-terms (i.e., $\d$-rational functions over $\Q$) in the coefficients of~$P$,~$Q$.

The goal of this section is the following improvement upon this mere existence of such a uniform
parametrization of equalizers:

\begin{thm}[Equalizer Theorem]\label{equalizerthm}
    Fix $d,e,w\in\N$ with $d\neq e$.  Then there are $\d$-rational functions $H_1,\ldots,H_N$ over $\Q$, where $N\leq2^{4w^2+2w}$,  such that if~$P,Q\ne 0$ are homogeneous with $\deg P=d$, $\deg Q=e$, and $\wt P,\wt Q\le w$, then there is an~$i\in\{1,\ldots,N\}$ such that for any nonzero element $c$ of a valued differential field extension of $K$ with small derivation,  we have
    $$P_{\times c}\asymp Q_{\times c}\quad\Longleftrightarrow\quad c^{d-e}\asymp H_i(P,Q).$$  Here the  $H_i$ can be chosen with numerator and denominator of  super-weight at most~$(w+2)\cdot(2w+1)!$.
\end{thm}

\noindent
We first give a useful reformulation of this theorem. 
If~$a\ne 0$, $P,Q\ne 0$ are homogeneous with $\deg P=d\ne e=\deg Q$, and $R:=\Ri(P),S:=\Ri(Q)\in K\{Z\}$, then~$P_{\times a}\asymp Q_{\times a}$ is equivalent to
\[
vS_{+k^{-1}b^\dagger}-vR_{+k^{-1}b^\dagger}\ =\ vb \qquad\text{where } k:=d-e\in\Z^{\ne} \text{ and } b:=a^k.
\]
Since $\swt R=\wt P$, the Equalizer Theorem will therefore follow from the following fact applied to $S$, $R$, $Z$  in place of $P$, $Q$, $Y$:

\begin{thm}[Equalizer Theorem, fixed point version]\label{equalizerthmriccati}
    Fix $s\in\N$ and $p,q\in\Q^\times$.  There exist $N\leq2^{4s^2+2s}$ and $\d$-rational functions $H_1,\ldots,H_N$ over $\Q$ such that for all $P,Q\ne 0$ of super-weight~$\le s$ there is an $i\in\{1,\ldots,N\}$ such that for each~$c\ne 0$ in a valued differential field extension of $K$ with small derivation, 
    \[
    vP_{+pc^\dagger}-vQ_{+qc^\dagger}\ =\ vc\quad\iff\quad c\asymp H_i(P,Q).
    \]
    The numerator and denominator of each $H_i$ can be chosen to have super-weight at most $(s+2)\cdot(2s+1)!$.
\end{thm}

\noindent
Thus finding an equalizer is reduced to finding a fixed point of the operator
$$f\colon \Gamma\to\Gamma,\qquad va\mapsto vP_{+pa^\dagger}-vQ_{+qa^\dagger}.$$
(In   Lemma~\ref{fproperties} below we see that this definition makes sense.) We have~$P_{+a,\bm{i}}=\sum_{\bm{j}\geq\bm{i}}\binom{\bm{j}}{\bm{i}}P_{\bm{j}}a^{\bm{j}-\bm{i}}$, and Lemma~\ref{nablalemma}(\ref{derivativeval}) gives an estimate of $v(a^{\bm{j}-\bm{i}})$, so if we could guarantee that a particular one of these terms was dominant, then we could attempt to find a fixed point of $f$ by controlling that term.  This is essentially the approach taken in~\cite{ADAMTT}: they approximate $vP_{+a^\dagger}$ by $vP$, effectively assuming that a dominant term in $P_{+a^\dagger}$ was already dominant in $P$; if this fails, then hopefully the new dominant terms will be of lower degree, since the higher-degree terms were already smaller before the conjugation.

This requires a transfinite induction to succeed: although the higher-degree terms of $P$ are smaller, their contribution to $P_{+a}$ might still be larger.  We solve this by, instead of focusing on the dominant terms of $P$,  looking at terms which are so ``close   to dominant''   that they contribute significantly to $P_{+a}$.  Then we modify $f$ to use these terms instead of the dominant ones. This strategy is made precise below.

We begin by establishing the properties of $f$ that we will need:

\begin{lemma}\label{fproperties}
\mbox{}

\begin{enumerate}[label=\textup{(\theenumi)}, ref=\theenumi]
\item\label{fwelldef} $f(\alpha)$ is well-defined, i.e., does not depend on the choice of $a$ with $va=\alpha$;
\item\label{fslow} if $\alpha\ne\beta$, then $|f(\alpha)-f(\beta)|\le(\deg P+\deg Q+1)\big|\nabla(\alpha-\beta)\big|$;
\item\label{fnear} if $f(\beta)=\beta$, then $\big|\beta-f(\alpha)\big|\le(\deg P+\deg Q+1)\big|\nabla(f(\alpha)-\alpha)\big|$.
\end{enumerate}
\end{lemma}
\begin{proof}
If $u\in K$, $u\asymp1$, then $u^{\dagger}\dominatedby1$, so $$vP_{+p(au)^{\dagger}}=v\big((P_{+pa^{\dagger}})_{+pu^{\dagger}}\big)=vP_{+pa^{\dagger}}$$ by Lem\-ma~\ref{additiveconjval}.  This shows (\ref{fwelldef}).  Given $va=\alpha\ne\beta=vb$, we have by Lemma~\ref{additiveconjval} and Lemma~\ref{nablalemma}(\ref{nablashrink}) that if $v(a^{\dagger}-b^{\dagger})<0$, then
\[
\begin{split}
|f(\alpha)-f(\beta)|\ &\le\ \big|vP_{+pb^{\dagger},+pa^{\dagger}-pb^{\dagger}}-vP_{+pb^{\dagger}}\big|+\big|vQ_{+qb^{\dagger},+qa^{\dagger}-qb^{\dagger}}-vQ_{+qb^{\dagger}}\big|\\
&\le\ (\deg P)\cdot|v(a^{\dagger}-b^{\dagger})|+(\wt P)\big|\nabla v(a^{\dagger}-b^{\dagger})\big|+{} \\
&\qquad\quad (\deg Q)\cdot |v(a^{\dagger}-b^{\dagger})|+(\wt Q)\big|\nabla v(a^{\dagger}-b^{\dagger})\big|\\
&\le\ (\deg P+\deg Q+1)\big|\nabla(\alpha-\beta)\big|.
\end{split}
\]
If $v(a^{\dagger}-b^{\dagger})\geq0$, then the first inequality still holds, and by Lemma~\ref{additiveconjval} the right-hand side of that inequality is zero.  This shows (\ref{fslow}).  Lastly, if $f(\beta)=\beta$, then from~(\ref{fslow}) we obtain
\[
f(\alpha)-\alpha\ =\ f(\beta)-\alpha+O\big(\nabla(\beta-\alpha)\big)\ =\ \beta-\alpha+O\big(\nabla(\beta-\alpha)\big),
\]
so by Lemma~\ref{nablalemma}(\ref{nablaslow}), $\nabla\big(f(\alpha)-\alpha\big)=\nabla(\beta-\alpha)$ and (\ref{fslow}) becomes
\[
\big|\beta-f(\alpha)\big|\le(\deg P+\deg Q+1)\big|\nabla(f(\alpha)-\alpha)\big|.\qedhere
\]
\end{proof}

\noindent
The uniqueness of the solution to $f(\alpha)=\alpha$ is immediate from Lemma~\ref{fproperties}(\ref{fslow}) and Lemma~\ref{nablalemma}(\ref{nablaslow}).  In view of Lemma~\ref{nablalemma}(\ref{derivativeval}), a reasonable guess for bounding~$v\big((P_d)_{+a}\big)$ is approximately $vP_d+d\cdot|va|$.  We therefore make the following

\begin{definition}  \label{def:qddeg}
The \emph{$\gamma$-quasi-dominant degree $\qddeg_\gamma P$ of $P$} is the largest value of~$m$ which minimizes $vP_m+m\gamma$. The \emph{$\gamma$-quasi-dominant valuation of $P$} is~$v_\gamma(P):=v(P_{m})$ where $m=\qddeg_{\gamma}P$.
\end{definition}

\begin{remark}\label{altqddef}
Equip $F:=K\langle Y\rangle$ with the gaussian extension of the valuation of~$K$. (See \cite[Section~4.5]{ADAMTT}.)
Let $X$ be a new indeterminate and let $P\mapsto\hat{P}$ be the $K$-algebra morphism $K\{Y\}\to F[X]$ with $Y^{(i)}\mapsto XY^{(i)}$ for all $i$.  If $g\in K$ is such that $vg=\gamma$, then $\qddeg_\gamma P=\ddeg \hat{P}_{\times g}$ and $v(\hat{P}_{\times g})=v_\gamma(P)+(\qddeg_\gamma P)\gamma$. 
\end{remark}

\begin{figure}
\begin{tikzpicture}[scale=1]
\foreach \pos in {(0,1.6), (1,3), (2,6),(3,1.3),(4,2.4),(5,1.3),(6,1.4),(7,5.8)}  {\fill \pos circle (2pt) ;}

\foreach \pos/\lbl in {(0,1.6)/1, (7.4,1.6)/2, (0,3)/3, (7.4,3)/4, (0,6)/5, (7.4,6)/6, (0,1.3)/7, (7.4,1.3)/8, (0,2.4)/9, (7.4,2.4)/{10},(0,1.4)/{11}, (7.4,1.4)/{12},(0,5.8)/{13}, (7.4,5.8)/{14}} {\path \pos node[anchor=center, inner sep=0] (\lbl) at \pos {} ;}

\foreach \posl/\posr in {1/2,3/4,5/6,7/8,9/{10},{11}/{12},{13}/{14}} {\draw[dash pattern=on 2pt off 3pt] (\posl) to (\posr) ;}

{\draw[red, semithick] (5,1.3) +(.12,.12) -- +(.12,-.12) -- +(-.12,-.12) -- +(-.12,.12) -- cycle ;}

\foreach \pos/\lbl in {(0,7)/{\Gamma}, (-.2,-.2)/0, (1,-.3)/1, (2,-.3)/2, (3,-.3)/3, (4,-.3)/4, (5,-.3)/5, (6,-.3)/6, (7,-.3)/7} {\path \pos node[anchor=center, inner sep=0] at \pos {$\lbl$} ;}
\foreach \x in {1,2,3,4,5,6,7} {\draw (\x,-.1) -- (\x, .1) ;}

{\draw[line width=1pt, arrows=->] (0,0) to (7.4,0) ;}
{\draw[line width=1pt, arrows=->] (0,3.5) to (0,-1) ;}
{\draw[line width=1pt, arrows=->] (0,5) to (0,6.6) ;}
{\draw[line width=1pt, dash pattern=on 2pt off 5pt] (0,3.5) to (0,5) ;}

\end{tikzpicture}
\caption{The dominant degree}
\end{figure}
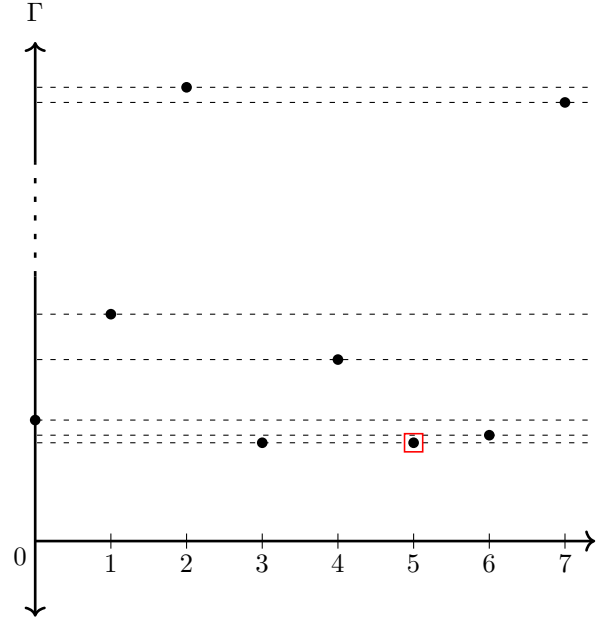

\begin{figure}
\begin{tikzpicture}[scale=1]
\foreach \pos in {(0,1.6), (1,3), (2,6), (3,1.3), (4,2.4), (5,1.3), (6,1.4),(7,5.8)}  {\fill \pos circle (2pt) ;}

\foreach \pos/\lbl in {(0,1.6)/1, (7.4,2.34)/2, (0,2.9)/3, (7.4,3.64)/4, (0,5.8)/5, (7.4,6.54)/6, (0,1)/7, (7.4,1.74)/8, (0,2)/9, (7.4,2.74)/{10}, (0,0.8)/{11}, (7.4,1.54)/{12},(0,0.8)/{13}, (7.4,1.54)/{14},(0,5.1)/{15}, (7.4,5.84)/{16}} {\path \pos node[anchor=center, inner sep=0] (\lbl) at \pos {} ;}

\foreach \posl/\posr in {1/2,3/4,5/6,7/8,9/{10},{11}/{12},{13}/{14},{15}/{16}} {\draw[dash pattern=on 2pt off 3pt] (\posl) to (\posr) ;}

{\draw[red, semithick] (6,1.4) +(.17,0) -- +(0,-.17) -- +(-.17,-.0) -- +(0,.17) -- cycle ;}

\foreach \pos/\lbl in {(0,7)/{\Gamma}, (-.2,-.2)/0, (1,-.3)/1, (2,-.3)/2, (3,-.3)/3, (4,-.3)/4, (5,-.3)/5, (6,-.3)/6, (7,-.3)/7} {\path \pos node[anchor=center, inner sep=0] at \pos {$\lbl$} ;}
\foreach \x in {1,2,3,4,5,6,7} {\draw (\x,-.1) -- (\x, .1) ;}

{\draw[line width=1pt, arrows=->] (0,0) to (7.4,0) ;}
{\draw[line width=1pt, arrows=->] (0,3.5) to (0,-1) ;}
{\draw[line width=1pt, arrows=->] (0,5) to (0,6.6) ;}
{\draw[line width=1pt, dash pattern=on 2pt off 5pt] (0,3.5) to (0,5) ;}

\end{tikzpicture}
\caption{The quasi-dominant degree}
\end{figure}
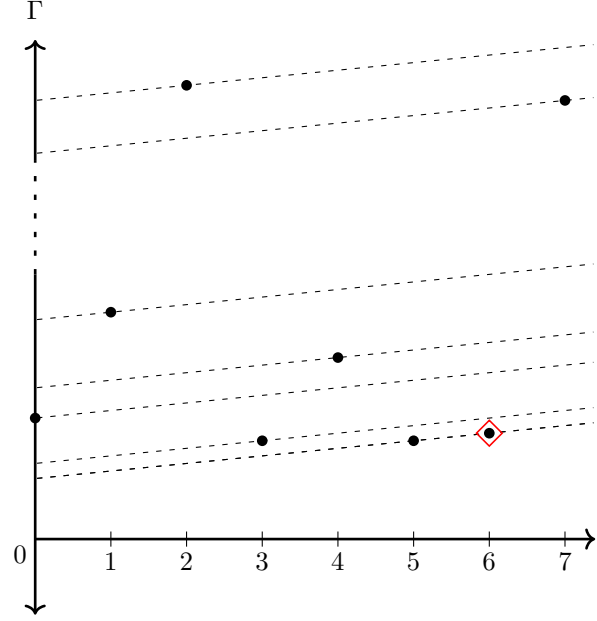

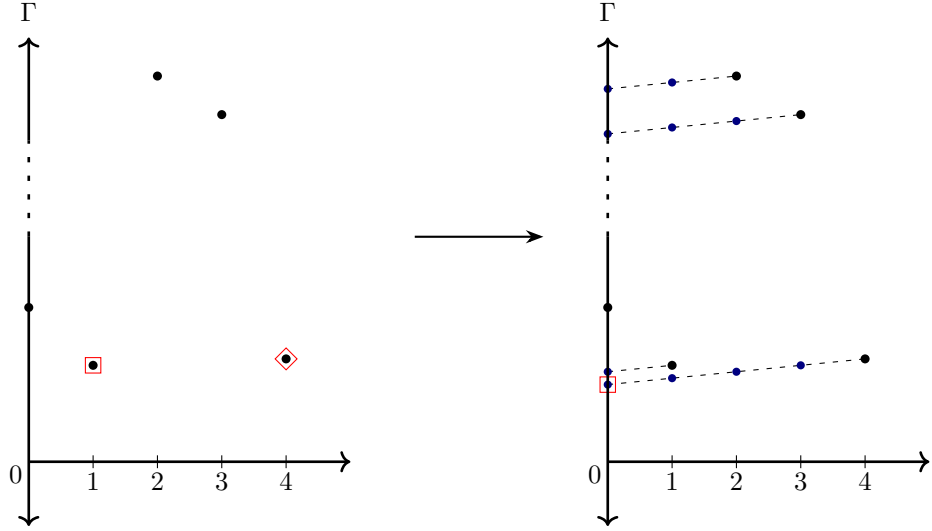
\begin{figure}
\begin{tikzpicture}[scale=0.85]
\foreach \pos/\lbl in {(0,2.4)/0, (0,1.4)/1, (1,1.5)/2, (0,5.8)/3, (1,5.9)/4, (2,6)/5, (0,5.1)/6, (1,5.2)/7, (2,5.3)/8, (3,5.4)/9, (0,1.2)/{10}, (1,1.3)/{11}, (2,1.4)/{12}, (3,1.5)/{13}, (4,1.6)/{14}} {\path \pos node[anchor=center, inner sep=0] (\lbl) at \pos {} ;}

\foreach \pos in {(0,2.4), (1,1.5), (2,6), (3,5.4), (4,1.6)} {\fill \pos circle (2pt) ;}

\foreach \posl/\posr in {1/2,3/5,6/9,{10}/{14}} {\draw[dash pattern=on 2pt off 3pt] (\posl) +(9,0) to ($(\posr) + (9,0)$) ;}

\foreach \pos in {(0,2.4), (1,1.5), (2,6), (3,5.4), (4,1.6)} {\fill \pos +(9,0) circle (2pt) ;}

\foreach \pos in {1, 3, 4, 6, 7, 8, 10, 11, 12, 13} {\fill[color=blue!50!black] ($(\pos)+(9,0)$) circle (1.8pt) ;}

{\draw[red] (1,1.5) +(.12,.12) -- +(.12,-.12) -- +(-.12,-.12) -- +(-.12,.12) -- cycle ;}
{\draw[red] (4,1.6) +(.17,0) -- +(0,-.17) -- +(-.17,-.0) -- +(0,.17) -- cycle ;}

{\draw[red] (0,1.2) ++(9,0) +(.12,.12) -- +(.12,-.12) -- +(-.12,-.12) -- +(-.12,.12) -- cycle ;}

{\draw[thick, arrows=-Stealth] (6,3.5) -- (8,3.5) ;}

\foreach \pos/\lbl in {(0,7)/{\Gamma}, (-.2,-.2)/0, (1,-.3)/1, (2,-.3)/2, (3,-.3)/3, (4,-.3)/4} {\path \pos node[anchor=center, inner sep=0] at \pos {$\lbl$} ;}
\foreach \x in {1,2,3,4} {\draw (\x,-.1) -- (\x, .1) ;}
\foreach \pos/\lbl in {(9,7)/{\Gamma}, (8.8,-.2)/0, (10,-.3)/1, (11,-.3)/2, (12,-.3)/3, (13,-.3)/4} {\path \pos node[anchor=center, inner sep=0] at \pos {$\lbl$} ;}
\foreach \x in {10,11,12,13} {\draw (\x,-.1) -- (\x, .1) ;}

{\draw[line width=1pt, arrows=->] (0,0) to (5,0) ;}
{\draw[line width=1pt, arrows=->] (0,3.5) to (0,-1) ;}
{\draw[line width=1pt, arrows=->] (0,5) to (0,6.6) ;}
{\draw[line width=1pt, dash pattern=on 2pt off 5pt] (0,3.5) to (0,5) ;}

{\draw[line width=1pt, arrows=->] (9,0) to (14,0) ;}
{\draw[line width=1pt, arrows=->] (9,3.5) to (9,-1) ;}
{\draw[line width=1pt, arrows=->] (9,5) to (9,6.6) ;}
{\draw[line width=1pt, dash pattern=on 2pt off 5pt] (9,3.5) to (9,5) ;}

\end{tikzpicture}
\caption{The relationship between quasi-dominant degree and additive conjugation.  Note that the new terms produced by additive conjugation do not need to actually be on the dashed lines; we can only guarantee that they are not significantly below the lines.}
\end{figure}

\begin{lemma}[basic properties of quasi-dominant valuations]
\label{basicndprops}
Let $\gamma,\gamma'\in\Gamma$.  Then
\begin{enumerate}[label=\textup{(\theenumi)}, ref=\theenumi]
\item\label{ndtrivial} $\qddeg_0P=\ddeg P$ and $v_0P=vP$;
\item\label{ndvalshape} if $n>\qddeg_\gamma P$, then $vP_n>v_\gamma P-(n-\qddeg_\gamma P)\gamma$;
\item\label{ndvalbd}$vP\ \le\ v_\gamma P\ \le\ vP+(\ddeg P-\qddeg_\gamma P)\gamma$;
\item\label{nddegbd} $\qddeg_\gamma P\le \deg P$, and if $\gamma\le0$, then $\qddeg_\gamma P\ge\ddeg P$;
\item\label{ndsetcompare} if  $\gamma\le\gamma'\le0$, then $$\qddeg_{\gamma'}P\le\qddeg_{\gamma}P\quad\text{ and }\quad v_{\gamma'}P\le v_\gamma P,$$ with 
$$\qddeg_{\gamma'} P=\qddeg_\gamma P \quad\Longleftrightarrow\quad v_{\gamma'}P=v_{\gamma}P;$$
\item\label{ndsetcomparepos} if  $0<\gamma\le\gamma'$, then $$\qddeg_{\gamma'}P\le\qddeg_{\gamma}P\quad\text{ and }\quad v_{\gamma'}P\ge v_\gamma P,$$ with $$\qddeg_{\gamma'} P=\qddeg_\gamma P \quad\Longleftrightarrow\quad v_{\gamma'}P=v_{\gamma}P;$$
\item\label{ndconj} if $\alpha=va$ is such that $\alpha>\gamma+\order(P)|\nabla\alpha|$, then 
$$\qddeg_\gamma P_{+a}=\qddeg_\gamma P \quad\text{ and }\quad  v_\gamma P_{+a}=v_\gamma P.$$  In particular, if $\gamma<0$ and $\alpha\ge c\gamma$ for some $c\in\Q^{<1}$,  then $$\qddeg_\gamma P_{+a}=\qddeg_\gamma P \quad\text{ and }\quad  v_\gamma P_{+a}=v_\gamma P.$$
\end{enumerate}
\end{lemma}
\begin{proof}
(\ref{ndtrivial}) is trivial.  From the definitions of quasi-dominant degree and valuation, we have that for any $n$,
\[
v_\gamma P+(\qddeg_\gamma P)\gamma\ \le\ vP_n+n\gamma,
\]
and if $n>\qddeg_\gamma P$ then the inequality is strict.  This immediately gives (\ref{ndvalshape}), and taking $n=\qddeg_{\gamma'}P$, we have that for any $\gamma'\in\Gamma$,
\begin{align*}
v_\gamma P+(\qddeg_\gamma P)\gamma &\ \le\ v_{\gamma'}P+(\qddeg_{\gamma'}P)\gamma\qquad\text{and} \\
v_{\gamma'} P+(\qddeg_{\gamma'} P)\gamma' &\ \le\ v_{\gamma}P+(\qddeg_{\gamma}P)\gamma'.
\end{align*}
Part~(\ref{ndvalbd}) follows by taking $\gamma'=0$, and (\ref{nddegbd}) from comparing the left and right sides of (\ref{ndvalbd}).  Subtracting the two inequalities and rearranging gives
\[
(\qddeg_\gamma P-\qddeg_{\gamma'}P)(\gamma-\gamma')\le0,
\]
which gives the quasi-dominant degree inequality in both (\ref{ndsetcompare}) and (\ref{ndsetcomparepos}).  If $\gamma>0$, then the first inequality gives $$v_\gamma P-v_{\gamma'} P\le(\qddeg_{\gamma'} P-\qddeg_{\gamma}P)\gamma,$$ which implies the rest of (\ref{ndsetcomparepos}).  The $\gamma'<0$ case of (\ref{ndsetcompare}) is similar, while the case $\gamma'=0$ follows from replacing $\gamma$ by $\gamma'$ in (\ref{ndvalshape}) and taking $n=\qddeg_{\gamma}P$.

For (\ref{ndconj}), we apply Lemma~\ref{nablalemma} to show that, for any $\boldsymbol{i}\le\boldsymbol{j}$, 
\[
\begin{split}
v\big((P_{\boldsymbol{j}}Y^{\boldsymbol{j}})_{+a,\boldsymbol{i}}\big)\ &=\ vP_{\boldsymbol{j}}+v(a^{\boldsymbol{j}-\boldsymbol{i}})\\
&\ge\ vP_{\boldsymbol{j}}+|\boldsymbol{j}-\boldsymbol{i}|\cdot\alpha-\|\boldsymbol{j}-\boldsymbol{i}\|\cdot|\nabla\alpha|\\
&\ge\ vP_{|\boldsymbol{j}|}+(|\boldsymbol{j}|-|\boldsymbol{i}|)\gamma,
\end{split}
\]
with equality iff $\boldsymbol{i}=\boldsymbol{j}$ and $vP_{\boldsymbol{j}}=vP_{|\boldsymbol{j}|}$.  Given $m\le n$, summing $(P_{\boldsymbol{j}}Y^{\boldsymbol{j}})_{+a,\boldsymbol{i}}$ over $\boldsymbol{i}$ and $\boldsymbol{j}$ with $|\boldsymbol{i}|=m$ and $|\boldsymbol{j}|=n$ gives $v(P_{n,+a,m})+m\gamma\ge vP_{n}+n\gamma$, with equality iff $m=n$.  Then for any $m$, 
\[
vP_{+a,m}+m\gamma\ \ge\ \min_{n\ge m}(vP_n+n\gamma)\ \ge\ v_\gamma P+(\qddeg_\gamma P)\gamma,
\]
and if $m=\qddeg_\gamma P$ and $n>m$, then
\[
vP_{n,+a,m}\ \ge\ vP_n+(n-m)\gamma\ >\ vP_m.
\]
Since $P_{m,+a,m}=P_m$, this implies that $vP_{+a,m}=vP_m$ for $m=\qddeg_\gamma P$, and hence $\qddeg_\gamma P_{+a}=\qddeg_\gamma P$ and $v_\gamma P_{+a}=v_\gamma P$, as required.

The second part of (\ref{ndconj}) follows by noting that if $\gamma<0$ and $\alpha\ge c\gamma$ where~${c\in\Q^{<1}}$,  then $\alpha-\gamma$ is in the same archimedean class as $\alpha$, so by Lemma~\ref{nablalemma}(\ref{nablashrink}), $\nabla(\alpha)=o(\alpha-\gamma)$ and $\alpha>\gamma+k|\nabla\alpha|$ for any $k\in\Z$.
\end{proof}

\noindent
In order to modify $f$ to use $\qddeg_\gamma$, we must come up with an appropriate $\gamma$.  The goal of this modification is to get a better handle on the result of additive conjugation, so we need an idea of what we are likely to conjugate by in the future.  This is the purpose of Lemma~\ref{fproperties}(\ref{fnear}): if we search for a fixed point of $f$ starting from $\alpha$, then we should not move by more than $(\deg P+\deg Q+1)\big|\nabla(f(\alpha)-\alpha)\big|$.  We therefore make the following definitions:
\begin{align*}
\gamma(\alpha)\ &:=\ \begin{cases}2\nabla(f(\alpha)-\alpha)&\quad\text{if }f(\alpha)\ne\alpha\\0&\quad\text{if }f(\alpha)=\alpha,\end{cases}\\
g(\alpha)\ &:=\ 
v_{\gamma(\alpha)}P_{+pa^{\dagger}}-v_{\gamma(\alpha)}Q_{+qa^{\dagger}},\quad\text{where $va=\alpha$}.
\end{align*}
The remainder of the proof consists of establishing a series of properties of $g$ and~$\gamma$:

\begin{lemma}\label{gproperties}
\mbox{}

\begin{enumerate}[label=\textup{(\theenumi)}, ref=\theenumi]
\item\label{gammaneg} $\gamma(\alpha)\le0$;
\item\label{gwelldef} $g(\alpha)$ is well-defined, i.e., does not depend on the choice of $a$ with~$va=\alpha$;
\item\label{gapprox} $\nabla\big(\alpha-g(\alpha)\big)=\nabla\big(\alpha-f(\alpha)\big)$, and in particular, $g(\alpha)=\alpha$ iff $f(\alpha)=\alpha$;
\item\label{gammadecrease} if $f(g(\alpha))\ne g(\alpha)$, then $\gamma(\alpha)<0$ and $\gamma(g(\alpha))=o\big(\gamma(\alpha)\big)>\gamma(\alpha)$;
\item\label{gdegdecrease} if $\beta=g(\alpha)\ne\alpha$ and   $va=\alpha,vb=\beta$, then $$\qddeg_{\gamma(\beta)}P_{+pb^{\dagger}}\le\qddeg_{\gamma(\alpha)}P_{+pa^{\dagger}},\quad \qddeg_{\gamma(\beta)}Q_{+qb^{\dagger}}\le\qddeg_{\gamma(\alpha)}Q_{+qa^{\dagger}},$$ and if both of these are equalities then $g(g(\alpha))=g(\alpha)$.
\end{enumerate}
\end{lemma}
\begin{proof}
Item (1) is clear from the definition, and so item (2) holds by Lemma~\ref{basicndprops}(\ref{ndconj}).  By Lemma~\ref{basicndprops}(\ref{ndvalbd}), if $f(\alpha)\ne\alpha$, then
\[
\big|g(\alpha)-f(\alpha)\big|\ =\ O(\gamma(\alpha))\ =\ O\big(\nabla(\alpha-f(\alpha)\big)\big),
\]
so $\nabla\big(\alpha-g(\alpha)\big)=\nabla\big(\alpha-f(\alpha)\big)$ by Lemma~\ref{nablalemma}(\ref{nablaslow}).  If $f(\alpha)=\alpha$, then $g(\alpha)=f(\alpha)=\alpha$ by Lemma~\ref{basicndprops}(\ref{ndtrivial}), which completes the proof of (\ref{gapprox}).

Suppose that $f(\alpha)\ne\alpha$.  Then
\[
\begin{split}
f(g(\alpha))-g(\alpha)\ &=\ f(g(\alpha))-f(\alpha)+f(\alpha)-g(\alpha)
\\&=\ O\big(\nabla(g(\alpha)-\alpha)\big)+O\big(\nabla(f(\alpha)-\alpha)\big)
\\&=\ O\big(\nabla(f(\alpha)-\alpha)\big).
\end{split}
\]
If $f(g(\alpha))\ne g(\alpha)$, then the left-hand side is nonzero, so applying $\nabla$ to both sides and using Lemma~\ref{nablalemma}(\ref{nablaslow}) will give (\ref{gammadecrease}).

Now set $\beta:=g(\alpha)$ and take $a$, $b$ with $va=\alpha,vb=\beta$.  Then $v(pb^{\dagger}-pa^{\dagger})\geq\nabla\big(f(\alpha)-\alpha\big)$ and $\nabla\big(f(\alpha)-\alpha\big)=\frac12\gamma(\alpha)\le0$, so by Lemma~\ref{basicndprops}(\ref{ndsetcompare}) and (\ref{ndconj}), we have
\[
\begin{split}
\qddeg_{\gamma(\beta)}P_{+pb^{\dagger}}\ &\le\ \qddeg_{\gamma(\alpha)}P_{+pb^{\dagger}}\\
&=\ \qddeg_{\gamma(\alpha)}P_{+pa^{\dagger}},
\end{split}
\]
with equality iff $v_{\gamma(\beta)}P_{+pb^{\dagger}}=v_{\gamma(\alpha)}P_{+pb^{\dagger}}$.  The same holds with $Q$ and $q$ in place of $P$ and $p$.  If equality holds for both $P$ and $Q$, then from the definition of $g$ we obtain $g(g(\alpha))=g(\beta)=g(\alpha)$.
\end{proof}

\begin{proof}[Proof of Theorem~\ref{equalizerthmriccati}]
Let
\[
r(\alpha)\ :=\ \qddeg_{\gamma(\alpha)}P_{+pa^{\dagger}}+\qddeg_{\gamma(\alpha)}Q_{+qa^{\dagger}}, \quad\text{where $va=\alpha$}.
\]
By Lemma~\ref{gproperties}(\ref{gdegdecrease}), for any $\alpha$, either $g(\alpha)=\alpha$, or $g(g(\alpha))=g(\alpha)$, or $r(g(\alpha))<r(\alpha)$.  Since $0\le r(\alpha)\le \deg P+\deg Q$, if we start with an arbitrary $\alpha_0$, then the sequence $$\alpha_0,\ \alpha_1=g(\alpha_0),\ \alpha_2=g(\alpha_1),\ \ldots$$ must stabilize at $\alpha_n$ for some $n\le\deg P+\deg Q+1\leq 2s+1$.  By Lemma~\ref{gproperties}(\ref{gapprox}), we then have $f(\alpha_n)=\alpha_n$, so $vP_{+pa^\dagger}-vQ_{+qa^\dagger}=va$ when $va=\alpha_n$.  To justify the claim about extensions, note that replacing $K$ by a valued differential field extension with small derivation does not change $f$, $\gamma$, $\qddeg$, and $g$, and that Lemma~\ref{fproperties} still goes through and shows uniqueness of the fixed point of~$f$.

Consider the (potentially undefined for a particular choice of $Q$, $\bm{j}$) $\d$-rational functions $$\widetilde{g}_{P,Q,\bm{i},\bm{j}}(Z):= P_{+pZ^\dagger,\bm{i}}/Q_{+qZ^\dagger,\bm{j}}\in K\langle Z\rangle.$$  For any $a\ne 0$, there are    $\bm{i}$, $\bm{j}$ such that
($\widetilde{g}_{P,Q,\bm{i},\bm{j}}(a)$ is defined
and)~$g( va)=v\big(\widetilde{g}_{P,Q,\bm{i},\bm{j}}(a)\big)$.  Thus we obtain $\bm{i}_1,\ldots,\bm{i}_n$ and $\bm{j}_1,\ldots,\bm{j}_n$ with $|\bm{i}_{k+1}|+|\bm{j}_{k+1}|<|\bm{i}_k|+|\bm{j}_k|$ for   $k=1,\dots,n-1$ such that, if
\[
a\ =\ \widetilde{g}_{P,Q,\bm{i}_n,\bm{j}_n}\big(\widetilde{g}_{P,Q,\bm{i}_{n-1},\bm{j}_{n-1}}\big(\cdots\big( \widetilde{g}_{P,Q,\bm{i}_1,\bm{j}_1}(1)\big)\big)\big),
\]
then $va=\alpha_n$.
Call a sequence $\bm{i}_1,\bm{j}_1,\ldots,\bm{i}_n,\bm{j}_n$ \emph{admissible} for $\d$-polynomials $P$, $Q$ as in the statement of the theorem if there are a sequence $\alpha_0,\dots,\alpha_n$ in $\Gamma$ obtained as above (that is, $\alpha_k=g(\alpha_{k-1})$ for $k=1,\dots,n$), and $a_0=1,a_1,\dots,a_n\in K$ with $va_k=\alpha_k$ for~$k=1,\dots,n$, such that 
$$vP_{+pa_{k-1}^{\dagger},\bm{i}_k}=v_{\gamma(\alpha_{k-1})}P_{+pa_{k-1}^{\dagger}},\quad 
vQ_{+qa_{k-1}^{\dagger},\bm{j}_k}=v_{\gamma(\alpha_{k-1})}Q_{+qa_{k-1}^{\dagger}}\quad (k=1,\ldots,n).$$ 
Regarding the coefficients of $P$ and $Q$ as tuples $P_\circ$ and $Q_\circ$ of new differential in\-de\-terminates, we obtain generic $\d$-polynomials $\widetilde P,\widetilde Q\in (\Q\{ P_{\circ},Q_{\circ}\})\{Y\}$ which depend only on the super-weight bound $s$. If $\bm{i}$, $\bm{j}$ appears in an admissible sequence for some particular $P=\widetilde P(p_\circ,Y)$, $Q=\widetilde Q(q_\circ,Y)$, then~$\widetilde{g}_{P,Q,\bm{i},\bm{j}}(Z)\in\Q\langle p_\circ,q_\circ\rangle\langle Z\rangle$.
Lemma~\ref{lemma: drational composition} yields 
a $\d$-rational function $H_{\bm{i}_1,\bm{j}_1,\ldots,\bm{i}_n,\bm{j}_n}\in\Q\langle P_\circ, Q_\circ\rangle$ such that
  if~$\bm{i}_1,\bm{j}_1,\ldots\bm{i}_n,\bm{j}_n$ is admissible for $P=\widetilde P(p_\circ,Y)$, $Q=\widetilde Q(q_\circ,Y)$, then
\[
a:=H_{\bm{i}_1,\bm{j}_1,\ldots,\bm{i}_n,\bm{j}_n}(p_\circ,q_\circ) = \big(\widetilde{g}_{P,Q,\bm{i}_n,\bm{j}_n}\circ\widetilde{g}_{P,Q,\bm{i}_{n-1},\bm{j}_{n-1}}\circ \cdots\circ\widetilde{g}_{P,Q,\bm{i}_1,\bm{j}_1}\big)(1)\in K
\]
is defined and  satisfies $vP_{+pa^\dagger}-vQ_{+qa^\dagger}=va$.

Since $P$, $Q$ have super-weight at most $s$, there are at most~$2^s$ choices for each~$\bm{i}_k$ or~$\bm{j}_k$, and thus at most $(2^{2s})^{2s+1}$ admissible sequences for $\d$-polynomials as in the statement of the theorem.  It remains to bound the complexity of the~$H$'s.  For this, fix~$\bm{i}$,~$\bm{j}$, and consider~$\widetilde{g}_{\widetilde{P},\widetilde{Q},\bm{i},\bm{j}}(R/S)$, where~$R,S\in\Q\{P_\circ,Q_\circ\}$.  Setting~$b_{\bm{i},\bm{k},\bm{l}}:=(-1)^{|\bm{l}|}p^{|\bm{k}-\bm{i}|}\binom{\bm{k}}{\bm{i},\bm{l},\bm{k}-\bm{i}-\bm{l}}$ we then have
\[
\widetilde{P}_{+p(R^\dagger-S^\dagger),\bm{i}}\ =\ \sum_{\bm{k}\geq\bm{i}}\sum_{\bm{l}\leq\bm{k}-\bm{i}}b_{\bm{i},\bm{k},\bm{l}}(R^\dagger)^{\bm{k}-\bm{i}-\bm{l}}(S^\dagger)^{\bm{l}}\widetilde{P}_{\bm{k}},
\]
where $\widetilde{P}_{\bm{k}}=P_{\bm{k}}$ is one of the indeterminates in $P_\circ$, and this also equals
$$R^{\|\bm{i}\|'-s}S^{\|\bm{i}\|'-s} \sum_{\substack{\bm{k}\geq\bm{i}\\\bm{l}\leq\bm{k}-\bm{i}}}b_{\bm{i},\bm{k},\bm{l}}R^{s-\|\bm{k}-\bm{l}\|'}S^{s-\|\bm{i}+\bm{l}\|'}S_{\bm{k}-\bm{i}-\bm{l}}(R)S_{\bm{l}}(S)P_{\bm{k}}.$$
By Lemma~\ref{superweightbounds}, 
$$\swt\!\big(R^{s-\|\bm{k}-\bm{l}\|'}S_{\bm{k}-\bm{i}-\bm{l}}(R)\big)\, \leq\, (s-\|\bm{k}-\bm{l}\|')\swt R+ \|\bm{k}-\bm{i}-\bm{l}\|'(\swt R+1)$$
and thus
$$\swt\!\big(R^{s-\|\bm{k}-\bm{l}\|'}S_{\bm{k}-\bm{i}-\bm{l}}(R)\big)\,\leq\, (s-\|\bm{i}\|')\swt R+\|\bm{k}-\bm{i}-\bm{l}\|',$$
and similarly
\[
\swt\big(S^{s-\|\bm{i}+\bm{l}\|'}S_{\bm{l}}(S)\big)\ \leq\ (s-\|\bm{i}\|')\swt S+\|\bm{l}\|',
\]
so
\[
\swt\big(R^{s-\|\bm{k}-\bm{l}\|'}S^{s-\|\bm{i}+\bm{l}\|'}S_{\bm{k}-\bm{i}-\bm{l}}(R)S_{\bm{l}}(S)\big)\ \leq\ (s-\|\bm{i}\|')(\swt R+\swt S+1).
\]
Applying similar reasoning to $\widetilde{Q}_{+q(R^\dagger-S^\dagger),\bm{j}}$, and accounting for cancellation between the initial $R^{\|\bm{i}\|'-s}S^{\|\bm{i}\|'-s}$ and $R^{\|\bm{j}\|'-s}S^{\|\bm{j}\|'-s}$, we find that if~$\widetilde{g}_{\widetilde{P},\widetilde{Q},\bm{i},\bm{j}}(R/S)=\widetilde{R}/\widetilde{S}$ in lowest terms, then
$\swt\widetilde{R}+\swt\widetilde{S}$ is bounded above by
$$(2s-\|\bm{i}\|'-\|\bm{j}\|')(\swt R+\swt S+1)+2+|\|\bm{i}\|'-\|\bm{j}\|'|,$$
hence
$$ \swt\widetilde{R}+\swt\widetilde{S} \ \leq\ (2s-\|\bm{i}\|'-\|\bm{j}\|')(\swt R+\swt S)+2s+2.$$
Now consider
\[
H_{\bm{i}_1,\bm{j}_1,\ldots,\bm{i}_n,\bm{j}_n}\ =\ \big(\widetilde{g}_{\widetilde{P},\widetilde{Q},\bm{i}_n,\bm{j}_n}\circ\widetilde{g}_{\widetilde{P},\widetilde{Q},\bm{i}_{n-1},\bm{j}_{n-1}}\circ\cdots\circ\widetilde{g}_{\widetilde{P},\widetilde{Q},\bm{i}_1,\bm{j}_1}\big)(1).
\]
Since $\|\bm{i}_k\|'+\|\bm{j}_k\|'\geq |\bm{i}_k|+|\bm{j}_k|\geq n-k$ and $\widetilde{g}_{\widetilde{P},\widetilde{Q},\bm{i}_1,\bm{j}_1}(1)=P_{\bm{i}_1}/Q_{\bm{j}_1}$, we find that~$H_{\bm{i}_1,\bm{j}_1,\ldots,\bm{i}_n,\bm{j}_n}(P_\circ,Q_\circ)=R/S$ with $\swt R+\swt S\leq (s+2)\cdot(2s+1)!$.
\end{proof}

\begin{remark}
By Lemma~\ref{gproperties}(\ref{gammadecrease}), the sequence $\big(\frac12\gamma(\alpha_i)\big)$ is strictly increasing until it stabilizes at $0$, and in fact the sequence of archimedean classes $\big([\gamma(\alpha_i)]\big)$ is strictly decreasing until it stabilizes at $[0]$.  Thus, if $\big[\nabla(\Gamma^{\neq})\big]$ is finite, then the sequence~$(\alpha_i)$ stabilizes after at most $N:=\big|\big[\nabla(\Gamma^{\neq})\big]\big|$ terms.  This provides a bound on the  length of the sequence which is independent of $P$ and $Q$ but depends on $N$.
In connection with this we note that if $K$ is  a differential subfield of~$\mathbb{T}$ which is finitely generated over $\mathbb{R}$, then~$\nabla(\Gamma^{\neq})$ is finite by~\cite[Proposition~3.13]{HFSurvey}.
\end{remark}

\noindent
A straightforward modification of the proof of Theorem~\ref{equalizerthmriccati} provides a generalization to more complicated equations:

\begin{cor}
    Fix $m,s\in\N$ and $k_1,\ldots,k_m\in\Z$.  There are $N\leq 2^{ms(ms+1)}$ and  $\d$-rational functions $H_1,\ldots,H_N$ over $\Q$ such that for any $P^1,\ldots,P^m\in K\{Y\}^{\neq}_{[\leq s]'}$ and any $c_1,\ldots,c_m\in C$, there is an $i\in\{1,\ldots,N\}$ such that for each~$f\neq0$ in a valued differential field extension of $K$ with small derivation, we have
    \[
    vf\, =\, k_1vP^1_{+c_1f^\dagger}+\cdots+k_mvP^m_{+c_mf^\dagger} \  \iff \  f\asymp H_i(P^1,\ldots,P^m,c_1,\ldots,c_m).
    \]
\end{cor}
\begin{proof}
We sketch the necessary modifications to the proof of Theorem~\ref{equalizerthmriccati}.  First, change the definition of $f$ to
\[
f\colon \Gamma\to\Gamma,\qquad va\mapsto k_1vP^1_{+c_1a^\dagger}+\cdots+k_mvP^m_{+c_ma^\dagger}.
\]
In Lemma~\ref{fproperties}, replace the factor $\deg P+\deg Q+1$ in (\ref{fslow}) and (\ref{fnear}) by $1+\sum_i|k_i|\deg P^i$.  The proof is essentially the same.
The definition of $\gamma$ does not change, but $g$ becomes
\[
g(\alpha)\ :=\ 
k_1v_{\gamma(\alpha)}P^1_{+c_1a^{\dagger}}+\cdots+k_mv_{\gamma(\alpha)}P^m_{+c_ma^{\dagger}},\quad\text{where $va=\alpha$}.
\]
Replace Lemma~\ref{gproperties}(\ref{gdegdecrease}) by:
\begin{itemize}
\item[(\ref{gdegdecrease})] if $\beta=g(\alpha)\ne\alpha$ and $va=\alpha$, $vb=\beta$, then 
\[
\qddeg_{\gamma(\beta)}P^i_{+c_ib^{\dagger}}\le\qddeg_{\gamma(\alpha)}P^i_{+c_ia^{\dagger}}\quad \text{for $i=1,\ldots,m$},
\]
and if this is an equality for all $i$ then $g(g(\alpha))=g(\alpha)$.
\end{itemize}
In the proof of Lemma~\ref{gproperties}(\ref{gdegdecrease}), replace $P_{+pa^\dagger}$, $P_{+pb^\dagger}$ by $P^i_{+c_ia^\dagger}$, $P^i_{+c_ib^\dagger}$, respectively, with $i\in\{1,\ldots,m\}$ arbitrary, remove the line about $Q$, and change the hypothesis of the last line from ``equality holds for both $P$ and $Q$'' to ``equality holds for all~$i=1,\ldots,m$''.
Also
change the definition of $r$ to
\[
r(\alpha)\ :=\ \qddeg_{\gamma(\alpha)}P^1_{+c_1a^\dagger}+\cdots+\qddeg_{\gamma(\alpha)}P^m_{+c_ma^\dagger}, \quad\text{where $va=\alpha$,}
\]
and replace the bound on the $n$ for which $\alpha_{n+1}=\alpha_n$ by $1+\sum_i\deg P^i\leq ms+1$.
Finally, replace the $\d$-rational functions $\widetilde{g}_{P,Q,\bm{i},\bm{j}}$ by
\[
\widetilde{g}_{P^1,\ldots,P^m,c_1,\ldots,c_m,\bm{i}_1,\ldots,\bm{i}_m}(Z)\ := \prod_{j=1}^m\big(P^j_{+c_jZ^\dagger,\bm{i}_j}\big)^{k_j},
\]
and take the $H_i$ to be compositions of $ms+1$ these.  The bound then becomes that there are at most $2^s$ choices for each $\bm{i}_j$ and thus at most $(2^{ms})^{ms+1}$ of the $H$'s.
\end{proof}

\noindent
Without going into details, we mention that Theorem~\ref{equalizerthmriccati} can also be extended to $\d$-polynomials in several indeterminates satisfying several equations.  This leads to the following result on equalizers:

\begin{remark}\label{multivariableequalizers}
Assume that $\Gamma$ is divisible. Fix $m$, $n$, $w$ with $m\le n$, and 
let~$i$,~$j$ range over~$\{0,\dots,m\}$,  $\{1,\dots,n\}$, respectively.
Let $d_{ij}\in\N$  be such that the matrix~${(d_{ij}-d_{0j})_{i,j\ge 1}\in\Z^{m\times n}}$ has rank $m$.  Then there are $D\in \N^{\ge 1}$, $N\in\N$, partial functions $$\gamma_1,\ldots,\gamma_n\colon K\{Y_1,\ldots,Y_n\}^{m+1}\times\Gamma^{n-m}\rightharpoonup\Gamma$$ and $\d$-rational functions $H_1,\ldots,H_N$ with rational coefficients, as well as
a matrix~$M\in \Q^{n\times (n-m)}$ of rank $n-m$, such that for all 
$P^0,\ldots,P^m\in K\{Y_1,\ldots,Y_n\}$    which are homogeneous in each variable, with~$\deg_{Y_j}P^i=d_{ij}$ for each $i$, $j$, we have:
\begin{enumerate}[(i)]
    \item if $a_j\in K^\times$ satisfy $$va_j=\gamma_j(P^0,\ldots,P^m,\beta_{m+1},\ldots,\beta_n)\quad\text{ for each $j$,}$$ then 
    \[
    P^0(a_1Y_1,\ldots,a_nY_n) \ \asymp\ \cdots\ \asymp\ P^m(a_1Y_1,\ldots,a_nY_n);
    \]
    \item if $b_{m+1},\ldots,b_n\in K^\times$ and $j\ge 1$, then there is some $k$ such that 
    \[
    \gamma_j(P^0,\ldots,P^m, \beta_{m+1},\ldots,\beta_n)\ =\ \textstyle\frac1DvH_k(P^0,\ldots,P^m, b_{m+1},\ldots,b_n);
    \]
    \item for  $\delta\in\Gamma^>$ and $\beta_{m+1},\ldots,\beta_n, \widetilde{\beta}_{m+1},\ldots,\widetilde{\beta}_n\in\Gamma$ such that $|\beta_k-\widetilde{\beta}_k|\leq\delta$, for each $j$,
    \begin{multline*}
    \gamma_j(P^0,\ldots,P^m,\beta_{m+1}, \ldots,\beta_n)-\gamma_j(P^0,\ldots,P^m,\widetilde{\beta}_{m+1},\ldots,\widetilde{\beta}_n) \\ \ =\ \sum_{k=1}^{n-m}M_{jk}(\beta_{m+k}-\widetilde{\beta}_{m+k})+o(\delta).
    \end{multline*}
\end{enumerate}
Here suitable $H_1,\dots,H_N$ and $N$ are computable from $m$, $n$, $w$, and the $d_{ij}$, with computable bounds on their complexity.  When $n>m$, equalizers are not unique, but are parametrized by $\beta_{m+1},\ldots,\beta_n$; this parametrization is equivalent to choosing $d_{ij}$ for $m+1\leq i\leq n$, $1\leq j\leq n$ such that $(d_{ij}-d_{0j})_{i,j\ge 1}$ is invertible and adding $P^k=b_kY_1^{d_{k1}}\cdots Y_n^{d_{kn}}$ for $k=m+1,\ldots,n$ to the list of $\d$-polynomials, where $vb_k=\beta_k$.
\end{remark}

\subsection*{An example}
The simplest nontrivial case for finding an equalizer is between a $\d$-polynomial $Y'-sY$ ($s\in K$) of order $1$ and degree~$1$  and a $\d$-polynomial $b\in K^{\ne}$ of degree~$0$.  In this case, the Riccati transforms are respectively $P:=Y-s$ and~$Q:=b$, and the iteration is guaranteed to terminate by step $\deg P+\deg Q+1=2$ with~$\alpha_2$. To illustrate our algorithm, we construct here~$\alpha_2$ explicitly. First we note that
$$P_{+a^\dagger}=Y-(s-a^\dagger),\qquad f(va)=vb-\min\{0,v(s-a^\dagger)\}.$$  Our iteration now proceeds as follows:
First, we set $a_0:=1$, $\alpha_0:=va_0=0$, and   calculate $f(\alpha_0)=vb-\min\{0,vs\}$.

\medskip
\noindent\underline{Case 1}: $s\dominatedby1$.
In this case, $$1=\deg P\geq\qddeg_{\gamma(0)}P\geq\ddeg P=1,$$ so we don't need to compute~$\gamma(0)$ and   immediately obtain $\alpha_1=g(0)=f(0)=vb$.  We then set~$a_1:=b$ and compute $f(vb)=vb-\min\{0,v(s-b^\dagger)\}$.

\medskip
\noindent\underline{Case 1a}: $s-b^\dagger\dominatedby1$.
In this case, $f(\alpha_1)=vb=\alpha_1$ and we are done.

\medskip
\noindent\underline{Case 1b}: $s-b^\dagger\succ 1$.
Here 
$$f(\alpha_1)=vb-v(s-b^\dagger), \qquad \gamma(\alpha_1)=2\nabla(f(\alpha_1)-\alpha_1))=2\nabla(v(s-b^\dagger)).$$  Since $v(s-b^\dagger)<0$, we have $0+2\nabla(v(s-b^\dagger))>v(s-b^\dagger)$ by Lemma~\ref{nablalemma}(\ref{nablacontract}), so~$\qddeg_{\gamma(\alpha_1)}P_{+a_1^\dagger}=\ddeg P_{+a_1^\dagger}=0$ and 
$$\alpha_2=g(\alpha_1)=f(\alpha_1)=vb-v(s-b^\dagger).$$
We can then verify that $f(\alpha_2)=\alpha_2$.  Since $s-b^\dagger\succ1$, we also have $(s-b^\dagger)^\dagger\prec(s-b^\dagger)$.  We then have 
$$f(\alpha_2)=vb-\min\big\{0,v\big(s-b^\dagger+(s-b^\dagger)^\dagger)\big\}=vb-v(s-b^\dagger)=\alpha_2.$$

\noindent\underline{Case 2}: $s\succ1$.
In this case, 
$$\ddeg P_{+a_0^\dagger}=0,\quad  f(\alpha_0)=vb-vs,\quad\text{ and }\quad\gamma(\alpha_0)=2\nabla(vb-vs).$$

\noindent\underline{Case 2a}: $vs<2\nabla(vb-vs)$.
In this case, 
$$\qddeg_{\gamma(\alpha_0)}P_{+a_0^\dagger}=0=\ddeg P_{+a_0^\dagger},\qquad \alpha_1=g(\alpha_0)=f(\alpha_0)=vb-vs.$$  With $a_1:=b/s$, we then compute $$f(\alpha_1)=vb-\min\{0,v(s-(b/s)^\dagger)\}=vb-vs,$$ since $s\succ 1$ and $vs<2\nabla(vb-vs)\leq v(b/s)^\dagger$.  Since $f(\alpha_1)=\alpha_1$, we are done.

\medskip
\noindent\underline{Case 2b}: $vs\geq 2\nabla(vb-vs)$.
In this case, 
$$\qddeg_{\gamma(\alpha_0)}P_{+a_0^\dagger}=1,\qquad \alpha_1=g(\alpha_0)=vb.$$  Then with $a_1:=b$, we have $f(\alpha_1)=vb-\min\{0,v(s-b^\dagger)\}$.

\medskip
\noindent\underline{Case 2bI}: $s-b^\dagger\dominatedby1$.
In this case, $f(\alpha_1)=vb=\alpha_1$ and we are done as in Case~1a.

\medskip
\noindent\underline{Case 2bII}: $s-b^\dagger\succ 1$.
This case proceeds identically to Case~1b.

\medskip
\noindent
We note that the procedure in Case~2bI and Case~2bII did not use the assumption that $vs\geq 2\nabla(vb-vs)$. Thus, while Case~2a occurs in our algorithmic solution, we may remove it from the final case distinction.  Indeed, if $s\succ 1$ and $vs<2\nabla(vb-vs)$, then $s\succ s^\dagger$ and $vs<\nabla v(b/s)\leq v(b/s)^\dagger$, so $s-b^\dagger=s-s^\dagger-(b/s)^\dagger\sim s$.  Thus
\[
\text{the equalizer of $Y'-sY$ and $b$ is}\quad\begin{cases}vb & \text{ if }s-b^\dagger\dominatedby1,\\v\!\left(\frac{b}{s-b^\dagger}\right) & \text{ if }s-b^\dagger\succ1.\end{cases}
\]
We note that the use of $g$ in place of $f$ was necessary to distinguish Case~2bI from the rest of Case~2.  Indeed, with 
$$K:=\T,\quad b:=\ex^{\ex^{\ex^x}},\quad s:=b^\dagger=\ex^{\ex^x+x},$$ recalling that 
$f(va) =vb-\min\{0,v(s-a^\dagger)\}$, we obtain
$$f(0)=v\big(\ex^{\ex^{\ex^x}-\ex^x-x}\big), \quad 
f^2(0)=v\big(\ex^{\ex^{\ex^x}-x}\big), \quad 
f^3(0)=vb.$$
With $b$ sufficiently large and $s:=b^\dagger$, reaching the equalizer can take arbitrarily many iterations of $f$, and if we take $b\succ\T$ in an appropriate extension of $\T$ then finitely many iterates of $f$ do not suffice.  Appendix~\ref{sec:equalizer example} expands on this construction.

\section{Achieving Cleanness}\label{sec:clean}

\noindent
\emph{In this section, $K$ is an $\upo$-free $H$-asymptotic field  with small derivation and~$2\Gamma=\Gamma$.}  Recall that $\phi$ ranges over active elements of $K$.

The eventual equalizer of $P$ and $Q$ is the equalizer of $P^\phi$ and $Q^\phi$, eventually with respect to $\phi$. Theorem~\ref{equalizerthm} gives this equalizer as the valuation of one of several $\d$-rational functions in the coefficients of $P^\phi$ and $Q^\phi$, which are themselves $\d$-polynomials in $\phi$ and the coefficients of $P$ and $Q$.  Under some fairly restrictive conditions, the eventual behavior of a $\d$-polynomial in $\phi$ can be understood using Lemma~\ref{vPylemma} (take $\phi=y'$).  In this section, we show that these conditions can be arranged through a suitable, definable, compositional conjugation.  Without the definability, this was established in~\cite[Section~13.3]{ADAMTT}.

Differential polynomials which have the form $P=D+R$ with $D\in K\cdot C[Y](Y')^{\N}$ and~${R\prec^{\flat} D}$ are particularly well-behaved.  (See \cite[Chapter~13]{ADAMTT}.)  For many purposes, including ours, it is sufficient to impose   the weaker condition
$D\in K[Y](Y')^{\N}$ and~$D_{\bm{i}}\asymp P$ for any~$\bm{i}$ such that~$D_{\bm{i}}\neq0$.  If $P$ satisfies the first condition, then it is said to be \emph{in clean form}, and if it satisfies the second, then it is \emph{in semi-clean form}.  Any~$\phi$  such that $P^\phi$ is in (semi-) clean form (in the compositional conjugate~$K^\phi$, so~$P^\phi=D+R$ with $R\prec^{\flat}_{\phi} D$ where $D\in K[Y](Y')^{\N}$ has the respective appropriate shape) is called a (\emph{semi-}) \emph{clean conjugator} for $P$.

If $P$ is in semi-clean form, then by Lemma~\ref{flatpreservation}, $P^\phi\sim^{\flat}\phi^w P$ for any $\phi\dominatedby1$, where~$w:=\dwt P$.  If $P$ is in clean form and $K$ is $\d$-valued, then additionally~$P\sim^\flat\dd_P N_P$.  If $\phi_0$ is a (semi-) clean conjugator for~$P$, then so is any $\phi\dominatedby\phi_0$.

Under the standing assumptions of this section, every $\d$-polynomial has a semi-clean conjugator, and if $K$ is also $\d$-valued, then every $\d$-polynomial has a clean conjugator.  We will prove that these are definable, and describe how to find them.
Fix $d,W\in\N$.  

\begin{thm}\label{cleanconjugationthm}
There exists a quantifier-free-definable conjugation-invariant map $$\gamma=(\gamma_1,\ldots,\gamma_n)\colon K\{Y\}_{\leq d,[\leq W]}^{\neq}\to\Gamma^n \qquad (n\leq 7W+3)$$
with $\operatorname{range}(\gamma_n)\subseteq\Psi^{\downarrow}$, such that if $P\in  K\{Y\}_{\leq d,[\leq W]}^{\neq}$, $v\phi_n=\gamma_n(P)$, then $\phi_n$ is a semi-clean conjugator for $P$.  If $K$ is $\d$-valued, then we can ensure that $\phi_n$ is a clean conjugator for $P$.
\end{thm}

\noindent
Recall from subsection~\ref{subsec:definable on value group} the notion of a conjugation-invariant map $\eta\colon K^n\to\Gamma$, the corresponding map $\widehat{\eta}\colon K^n\times v^{-1}(\Psi^{\downarrow})\to\Gamma$, and the composition of such maps.  The above theorem follows from the more precise

\begin{prop}\label{cleanconjugationthmprecise}
There exists $n\leq 7W+1$ and a conjugation-invariant map 
$$\eta\colon K\{Y\}_{\leq d,[\leq W]}^{\neq}\to\Psi^{\downarrow}$$ such that if $P\in  K\{Y\}_{\leq d,[\leq W]}^{\neq}$, $v\phi_0=\widehat{\eta}^n(P,0)$, then $\phi_0$ is a semi-clean conjugator for $P$.  If $K$ is $\d$-valued, then can also ensure that $\phi_0$ is a clean conjugator for $P$.
Moreover, there is a partition of $K\{Y\}_{\leq d,[\leq W]}$ into finitely many quantifier-free-definable pieces where on each piece, $\eta$ is given by the composition of one of $\psi\circ v$, $s\circ v$, $d_{\upl}$, $\frac12d_{\upo}$ with a $\d$-rational function with rational coefficients.
\end{prop}

\noindent
The proof consists of a series of lemmas.  Each lemma describes how, if $P$ satisfies the first $k$ of the below conditions with $w=\dwm P$, we can find a $\phi$ such that either $\dwm P^\phi<\dwm P$ or $P^\phi$ satisfies the first $k+1$ conditions:
\begin{enumerate}[leftmargin=5em, labelsep=1.5em]
    \item[(A)] $P_{[>w]}\prec^{\flat} P$
    \item[(B)] $\nabla_1(P_{[w]})\prec P$
    \item[(C)] $\nabla_1(P_{[w]})\prec^\flat P$
    \item[(D$_{w-1}$)] $P_{[w-1]}\prec^\flat P$
    \item[(E)] $\nabla_2(P_{[w]})\prec P$
    \item[(F)] $\nabla_2(P_{[w]})\prec^\flat P$
    \item[(G$_{w-2}$)] $P_{[w-2]}\prec^\flat P$
    \item[\vdots\phantom{))}]
    \item[(G$_{0}$)] $P_{[0]}\prec^\flat P$
    \item[(H)] $P$ is in semi-clean form
\end{enumerate}
If we are assuming that $K$ is $\d$-valued, then we replace (H) with
\begin{enumerate}[leftmargin=5em, labelsep=1.5em]
\item[(H$'$)]$P$ is in clean form
\end{enumerate}
The maps $\eta_{X}\colon P\mapsto v\phi$ obtained in each case $X=\textup{A},\textup{B},\textup{C},\dots,\textup{H}$ will be con\-ju\-ga\-tion-invariant and be given by a composition of $d_{\upl}$ or $\frac12d_{\upo}$ with a $\d$-rational function.  Composing these maps yields the theorem.

Note that if $\phi\dominatedby1$, $\dwm P^\phi=\dwm P$, and $P$ satisfies the first $k$ of the above conditions for $k\not\in\{2,5\}$, then the same is true of $P^\phi$ by Lemmas~\ref{magicformulas}, \ref{nabladecomposition}, and \ref{flatpreservation}.  This is in particular true for $\phi\asymp1$, in which case $\dwm P^\phi=\dwm P$ is guaranteed.

\emph{For the remainder of this section, $w:=\dwm P$}.

We call a partial function $f\colon K^m\rightharpoonup K$ \emph{piecewise $\d$-rational} if its domain can be partitioned into finitely many quantifier-free $\L_{\textrm{f}}$-definable sets such that $f$ is given by a $\d$-rational function on each.  Our first lemma lets us conjugate to ensure condition~(A).

\begin{lemma}\label{highweightremoval}
If $v\phi\geq s0$, then $(P_{[>w]})^{\phi}\prec^{\flat}_{\phi} P^{\phi}$ and  $(P^{\phi})_{[>w]}\prec^{\flat}_{\phi} P^{\phi}$.  The constant map~$\eta_{\textup{A}}\colon P\mapsto s0$ is conjugation-invariant and is the composition of $d_{\upl}$ and a piecewise $\d$-rational function.
\end{lemma}
\begin{proof}
    Since   $(P^\phi)_{[>w]}=((P_{[>w]})^\phi)_{[>w]}$, it is enough to prove the statement concerning~$(P_{[>w]})^\phi$, and by Lemma~\ref{flatpreservation} it suffices to prove this statement only for~$v\phi=s0$.  Let $\phi\prec 1$.  By Lemma~\ref{compconjvalbound}, 
    \[
    v\big((P_{[>w]})^\phi\big)\geq vP+(w+1)\min\{v\phi,\psi(v\phi)\}\qquad \text{and} \qquad v\big((P_{[\leq w]})^\phi\big)\leq vP+wv\phi.
    \]
    Taking $\phi$ with $v\phi=s0$, we obtain $v\big((P_{[>w]})^{\phi}\big)\geq v\big((P_{[\leq w]})^{\phi}\big)+v\phi$; since $\phi\prec^{\flat}_{\phi}1$, this implies $(P_{[>w]})^{\phi}\prec^{\flat}_{\phi}(P_{[\leq w]})^{\phi}$.
    
    The final statement follows from Lemma~\ref{drelations}(\ref{sfromdl}) and Lemma~\ref{definableconjinvariantlemma}.
\end{proof}

\noindent
Much the same proof gives another result which we will find useful:

\begin{lemma}\label{highweightreduction}
    Suppose $\phi\prec 1$.  Then $(P_{[>w]})^\phi\prec P^\phi$ and $(P^\phi)_{[>w]}\prec P^\phi$.
\end{lemma}
\begin{proof}
    Since $\wt(P_{[\leq w]})^{\phi}\leq w$, we have $(P^\phi)_{[>w]}=((P_{[>w]})^\phi)_{[>w]}$.  If $\phi^\dagger\dominates\phi$, then~$v\phi\geq s0$, and the previous lemma applies.  If $\phi^{\dagger}\prec\phi$, then $$(w+1)\min\{v\phi,\psi(v\phi)\}=(w+1)v\phi>wv\phi.$$ Thus $(P_{[>w]})^\phi\prec (P_{[\leq w]})^\phi+(P_{[>w]})^\phi=P^\phi$.
\end{proof}

\begin{lemma}\label{nabla1reduction}
    Suppose $P$ satisfies \textup{(A)} but not \textup{(B)}, that is, $P_{[>w]}\prec^{\flat}P \asymp \nabla_1(P_{[w]})$.  Let $\bm{i}$ be lexicographically minimal such that $a:=(\nabla_1P_{[w]})_{\bm{i}}\asymp P$, and  set~${b:=P_{\bm{i}}}$.  Then~$\dwm P^\phi<w$ for any $\phi$ with $v\phi\geq d_{\upl}(-b/a)$.  Moreover, the map~${P\mapsto -b/a}$ is piecewise $\d$-rational, and $\eta_{\textup{B}}\colon P\mapsto d_{\upl}(-b/a)$ is conjugation-invariant.
\end{lemma}
\begin{proof}
    If $w=0$, then $P_{[w]}\in K[Y]$ and $\nabla_1(P_{[w]})=0$, so $w\geq 1$.  Let $Q:=P_{[\leq w]}$, so~$P^\phi\sim^\flat Q^\phi$ for all $\phi\dominatedby1$ by Lemma~\ref{flatpreservation}.  Lemma~\ref{magicformulas} yields
    \[
    (Q^\phi)_{\bm{i}}\ =\ \phi^{w-1}(Q_{[w-1]}-\phi^\dagger\nabla_1(Q_{[w]}))_{\bm{i}}\ =\ \phi^{w-1}(b-\phi^\dagger a).
    \]
    For any $\phi$ with $v\phi\geq d_{\upl}(-b/a)$, by Lemma~\ref{drelations}(\ref{dlrepresentatives}),
    \[
    \begin{split}
        v\big((Q^\phi)_{[w-1]}\big)\ &\leq\ v(Q^\phi)_{\bm{i}}\\
        &=\ (w-1)v\phi+va+d_{\upl}(-b/a)\\
        &\leq\ vP+wv\phi\\
        &=\ v\big((Q^\phi)_{[w]}\big).
    \end{split}
    \]
    Thus $\dwm(Q^\phi)\neq w$, and since $$d_{\upl}(-b/a)\geq\min\{v(-b/a),d_{\upl}(0)\}\geq0,$$ we have $\dwm(P^\phi)<w$.

    Now suppose $\theta\asymp1$.  Then   $P_{[>w]}\prec^\flat P$ yields $\dwm P^\theta=w$, $(P^\theta)_{[w]}\sim^\flat \theta^wP_{[w]}$, and hence $\nabla_1\big((P^\theta)_{[w]}\big)\sim^\flat\theta^w\nabla_1(P_{[w]})$, and $$(P^\theta)_{[w-1]}-\theta^{w-1}\big(P_{[w-1]}-\theta^\dagger\nabla_1(P_{[w]})\big)\prec^\flat P.$$  Thus $a^\theta\sim^\flat\theta^wa$ and $$b^\theta-\theta^{w-1}(b-\theta^\dagger a)\prec^\flat P\asymp a,$$ so $$-b^\theta/a^\theta-\theta^{-1}(-b/a+\theta^\dagger)\prec^\flat1$$ and $d_{\upl}^\theta(-b^\theta/a^\theta)=d_{\upl}(-b/a)$ by Lemma~\ref{dbasics}(\ref{dinvariance},\ref{dconjugation}).  Hence the map $P\mapsto d_{\upl}(-b/a)$ is conjugation-invariant.
\end{proof}

\begin{lemma}\label{nabla1clearing}
    Suppose $P$ satisfies conditions~\textup{(A)--(B)} but not \textup{(C)}, i.e., $$P_{[>w]}\prec^{\flat}P, \quad \nabla_1(P_{[w]})\prec P,\quad \nabla_1(P_{[w]})\asymp^\flat P_{[w]}.$$  Let $\bm{i}$, $\bm{j}$ be lexicographically minimal such that $a:=\nabla_1(P_{[w]})_{\bm{i}}\asymp\nabla_1P_{[w]}$ and~$b:= P_{\bm{j}}\asymp P$, hence   $\phi_0:=(a/b)^{\dagger}\prec 1$ is active.  Then either $\dwm P^{\phi_0}<w$ or~$P^{\phi_0}$ satisfies conditions~\textup{(A)--(C)}.  Moreover, the map $P\mapsto\phi_0$ is piecewise $\d$-rational and $\eta_{\textup{C}}\colon P\mapsto s(v\phi_0)$ is conjugation-invariant.
\end{lemma}
\begin{proof}
    From $P_{[>w]}\prec^\flat P$ we get $(P^{\phi_0})_{[w]}\sim^\flat\phi_0^wP_{[w]}$,   thus $$\nabla_1\big((P^{\phi_0})_{[w]}\big)\sim^\flat\phi_0^w\nabla_1(P_{[w]})$$   and hence $$v\big(\nabla_1\big((P^{\phi_0})_{[w]}\big)\big)-v\big((P^{\phi_0})_{[w]}\big)=v(a/b).$$ Since $v\phi_0=\psi(v(a/b))$, either $\dwm P^{\phi_0}<w$, or $P^{\phi_0}$ satisfies (C),   hence also~(B).
    
    For conjugation-invariance, we note that for $\theta\asymp1$ we have $\dwm P^\theta=w$, ${(P^\theta)_{[w]}\sim^{\flat} \theta^wP_{[w]}}$, $\nabla_1\big((P^\theta)_{[w]}\big)\sim^\flat\theta^w\nabla_1P_{[w]}$, $\bm{i}^\theta=\bm{i}$, $\bm{j}^\theta=\bm{j}$, $a^\theta\sim^\flat\theta^wa$, $b^\theta\sim^\flat\theta^wb$, and~${\phi_0^\theta\sim^\flat\phi_0}$.
\end{proof}

\noindent
The methods of arranging (D) and (G$_i$) are the same, so we combine them in a single lemma:

\begin{lemma}\label{lowerweightremoval}
    Let $i\in\{0,\dots,w-1\}$ and suppose that 
    $$P_{[i]}\asymp^\flat P,\quad \nabla_1(P_{[j]})\prec^\flat P\text{ for all $j>i$,}\quad\text{ and }\quad\nabla_2(P_{[j]})\prec^\flat P\text{ for all $j>i+1$.}$$  Let $\alpha:=\frac{1}{w-i}(vP_{[i]}-vP_{[w]})\in\Q\Gamma$.  Then either
    \begin{enumerate}[label=\textup{(\theenumi)}, ref=\theenumi]
        \item $\alpha\in\Psi^{\downarrow}$, and $\dwm P^\phi<w$ whenever $v\phi\geq\alpha$, in particular for $v\phi=s\alpha$; or
        \item $\alpha>\Psi$, and if $v\phi\geq s\alpha$, then $(P^\phi)_{[i]}\prec^{\flat}_{\phi}P^{\phi}$.  In particular, if $P$ satisfies conditions~\textup{(A)--(C)} but not \textup{(D)}, then $P^{\phi}$ satisfies \textup{(A)--(D)}, and if $P$ satisfies conditions~\textup{(A)--(G$_{i+1}$)} but not \textup{(G$_i$)}, then $P^{\phi}$ satisfies \textup{(A)--(G$_i$)}.
    \end{enumerate}
    Moreover, for a fixed $i$ the map $\eta_i\colon P\mapsto s(\alpha)$ is conjugation-invariant and is the composition of $d_{\upl}$ and a piecewise $\d$-rational function.
\end{lemma}
\begin{proof}
    For $j>i+1$, since $\nabla_1(P_{[j]}), \nabla_2(P_{[j]})\prec^\flat P$, we have $\big((P_{[j]})^\phi\big)_{[i]}\prec^\flat P$ by Lem\-mas~\ref{nabladecomposition} and \ref{flatpreservation}.  Since also $\nabla_1P_{[i+1]}\prec^\flat P$, this implies~${(P^\phi)_{[i]}-\phi^iP_{[i]}\prec^\flat P}$, and since $P_{[i]}\asymp^\flat P$, these yield $(P^\phi)_{[i]}\sim^\flat \phi^iP_{[i]}$, in addition to the usual relation~$(P^\phi)_{[w]}\sim^\flat \phi^wP_{[w]}$.

    If $\alpha\in\Psi^{\downarrow}$, then the above gives $(P^\phi)_{[i]}\dominates (P^\phi)_{[w]}$ when $v\phi\geq\alpha$, which shows~(1).  Now assume $\alpha>\Psi$, and let $\phi$ be such that $v\phi=s\alpha$.  Since $P_{[i]}\asymp P$, by Lem\-ma~\ref{sbasics}(\ref{sflatcharacterization}) we have $v\phi>0$.  Then
    \[
    \begin{split}
        v\big((P^{\phi})_{[i]}\big)\ &=\ vP_{[i]}+iv\phi\\
        &=\ vP_{[i]}+iv\phi+v\big((P^{\phi})_{[w]}\big)-vP_{[w]}-wv\phi\\
        &=\ v\big((P^{\phi})_{[w]}\big)+(w-i)(\alpha-v\phi)\\
        &\geq\ v(P^{\phi})+(w-i)(\alpha-s\alpha),
    \end{split}
    \]
    and $\psi\big((w-i)(\alpha-s\alpha)\big)=\psi(\alpha-s\alpha)=s\alpha=v\phi$ by Lemma~\ref{sbasics}(\ref{ssqueezing}), so $(P^{\phi})_{[i]}\prec^{\flat}_{\phi}P^{\phi}$.

    Suppose $\theta\asymp1$.  Then $\dwm P^\theta=w$, $(P^\theta)_{[i]}\sim^\flat\theta^i P_{[i]}$ and $(P^\theta)_{[w]}\sim^\flat\theta^w P_{[w]}$, so $\alpha^\theta=\frac{1}{w-i}(vP_{[i]}-vP_{[w]})=\alpha$ and $s^\theta(\alpha^\theta)=s(\alpha)$.  Additionally, for any $j>i$, $(P_{[j]})^\theta-\theta^jP_{[j]}\prec^\flat P^\theta$, so $P^\theta$ satisfies the hypotheses of the lemma.  Thus $P\mapsto s(\alpha)$ is conjugation-invariant.

    By Lemma~\ref{drelations}(\ref{sfromdl}) we have $s(\alpha)=d_{\upl}\big({-\frac{1}{w-i}(P_{\bm{i}}/P_{\bm{j}})^\dagger}\big)$, where $\bm{i}$ and $\bm{j}$ are lexicographically minimal such that $\|\bm{i}\|=i$, $P_{\bm{i}}\asymp P_{[i]}$,   $\|\bm{j}\|=w$, and $P_{\bm{j}}\asymp P_{[w]}$, which yields the piecewise $\d$-rationality.
\end{proof}

\begin{lemma}\label{nabla2reduction}
    Suppose that $P$ satisfies conditions~\textup{(A)--(D)} but not \textup{(E)}, i.e., 
    $$P_{[>w]}\prec^{\flat} P,\quad \nabla_1(P_{[w]})\prec^\flat P,\quad P_{[w-1]}\prec^\flat P,\quad\text{ and }\quad \nabla_2(P_{[w]})\asymp P.$$  Let $\bm{i}$ be lexicographically minimal such that $a:=(\nabla_2P_{[w]})_{\bm{i}}\asymp P$, and let~${b:=P_{\bm{i}}}$.  Then $\dwm P^\phi\leq w-2$  if $v\phi\geq\frac12d_{\upo}(-b/a)$.  Moreover, the map~${P\mapsto-b/a}$ is piecewise $\d$-rational, and $\eta_{\textup{E}}\colon P\mapsto \frac12d_{\upo}(-b/a)$ is conjugation-invariant.
\end{lemma}
\begin{proof}
    If $w\leq 1$, then $P_{[w]}\in K[Y](Y')^w$ and $\nabla_2(P_{[w]})=0$, so $w\geq2$.  Let $Q:=P_{[\leq w]}-P_{[w-1]}$, so $P^\phi\sim^\flat Q^\phi$ for all $\phi\dominatedby1$ by Lemma~\ref{flatpreservation}.  Lemma~\ref{magicformulas} yields
    \[
    \begin{split}
        (Q^\phi)_{\bm{i}}\ &=\ \phi^{w-2}\big(Q_{[w-2]}+\omega(-\phi^\dagger)\nabla_2Q_{[w]}+\frac12(\phi^\dagger)^2\nabla_1^2Q_{[w]}\big)_{\bm{i}}\\
        &=\ \phi^{w-2}(b+\omega(-\phi^\dagger)a)+c,
    \end{split}
    \]
    where $c\prec^\flat P$.  By Lemma~\ref{dbasics}(\ref{dorepresentatives}), if $v\phi\geq\frac12d_{\upo}(-b/a)\geq0$, then
    \[
    \begin{split}
        v\big(\phi^{w-2}(b+\omega(-\phi^\dagger)a)\big)\ &=\ (w-2)v\phi+va+d_{\upo}(-b/a)\\
        &\leq\ vP+wv\phi\\
        &=\ v\big((Q^\phi)_{[w]}\big).
    \end{split}
    \]
    Since $vc>vP+wv\phi$, we have $(Q^\phi)_{[w-2]}\dominates (Q^\phi)_{[w]}$, and hence $\dwm(Q^\phi)\neq w$.  Since $d_{\upo}(-b/a)\geq\min\{d_{\upo}(0),v(-b/a)\}\geq0$ and $P_{[w-1]},\nabla_1P_{[w]}\prec^\flat P$, we obtain~$\dwm P^\phi=\dwm Q^\phi\leq w-2$.

    The piecewise $\d$-rationality is clear.  Suppose $\theta\asymp1$.  Then $\dwm P^\theta=w$, $(P^\theta)_{[w]}\sim^\flat \theta^wP_{[w]}$, and hence $\nabla_2\big((P^\theta)_{[w]}\big)\sim^{\flat}\theta^w\nabla_2(P_{[w]})$, and $$(P^\theta)_{[w-2]}-\theta^{w-2}(P_{[w-2]}+\omega(-\theta^\dagger)\nabla_2P_{[w]})\prec^\flat P.$$  Thus 
    $$a^\theta\sim^\flat \theta^wa,\qquad b^\theta-\theta^{w-2}(b+\omega(-\theta^\dagger)a)\prec^\flat P\asymp a,$$ so 
    $$-b^\theta/a^\theta-\theta^{-2}(-b/a+\omega(-\theta^\dagger))\prec^\flat1,\qquad d_{\upo}^\theta(-b^\theta/a^\theta)=d_{\upo}(-b/a)$$ by Lemma~\ref{dbasics}(\ref{dinvariance},\ref{dconjugation}).
\end{proof}

\begin{lemma}\label{nabla2clearing}
    Suppose that $P$ satisfies conditions~\textup{(A)--(E)} but not \textup{(F)}, i.e., 
    $$P_{[>w]}\prec^{\flat} P,\quad \nabla_1(P_{[w]})\prec^\flat P,\quad P_{[w-1]}\prec^\flat P,\quad \nabla_2(P_{[w]})\prec P, \quad \nabla_2(P_{[w]})\asymp^{\flat}P.$$  Let $\bm{i}$, $\bm{j}$ be lexicographically minimal such that $a:=(\nabla_2P_{[w]})_{\bm{i}}\asymp\nabla_2P_{[w]}$ and~$b:=P_{\bm{j}}\asymp P$. Then   $\phi_0:=(a/b)^\dagger\prec 1$ is active, and either $\dwm P^{\phi_0}<w$ or $P^{\phi_0}$ satisfies \textup{(A)--(F)}.  Moreover, $P\mapsto \phi_0$ is piecewise $\d$-rational and $\eta_{\textup{F}}\colon P\mapsto s(v\phi_0)$ is conjugation-invariant.
\end{lemma}
\begin{proof}
    The proof is identical to that of Lemma~\ref{nabla1clearing}, with $\nabla_2$ in place of $\nabla_1$.
\end{proof}

\begin{lemma}\label{nondominantclearing}
    Suppose that $P$ satisfies conditions~\textup{(A)--(G$_0$)}:
    $$P\sim^\flat P_{[w]},\quad \nabla_1(P_{[w]})\prec^\flat P, \quad \nabla_2(P_{[w]})\prec^\flat P.$$  Then there is a piecewise $\d$-rational function $f$ such that $\eta_{\textup{H}}:=d_{\upl}\circ f$ is conjugation-invariant and if $v\phi\geq\eta_{\textup{H}}(P)$, then~$P^\phi$ satisfies~\textup{(A)--(H)}.
\end{lemma}
\begin{proof}
    Let
    \[
    D\ :=\ \sum_{\|\bm{i}\|=i_1=w}P_{\bm{i}}Y^{\bm{i}}\ \in\ K[Y](Y')^w.
    \]
    By our assumptions on $P$ and Lemma~\ref{nabladecomposition}, $P\sim^\flat D$ and $\dd_P=D_{\bm{i}}$ for the lexicographically minimal $\bm{i}$ such that $D_{\bm{i}}\asymp D$.  For any $\phi$, we have $D^\phi=\phi^w D$.

    Let $\phi$ be such that $$v\phi\geq\eta_H(P):=\max\big(\{0\}\cup\{\psi(v(D_{\bm{j}}/\dd_P))\, :\, D_{\bm{j}}\neq 0,D_{\bm{j}}\not\asymp P\}\big).$$  Then for all $\bm{j}$ with $D_{\bm{j}}\prec P$, we have $D_j\prec^{\flat}_{\phi}P$, and hence $(P^\phi)_{\bm{j}}\prec^{\flat}_\phi P^\phi$.

    Let $\bm{j}$ be lexicographically minimal such that $a:=D_{\bm{j}}/\dd_P$ has minimal positive valuation; then by Lemmas~\ref{sbasics}(\ref{sshape}) and~\ref{drelations}(\ref{sfromdl},\ref{dlrepresentatives}), we have $\eta_{\textup{H}}(P)=\psi(va)=d_{\upl}(-(a')^\dagger)$ if~$a=0$ or~$a^\dagger\prec1$ and $\eta_{\textup{H}}(P)=0=d_{\upl}(1)$ otherwise.  Conjugation-invariance follows from~$D^\theta=\theta^w D$.
\end{proof}

\begin{lemma}\label{dominantpartextraction}
    Suppose that $K$ is $\d$-valued and $P$ satisfies conditions~\textup{(A)--(G$_0$)}:
    $$P\sim^\flat P_{[w]},\quad \nabla_1(P_{[w]})\prec^\flat P, \quad \nabla_2(P_{[w]})\prec^\flat P.$$  Then there is a piecewise $\d$-rational function $f$ such that $\eta_{\textup{H}'}:=d_{\upl}\circ f$ is conjugation-invariant and if $v\phi\geq\eta_{\textup{H}'}(P)$, then~$P^\phi$ satisfies~\textup{(A)--(H$'$)}.
\end{lemma}
\begin{proof}
    Let $D,\dd_P$ be as before.

    For any $\bm{j}$ such that $D_j\neq0$,  if $D_j/\dd_P\in C+\varepsilon$ for some $0\neq\varepsilon\prec1$, then
    \[
    \psi(v\varepsilon)\ =\ s(v\varepsilon')\ =\ s\big(v(D_j/\dd_P)'\big).
    \]
    Let $\phi$ be such that $$v\phi\geq\eta_H(P):=\max\big(\{0\}\cup\{s(v(D_{\bm{j}}/\dd_P)')\, :\, D_{\bm{j}}\neq 0\}\big).$$  Then for all $\bm{j}$, we have $D_j/\dd_P\in C+\varepsilon$ for some $\varepsilon\prec^{\flat}_{\phi}1$, i.e., $D/\dd_P\sim^{\flat}_{\phi}D_D=D_P$.  Hence $P^\phi\sim^{\flat}_\phi D^\phi\sim^{\flat}_\phi\phi^w\dd_PD_P=\dd_{P^\phi}D_{P^\phi}$, as required.

    Let $\bm{j}$ be lexicographically minimal such that $a:=(D_{\bm{j}}/\dd_P)'$ has minimal valuation; then by Lemmas~\ref{sbasics}(\ref{sshape}) and~\ref{drelations}(\ref{sfromdl},\ref{dlrepresentatives}), we have $\eta_{\textup{H}'}(P)=d_{\upl}(-a^\dagger)$ if~$a^\dagger\prec1$ and $\eta_{\textup{H}'}(P)=d_{\upl}(1)$ otherwise.  Conjugation-invariance follows from~$D^\theta=\theta^w D$.
\end{proof}

\begin{proof}[Proof of Proposition~\ref{cleanconjugationthmprecise}]
For $P\in K\{Y\}_{\leq d,[\leq W]}$ we define~$\eta(P)$  to be $\eta_i(P)$ if the first condition which $P$ does not satisfy is (D$_i$) or~(G$_i$), $\eta_\textup{X}(P)$ if the first condition which $P$ does not satisfy is ($\textrm{X}$) for some~$\textup{X}\in\{\textup{A},\textup{B},\textup{C},\textup{E},\textup{F}\}$, and $\eta_{\textup{H}}(P)$ if all of (A)--(G$_0$) are satisfied.  If we are assuming that $K$ is $\d$-valued, then we use $\eta_{\textup{H}'}(P)$ rather than $\eta_{\textup{H}}(P)$.

The conditions (A)--(G$_0$) are quantifier-free $\L_{\textrm{f}}$-definable, and for each~$\textup{X}$, $\eta_{\textup{X}}=\frac12d_{\upo}\circ f_{\textup{X}}$ for some piecewise $\d$-rational~$f_X$ (using Lemma~\ref{drelations}(\ref{sfromdl},\ref{dlfromdo})), so $\eta=\frac12d_{\upo}\circ f$ for another such $f$.  Additionally, by Lemma~\ref{conjinvbasics}, $\eta$ is conjugation-invariant.

If $v\phi=\eta(P)$, then either $\dwm P^\phi\leq\dwm P-1$, or $P^\phi$ satisfies a longer initial segment of the conditions than $P$.  Additionally, if $P$ satisfies (A)--(G$_{i+1}$) but not~(G$_i$), then $\dwm P^\phi\leq \dwm P-1$ implies $\dwm P^\phi\leq i$, since $(P^\phi)_{[j]}\prec P^\phi$ for~$i<j<W$ by Lemmas~\ref{flatpreservation} and \ref{nabladecomposition}.

Thus, when constructing the sequence $P,\widehat{\eta}(P,1),\widehat{\eta}^2(P,1),\ldots$, the cases corresponding to (A)--(F) can each occur at most $W$ times, while the case corresponding to a given (G$_i$) can occur at most once, so for $k\geq7W$, $\widehat{\eta}^k(P,1)$ satisfies (A)--(G$_0$).  Thus if $v\phi_0\geq\widehat{\eta}^{7W+1}(P,1)$, then $P^{\phi_0}$ satisfies all of (A)--(H), and (H$'$) if we assume~$K$ to be $\d$-valued.
\end{proof}

\begin{proof}[Proof of Theorem~\ref{cleanconjugationthm} from Proposition~\ref{cleanconjugationthmprecise}]
Let $\eta$ be as in Proposition~\ref{cleanconjugationthmprecise}.  Then $$\zeta \colon K\{Y\}_{\leq d,[\leq W]}^{\neq}\times K^{\succ }\to\Gamma^<, \quad \zeta(P,\theta):=\textstyle\int\widehat{\eta}(P,\theta')$$ is conjugation-invariant and quantifier-free definable.  For $P\in  K\{Y\}_{\leq d,[\leq W]}^{\neq}$ set $\gamma_1(P):=-s0$
$$\gamma_k(P):=\zeta^{k-1}(P,-s0)\quad (k=2,\ldots,n+1), \qquad \gamma_{n+2}(P)=\gamma_{n+1}(P)'.$$  Then 
$(\gamma_1,\ldots,\gamma_{n+2})$ is quantifier-free definable and satisfies $\gamma_{n+2}(P)=\eta^n(P,0)$ for each~$P\in  K\{Y\}_{\leq d,[\leq W]}$.
\end{proof}

\begin{remark}
    A slightly more careful analysis of the conditions (A)--(H) would show that we can take $n\leq 4W+2$.  We could also directly make use of the $\phi_0$ constructed in Lemmas~\ref{nabla1clearing} and \ref{nabla2clearing}, which would complicate the definition of~$\eta$ but bring the bound down to $n\leq3W+1$.
\end{remark}

\begin{remark}
    In the above proof, we make use of $\frac12d_{\upo}$, which requires the assumption $2\Gamma=\Gamma$.  However, we only need an upper bound for $\frac12d_{\upo}$ which lies in $\Psi^{\downarrow}$.  We may use $s\circ\frac12d_{\upo}$; from Lemma~\ref{drelations}(\ref{sfromdl},\ref{ddeffromA}), this yields a version of Theorem~\ref{cleanconjugationthmprecise} without the assumption $2\Gamma=\Gamma$.
\end{remark}

\noindent
Since $\nabla_2(Y'')=\nabla_2(Y')=0$, if $P$ has order at most $2$, then conditions (E) and~(F) are always satisfied.  In particular, we don't use Lemma~\ref{nabla2reduction}, which is the only time that we make use of $d_{\upo}$ and the $\upo$-freeness of $K$ or of $2\Gamma=\Gamma$.  Thus, suspending our standing hypothesis about $K$ from the beginning of the section, we also have:

\begin{cor}\label{cleanconjugationlambdafree}
    Let $K$ be a $\upl$-free asymptotic field of $H$-type, and fix~$d\in\N$.  Then there exists $n\leq 14d+2$ and a conjugation-invariant map $$\eta\colon K[Y,Y',Y'']_{\leq d}^{\neq}\to\Psi^{\downarrow}$$ such that if $P\in  K[Y,Y',Y'']_{\leq d}^{\neq}$ and $v\phi_0=\widehat{\eta}^n(P,0)$, then $\phi_0$ is a semi-clean conjugator for $P$.
    If $K$ is $\d$-valued, then we can also ensure $\phi_0$ is a clean conjugator for~$P$.  Moreover, there is a partition of $K[Y,Y',Y'']^{\neq}_{\leq d}$ into finitely many quantifier-free-definable pieces such that, on each piece, $\eta$ is given by the composition of one of~$\psi\circ v$, $s\circ v$, $d_{\upl}$ with a $\d$-rational function with rational coefficients.
\end{cor}

\noindent
Returning to our standing hypothesis about $K$, Lemma~\ref{vPylemma} immediately gives the following

\begin{cor}\label{vPydefinable}
    There exist $n\leq 7(w+d)+2$, a quantifier-free definable 
    $$\gamma=(\gamma_1,\ldots,\gamma_n)\colon K\{Y\}_{\leq d,[\leq w]}^{\neq}\to\Gamma^n$$ such that $\operatorname{range}(\gamma_n)\subseteq\Psi^{\downarrow}$, and if $v\phi\geq\gamma_n(P)$, then for $y\in K$ we have
    \[
    \begin{aligned}
    1\succ y\asymp^{\flat}_{\phi}1\quad&\implies\quad vP(y)=vP^{\phi}+\dmul(P^{\phi})vy+\dwt(P^{\phi})(\psi(vy)-v\phi),\\
    1\prec y\asymp^{\flat}_{\phi}1\quad&\implies\quad vP(y)=vP^{\phi}+\ddeg(P^{\phi})vy+\dwt(P^\phi)(\psi(vy)-v\phi).
    \end{aligned}
    \]
\end{cor}

\section{Eventual Equalizers}\label{sec:evequ}
\noindent
\emph{In this section, $K$ is an $H$-asymptotic field.}
Our explicit version of the Eventual Equalizer Theorem is the following:

\begin{thm}[Eventual Equalizer Theorem]\label{eventualequalizerthm}
Suppose $K$ is $\upo$-free. Fix~${d,e,w\in\N}$, $d\neq e$.  There exists a quantifier-free $\L_{\textrm{r}}$-definable 
\[
\gamma=(\gamma_1,\ldots,\gamma_n)\colon K\{Y\}_{d,[\leq w]}\times K\{Y\}_{e,[\leq w]}\to\Gamma^n
\] such that if $v\phi_0=\gamma_n(P,Q)$ and $a\neq0$ satisfies $P^{\phi_0}_{\times a}\asymp Q^{\phi_0}_{\times a}$, then $P^{\phi}_{\times a}\asymp Q^{\phi}_{\times a}$ for every active~$\phi\dominatedby\phi_0$.
\end{thm}

\noindent
\emph{Until the end of the proof of Theorem~\ref{eventualequalizerthm}, we assume that $K$ is $\upo$-free.}

We noted in Subsection~\ref{subsec:definable on value group} that any quantifier-free $\L_{\textrm{f}}$-formula is equivalent to a quantifier-free $\L_{\textrm{r}}$-formula, so we will freely make use of multiplicative inversion.

As in the argument from Herbrand's Theorem, $v\widetilde{P}_{\times a}=v\widetilde{Q}_{\times a}$ is described by a quantifier-free $\L_{\textrm{f}}$-formula.  Thus Theorem~\ref{equalizerthm} gives a piecewise $\d$-rational function~$h(P,Q)$ such that if $P_{\times a}\asymp Q_{\times a}$, then $a^{d-e}\asymp h(P,Q)$.

To obtain our $\phi_0$, we will examine the eventual behavior of $h_\phi(P^\phi,Q^\phi)$, where~$h_\phi$ is $h$ with the derivative interpreted as $\upd=\phi^{-1}\der$.  A quantifier-free $\L_{\textrm{f}}$-definable set is given by a boolean combination of equations and dominance relations between differential polynomials over $\Q$, and this remains true after conjugation.  Thus, it suffices to analyze the eventual behavior of valuations $v(P(\phi))$, equations $P(\phi)=0$, and dominance relations $P(\phi)\dominatedby Q(\phi)$.

The first two are addressed by the following lemma:

\begin{lemma}\label{eventualvPphi}
    There exist $n\leq 7(w+d)+3$, a quantifier-free definable 
    $$\gamma=(\gamma_1,\ldots,\gamma_n)\colon K\{Y\}_{\leq d,[\leq w]}^{\neq}\to\Gamma^n$$ such that $\operatorname{range}(\gamma_n)\subseteq\Psi^{\downarrow}$, and if $v\phi_0=\gamma_n(P)$, then for any active $\phi\dominatedby\phi_0$,
    \[
    vP(\phi)\ =\ v(P^{\phi_0}_{\times\phi_0})+\ddeg(P^{\phi_0}_{\times\phi_0})v(\phi/\phi_0).
    \]
    In particular, $P(\phi)\neq0$.
\end{lemma}
\begin{proof}
    Apply Theorem~\ref{cleanconjugationthm} to $Q:=P(Y')\in K\{Y\}_{\leq d,[\leq w+d]}$.  Since the  coefficients of $Q$ are also coefficients of $P$, this does not harm the definability.  Let $\phi_0$ be such that $v\phi_0=\gamma_n(Q)$, and suppose $\phi\dominatedby\phi_0$.  Then there are $D\in K[Y]$, $m\in\N$, and~$R\in K\{Y\}$ with $R\prec^{\flat}_{\phi_0}Q^{\phi_0}$ such that $Q^{\phi_0}=D\cdot(Y')^m+R$.  Since $Q^{\phi_0}\in K\{Y'\}$, we can take $D\in K$, so $m=\ddeg Q^{\phi_0}=\ddeg P^{\phi_0}_{\times\phi_0}$.  By Lemma~\ref{vPphilemma}, we have~$vP(\phi)=vP^{\phi_0}_{\times\phi_0}+mv(\phi/\phi_0)$.
\end{proof}

\noindent
Using Corollary~\ref{cleanconjugationlambdafree} instead of Theorem~\ref{cleanconjugationthmprecise}, and temporarily dropping the assumption that $K$ is $\upo$-free, we obtain:

\begin{cor}\label{eventualvPphilambdafree}
    If $K$ is $\upl$-free, then Lemma~\ref{eventualvPphi} holds with $K[Y,Y']_{\leq d}$ in place of~$K\{Y\}_{\leq d,[\leq w]}$ and $14d+2$ in place of $7(d+w)+2$.
\end{cor}

\noindent
In order to use this to decide whether $P(\phi)\dominatedby Q(\phi)$ eventually, we need to compare terms $a\phi^k$ ($a\in K^\times$, $k\in\Z^{\ne}$) eventually:

\begin{lemma}\label{eventualdominationsimple}
    Let $a\in K^\times$, $k\in\Z^{\ne}$.  Then either $\phi^k\prec a$ for all $\phi$ with~$v\phi\geq s(\frac1kva)$, or $\phi^k\succ a$ for all $\phi$ with $v\phi\geq s(\frac1kva)$.
\end{lemma}
\begin{proof}
    Replacing $k$ by $-k$ and $a$ by $a^{-1}$ if necessary we ensure that $k>0$.  Then either~$va>k\Psi$, in which case $\phi^k\succ a$ for all active $\phi$, or $va\in (k\Psi)^{\downarrow}$, in which case~$v(\phi^k)\geq ks(\frac1kva)>va$ whenever $v\phi\geq s(\frac1kva)$.
\end{proof}

\begin{lemma}\label{eventualdominationpoly}
    There exist $n\leq 14(w+d)+8$ and a quantifier-free definable map 
    $$\gamma=(\gamma_1,\ldots,\gamma_n)\colon \big(K\{Y\}_{\leq d,[\leq w]}^{\neq}\big)^2\to\Gamma^n$$ such that $\operatorname{range}(\gamma_n)\subseteq\Psi^{\downarrow}$ and either $P(\phi)\dominatedby Q(\phi)$ whenever $v\phi\geq\gamma_n(P,Q)$ or~$P(\phi)\succ Q(\phi)$ whenever $v\phi\geq\gamma_n(P,Q)$.
\end{lemma}
\begin{proof}
    Let $m\leq 7(w+d)+3$ and $\gamma_1^1,\ldots,\gamma_{m}^1$ be as in Lemma~\ref{eventualvPphi}, and let $\phi_0$ be such that $v\phi_0=\max\{\gamma_m(P),\gamma_m(Q)\}$.  Then for all active $\phi\dominatedby\phi_0$,
    \[
    vP(\phi)\ =\ vP^{\phi_0}_{\times\phi_0}+(\ddeg P^{\phi_0}_{\times\phi_0})v(\phi/\phi_0),
    \]
    and similarly with $Q$ in place of $P$.  Let $$a:=\dd_{P^{\phi_0}_{\times\phi_0}},\quad b:=\dd_{Q^{\phi_0}_{\times\phi_0}},\quad d:=\ddeg P^{\phi_0}_{\times\phi_0},\quad e:=\ddeg Q^{\phi_0}_{\times\phi_0}.$$  Then for $\phi\dominatedby\phi_0$, we have $P(\phi)\dominatedby Q(\phi)$ iff $\phi^{d-e}\dominatedby \phi_0^{d-e}b/a$.  If $d=e$, then this is independent of $\phi\dominatedby\phi_0$.  If $d\neq e$, then by Lemma~\ref{eventualdominationsimple}, we have either 
    \begin{align*}
    P(\phi)\dominatedby Q(\phi)&\quad\text{whenever $v\phi\geq\max\big\{v\phi_0,s\big(\phi_0+\textstyle\frac{1}{d-e}v(b/a)\big)\big\}$, or} \\ 
    P(\phi)\succ Q(\phi)&\quad\text{whenever $v\phi\geq\max\big\{v\phi_0,s\big(\phi_0+\textstyle\frac{1}{d-e}v(b/a)\big)\big\}$.}
\end{align*}
    Let $n:=2m+2$ and
    \[
    \gamma_i(P,Q)\ :=\ \begin{cases}
    \gamma_i^1(P) & \text{ if $1\leq i\leq m$,}\\
    \gamma_{i-m}^1(Q) & \text{ if $m+1\leq i\leq 2m$,}\\
    v\phi_0 & \text{ if $i\in\{2m+1,2m+2\}$ and $d=e$,}\\
    s\big(\phi_0+\frac{1}{d-e}v(b/a)\big) & \text{ if $i=2m+1$ and $d\neq e$}\\
    \max\big\{v\phi_0,\gamma_{2m+1}(P,Q)\big\} & \text{ if $i=2m+2$ and $d\neq e$.}
    \end{cases}
    \]
    Then $\gamma=(\gamma_1,\ldots,\gamma_n)$ has the required properties.
\end{proof}

\begin{remark}\label{rmk: eventual behavior in extensions}
The conclusions of the lemmas and corollary above also hold for $\phi$ in any asymptotic field extension of $K$ of $H$-type; in particular, if we take an element~$\ell\succ 1$ in an elementary extension of $K$ such that $\upg:=\ell^\dagger$ satisfies $v\upg>\Psi_K$, then they hold for~$\phi=\upg$.  Any first-order property of~$\phi$ which holds eventually in~$K$ will then hold with $\phi=\upg$, as in the original proof of the Eventual Equalizer Theorem.
\end{remark}

\noindent
This is all we need for eventual equalizers.

\begin{proof}[Proof of Eventual Equalizer Theorem]
    There exist $\d$-polynomials 
    $$G_1,\ldots,G_M,H_1,\ldots,H_{2k}\in\Z\{Z_1,\ldots,Z_m\}$$ and a partition $A_1,\ldots,A_k$ of $K\{Y\}_{d,[\leq w]}\times K\{Y\}_{e,[\leq w]}$, such that each $A_i$ is a boolean combination of sets of the form $G_i(P,Q)\dominatedby G_j(P,Q)$ or $G_i(P,Q)=0$, and for each $i$, if $(P,Q)\in A_i$, then $P_{\times a}\asymp Q_{\times a}$ iff $a^{d-e}\asymp H_{2i-1}(P,Q)/H_{2i}(P,Q)$.

    We regard $G_i(P^\phi,Q^\phi)$ and $H_i(P^\phi,Q^\phi)$ as differential polynomials in $\phi$.  If $G_i(P^\phi,Q^\phi)=0\in K\{\phi\}$, which is a quantifier-free condition in $P,Q$, then any conditions involving $G_i$ are trivial and do not depend on $\phi$, and similarly for~$H_i$.  In each instance where the involved $\d$-polynomials are nonzero, we may apply Lemma~\ref{eventualvPphi} to $G_i(P^\phi,Q^\phi)$ and $H_i(P^\phi,Q^\phi)$ and Lemma~\ref{eventualdominationpoly} to the relation $G_i(P^\phi,Q^\phi)\dominatedby G_j(P^\phi,Q^\phi)$ and take a maximum.  We thereby obtain  $$\gamma=(\gamma_1,\ldots,\gamma_n)\colon K\{Y\}_{d,[\leq w]}\times K\{Y\}_{e,[\leq w]}\to\Gamma^n$$ such that $\operatorname{range}(\gamma_n)\subseteq\Psi^{\downarrow}$ and if $v\phi_0=\gamma_n(P,Q)$ then for $\phi\dominatedby\phi_0$ and any $i$,~$j$, whether $G_i(P^\phi,Q^\phi)=0$ and whether $G_i(P^\phi,Q^\phi)\dominatedby G_j(P^\phi,Q^\phi)$ are both independent of $\phi$, and $$vH_i(\phi)\ =\ v\big((H_i)^{\phi_0}_{\times\phi_0}\big)+\ddeg\big((H_i)^{\phi_0}_{\times\phi_0}\big)(v\phi-v\phi_0).$$
    Thus there exist $i_0\leq k$, $\alpha\in\Gamma$, and $l\in\Z$ such that for any active $\phi\dominatedby\phi_0$, we have~$(P^{\phi},Q^{\phi})\in A_{i_0}$ and $v\big(H_{2i_0-1}(P^\phi,Q^\phi)/H_{2i_0}(P^\phi,Q^\phi)\big)=\alpha+lv\phi$.

    The original Eventual Equalizer Theorem (as extended to non-$\d$-valued fields in~\cite[Corollary 13.6.9]{ADAMTT}) shows that eventual equalizers exist, and hence $l=0$.  Thus for any~$\phi\dominatedby\phi_0$ and any $a\in K^\times$, $vP^\phi_{\times a}=vQ^\phi_{\times a}$ iff $(d-e)va=\alpha$.
\end{proof}

\noindent
In \cite[Proposition 11.6.17]{ADAMTT}, it is shown that the existence of eventual equalizers between differential polynomials of degree and order $1$ and those of degree $0$ is equivalent to $\upl$-freeness.  Below, we extend this to all differential polynomials of order $1$, and prove that existence of eventual equalizers up to order $2$ is equivalent to $\upo$-freeness.
As a warm-up, we re-prove the implication (ii)~$\Rightarrow$~(iii) of \cite[Proposition~11.6.17]{ADAMTT} in our notation.

\begin{prop}
    Assume $K$ is ungrounded.  Let $b\in K^\times$, and let $P:=Y'-bY$, $Q:=1$.  Then $\alpha$ is an eventual equalizer for $P$ and $Q$ iff $\alpha=-d_{\upl}(b)$.  In particular, if $K$ has such eventual equalizers for all $b$, then $K$ is $\upl$-free.
\end{prop}
\begin{proof}
    Let $a\in K^\times$, $c:=a^{-1}$.  Then $Q_{\times a}^\phi=1$ and
    \[
    P_{\times a}^\phi\ =\ a\phi Y'+(a'-ab)Y\ =\ a\phi Y'-a(b+c^\dagger)Y
    \]
    Thus $vP^\phi_{\times a}=va+\min\{v\phi,v(b+c^\dagger)\}$.  Since $vQ_{\times a}^\phi=0$, for $a$ to be an eventual equalizer of $P$ and $Q$, the quantity $vP^\phi_{\times a}$ must be eventually constant.  Therefore~$va+v(b+c^\dagger)=0$ and $vc=v(b+c^\dagger)\leq v\phi$.  By Lemma~\ref{drelations}(\ref{ddeffromA}), this is equivalent to $vc=d_{\upl}(b)$, i.e., $va=-d_{\upl}(b)$.
\end{proof}

\begin{prop}\label{omegafreeiffeventualequalizers}
    Assume $K$ has asymptotic integration.  Let $b\in K$, and let $$P:=4YY''-3(Y')^2+4bY^2,\qquad Q:=Y.$$  Then $\alpha$ is an eventual equalizer for $P$ and $Q$ iff $\alpha=-d_{\upo}(b)$.  In particular, if $K$ has such eventual equalizers for all $b$, then $K$ is $\upo$-free.
\end{prop}
\begin{proof}
    Let $a\in K^\times$, $c:=a^{-1}$.  We have $Q^\phi_{\times a}=aY$ and
    \[
    \begin{split}
    P^\phi_{\times a}\ &=\ a^2\phi^2(4YY''-3(Y')^2)+(4a^2\phi'+2aa'\phi)YY'+(4aa''-3(a')^2+4ba^2)Y^2\\
    &=\ a^2\phi^2(4YY''-3(Y')^2)+4a^2\phi(\phi^\dagger-\textstyle\frac12c^\dagger)+4a^2(b-\omega(-\frac12c^\dagger))Y^2.
    \end{split}
    \]
    Since $K$ has asymptotic integration, $v(\phi^\dagger-\frac12c^\dagger)=s(\frac12vc)$ eventually.  Thus
    \[
    vP^\phi_{\times a}\ =\ 2va+\min\big\{2v\phi,v\phi+s(\textstyle\frac12vc),v\big(b-\omega(-\frac12c^\dagger)\big)\}.
    \]
    Since $vQ^\phi_{\times a}=va$, for $a$ to be an eventual equalizer of $P$ and $Q$, $vP^\phi_{\times a}$ must be eventually constant.  Thus we must have 
    $v\big(b-\omega(-\frac12c^\dagger)\big)=-va=vc$ and $$vc=v\big(b-\omega(-\textstyle\frac12c^\dagger)\big)\leq v\phi+s(\frac12vc)\in (2\Psi)^{\downarrow}.$$  By Lemma~\ref{drelations}(\ref{ddeffromA2}), this happens exactly when $vc=d_{\upo}(b)$, i.e. $va=-d_{\upo}(b)$.
\end{proof}

\begin{remark}
    Let $P(Y):=2YY''-3(Y')^2-bY^2$.  The proof of \cite[Corollary~13.3.14]{ADAMTT} shows that there exists $\phi_n$ as in Theorem~\ref{cleanconjugationthm} for $P(Y')$ iff $d_{\upo}(b)\neq2\upg$, and an appropriate modification of the above proof shows that the eventual equalizer of $P$ and $Y^4$ is $\frac12d_{\upo}(b)$, if that is in $\Gamma$.
\end{remark}

\noindent
To prove the existence of eventual equalizers for $\d$-polynomials of order $1$ over a $\upl$-free field, we cannot directly imitate the proof of the Eventual Equalizer Theorem, since the differential polynomials produced by the Equalizer Theorem may have high order.  Instead, we consider the iteration in the proof of the Equalizer Theorem and examine a single step.

Recall that $\Ri(P^\phi)=\Ri(P)^\phi_{\times \phi}$, and that if $P$ has order at most $1$, then $\Ri(P)$ has order zero, i.e., is a polynomial.

\emph{In the following lemmas, $K$ is assumed to be $\upl$-free, and $P\in K[Y]$.}

\begin{lemma}
    Suppose $d:=\deg P\ge 1$. Then $vP(\phi^\dagger)$ is eventually constant and~$vP(\phi^\dagger)<vP+dv\phi$, eventually.
\end{lemma}
\begin{proof}
    The algebraic closure $K^{\textrm{a}}$ of $K$ is $\upl$-free by \cite[Corollary 11.6.8]{ADAMTT}, so there are~$a,b_1,\ldots,b_d\in K^{\textrm{a}}$ such that $P=a(y+b_1)\cdots(y+b_d)$.  Then 
    $$vP=va+\min\{0,vb_1\}+\cdots+\min\{0,vb_d\}$$ and, eventually,
    \[
    \begin{split}
    vP(\phi^\dagger)\ &=\ v\big(a(\phi^\dagger+b_1)\cdots(\phi^\dagger+b_d)\big)\\
    &=\ va+d_{\upl}(b_1)+\cdots+d_{\upl}(b_d)\\
    &=\ vP+\max\{0,d_{\upl}(b_1)\}+\cdots+\max\{0,d_{\upl}(b_d)\}\\
    &<\ vP+dv\phi.
    \end{split}
    \]
    The third line uses that $d_{\upl}(b)=vb$ when $vb\leq 0$ and $d_{\upl}(b)\geq0$ when $vb\geq0$.
\end{proof}

\begin{lemma}\label{compmultddeg}
    Let $q\in\Q^\times$.  Then $\ddeg P_{+q\phi^\dagger,\times\phi}=0$, eventually, and $vP_{+q\phi^\dagger,\times\phi}$ is eventually constant.
\end{lemma}
\begin{proof}
    Replacing $P$ by $P_{\times q}$ and $\phi$ by $\phi/q$, we may assume $q=1$.  If $i>d:=\ddeg P$ and $j\leq i$, then for~$\phi\dominatedby1$, we have $vP_{\leq d,+\phi^\dagger,\times\phi}\leq vP+dv\phi$ and  
    \[
    vP_{i,+\phi^\dagger,\times\phi,j}\ =\ vP_i+(i-j)v\phi^\dagger+jv\phi\ >\ vP_d+dv\phi,
    \]
    so $\ddeg P_{+\phi^\dagger,\times \phi}\leq d$ and $P_{\leq d,+\phi^\dagger,\times\phi}\asymp P_{+\phi^\dagger,\times\phi}$.  By the previous lemma, if $d\geq 1$, then
    \[
    vP_{\leq d,+\phi^\dagger,\times\phi,0}\ =\ vP_{\leq d}(\phi^\dagger)\ <\ vP_{\leq d}+dv\phi,\quad\text{eventually,}
    \]
    and $\ddeg P_{+\phi^\dagger,\times\phi}<d$, eventually.  Let $\phi_0$ be such that $\ddeg P_{+\phi_0^\dagger,\times\phi_0}$ is minimal.  Replacing $K$ and $P$ by $K^{\phi_0}$ and $P^{\phi_0}$ in the above shows that $\ddeg P_{+\phi^\dagger,\times\phi}=0$, eventually.  Since $P_{+\phi^\dagger,\times\phi,0}=P(\phi^\dagger)$, the previous lemma shows that $vP_{+\phi^\dagger,\times\phi}$ is eventually constant.
\end{proof}

\begin{lemma}\label{lambdafreeonestep}
    Suppose $q\in\Q$, $k\in\Z$, $a\in K^\times$, and let $\upd=\phi^{-1}\der$ be the derivation of $K^\phi$.  Then there exist $\beta\in\Gamma$ and $l\in\Z$ such that
    $$vP^\phi_{\times \phi,+q\upd(\phi^ka)/(\phi^ka)}=\beta+lv\phi,\quad\text{eventually.}$$
\end{lemma}
\begin{proof}
    The coefficients of $$Q:=\phi^{\deg P}P^\phi_{\times \phi,+q\upd(\phi^ka)/(\phi^ka)}$$ are differential polynomials in $\phi$ of order at most $1$.  Thus by Corollary~\ref{eventualvPphilambdafree}, there exist $\beta_0,\ldots,\beta_m\in\Gamma$ and $l_0,\ldots,l_m\in\Z$ such that, eventually,
    \[
    vQ\ =\ \min\{\beta_0+l_0v\phi,\beta_1+l_1v\phi,\ldots,\beta_m+l_mv\phi\}.
    \]
    There is a particular $i$ such that this minimum is eventually equal to $\beta_i+l_iv\phi$; we may take $\beta:=\beta_i$ and $l:=l_i-\deg P$.
\end{proof}

\begin{prop}\label{order1eventualequalizers}
    Assume $K$ is $\upl$-free.  Suppose $P,Q\in K[Y,Y']$ are homogeneous of different degrees.  Then there exists an active $\phi_0$ such that for any $\phi\dominatedby\phi_0$ and any $a\in K^\times$, $vP^\phi_{\times a}=vQ^\phi_{\times a}$ iff $vP^{\phi_0}_{\times a}=vQ^{\phi_0}_{\times a}$.
\end{prop}
\begin{proof}
Let $R:=\Ri(P)$, $S:=\Ri(Q)$.  Let $g_\phi$ be the function defined in Section~3 with the polynomials $\Ri(P^\phi)=R^\phi_{\times \phi}$ and $\Ri(Q^\phi)=S^\phi_{\times \phi}$ in place of $P$ and $Q$ and $K^\phi$ in place of $K$.  By Lemma~\ref{lambdafreeonestep}, we have that for any fixed $\alpha$ and any $k\in\Z$, there exist $\beta\in\Gamma$ and $l\in \Z$ such that, eventually, $g_\phi(\alpha+kv\phi)=\beta+lv\phi$.  Thus, taking $n$ sufficiently large, we obtain  $\alpha\in\Gamma$ and $k\in\Z$ such that, eventually, $g_\phi^n(0)=\alpha+kv\phi$ and, with $q:=(\deg P-\deg Q)^{-1}$:
$$vS^\phi_{\times \phi,+q\upd(a)/a}-vR^\phi_{\times \phi,+q\upd(a)/a}=va \quad\Longleftrightarrow\quad va=\alpha+kv\phi.$$
Let $b$ be such that $vb=\alpha$.  Then $R^\phi_{\times\phi,+q\upd(b\phi^k)/b\phi^k}=R^\phi_{+qb^\dagger+qk\phi^\dagger,\times\phi}$.  Applying Lem\-ma~\ref{compmultddeg} to $R_{+qb^\dagger}$ and using that $R_{+qb^\dagger}^\phi=R_{+qb^\dagger}$ since $R\in K[Y]$, we see that~$vR^\phi_{+qb^\dagger+qk\phi^\dagger,\times\phi}$ is eventually constant.  Similarly, $vS^\phi_{\times \phi,+q\upd(b\phi^k)/b\phi^k}$ is eventually constant.  Thus $v(b\phi^k)=\alpha+kv\phi$ is eventually constant, so $k=0$ and, eventually, $vP^\phi_{\times a}=vQ^\phi_{\times a}$ iff $(\deg P-\deg Q)va=\alpha$.
\end{proof}

\appendix
\section{A Non-definable Bound}\label{sec:upshift lower bound}

\noindent
The material in this appendix up through the proof of Proposition~\ref{prop:bound on comp conj} is modified from material prepared for~\cite{ADAMTT} but not included in the final version, provided by Matthias Aschenbrenner.

\medskip\noindent
\emph{In this appendix, $K$ is $\d$-valued    of $H$-type with asymptotic integration and small derivation, and $P\neq 0$, $\operatorname{order}(P)\leq r$.}  If~$K$ is $\upo$-free, then by \cite[Lemma~13.6.11]{ADAMTT}, $P$ has a semi-clean conjugator~$\phi_*$, and Lemma~\ref{dominantpartextraction} (essentially the third case of~\cite[Lem\-ma~13.2.9]{ADAMTT}, with weaker assumptions) upgrades this to a clean conjugator.
By~\cite[Corollary 11.7.15]{ADAMTT}, if~$K$ is a union of grounded asymptotic subfields, then~$K$ is $\upo$-free, hence such an element $\phi_*$ exists.
 Proposition~\ref{prop:bound on comp conj} below provides an explicit description of such a~$\phi_*$ in this situation, under a small extra hypothesis (assuming  that $K$ is closed under integration is more than enough). This is then applied to~$K=\T$ in Corollary~\ref{cor:newton poly, upward shift}.

\medskip\noindent
We begin with a lemma and its corollary, for which we assume that $x\in K$ satisfies~$x\succ 1$ and $x'=1$, and set $t:=1/x$. Then $t^\dagger=-t\prec 1$, so $t\asymp^\flat 1$ and $t$ is active in $K$.  For words~$\btau\geq\bomega$ in $\{0,\dots,r\}^*$ we define $Y^{[\bomega]},P_{[\bomega]}$ as in~\cite[Section~4.2]{ADAMTT} and~$s(\btau,\bomega)\in\Z$ as in \cite[Section~5.7]{ADAMTT}.  Here~$\btau\geq\bomega$ means that $\btau$ and $\bomega$ have the same length and $\btau_i\geq\bomega_i$ for all $i$.

\begin{lemma}\label{compconjval of constant polynomials}
Suppose  
$P\in C\{Y\}$ and $\mu:=\wm(P)$. Then $P^t = t^{\mu} D + R$, where
\[
D\ :=\  \sum_\bomega \left(\sum_{\substack{\btau\geq\bomega \\ \dabs{\btau}=\mu}}s(\btau,\bomega)P_{[\btau]}\right)Y^{[\bomega]}\in C\{Y\}^{\ne},\qquad R\prec t^\mu,
\]
hence 
$$P^t\asymp(P_{[\mu]})^t\asymp t^\mu,\  D_{P^t}= D_{(P_{[\mu]})^t}\in C\cdot D,\  R_{P^t}=R,\  \dwt P^t=\dwt\, (P_{[\mu]})^t=\mu.$$ Moreover, we have $D_{[\mu]}=P_{[\mu]}$; if $D$ is isobaric, then $D=P_{[\mu]}$; and $D$ is isobaric iff~$P_{[\mu]}\in C[Y](Y')^\mu$.
\end{lemma}
\begin{proof} By the remarks after \cite[Example 5.7.7]{ADAMTT} we have   
\[
R\ =\ P^t-t^{\mu}D\ =\ \sum_\bomega \left(\sum_{\substack{\btau\geq\bomega \\ \dabs{\btau}>\mu}}s(\btau,\bomega)t^{\dabs{\btau}}P_{[\btau]}\right)Y^{[\bomega]},
\]
so $R\prec t^\mu$. Also 
\[
D_{[\mu]}\ =\ \sum_{\dabs{\bomega}\ =\ \mu}s(\bomega,\bomega) P_{[\bomega]} Y^{[\bomega]}\ =\ 
\sum_{\dabs{\bomega}\ =\ \mu} P_{[\bomega]} Y^{[\bomega]}=P_{[\mu]}\ \ne\ 0,
\]
so $t^\mu D\asymp t^\mu$, and thus $P^t\asymp t^\mu D$ and $\dwt(P^t)=\wt(D)=\mu$.  The claims about~$(P_{[\mu]})^t$ follow from the observation that $D$ depends only on $P_{[\mu]}$.  The last claim of the lemma follows from~\cite[Lemma 12.8.2]{ADAMTT}.
\end{proof}

\begin{cor}\label{cor:compconjval of constant polynomials}
Suppose $R_P \prec^\flat P$, and let $\mu:=\dwm(P)$ and 
$Q:=(D_{P_{[\mu]}})^t$. Then~$D_{P^t}=D_{Q}$ and $\dwt(P^t)=\mu$.  If $\dwm(P^t)=\dwt(P^t)$, then $D_{P^t}=D_{P_{[\mu]}}$
\end{cor}
\begin{proof}
We have $P^t = \dd_P(D_P)^t + (R_P)^t$
where $v\big((R_P)^t\big)\geq v(R_P)>v(P)+\Z vt$.
Applying Lemma~\ref{compconjval of constant polynomials} to $D_P$ in the role of $P$ yields
\begin{align*}
v(Q)\ &=\ v\big((D_P)^t\big)\ =\  \wm(D_P)\, vt\ =\ \mu \,vt, \\ \dwt(Q)\ &=\ \dwt\big((D_P)^t\big)\ =\ \wm(D_P)\ =\ \mu,
\end{align*}
as well as $D_{(D_P)^t}=D_Q$ and $(D_Q)_{[\mu]}=(D_P)_{\mu}=D_{P_{[\mu]}}$.  Thus $v\big((R_P)^t\big)>v(P)+v(Q)$, which gives $D_{P^t}=D_{(D_P)^t}=D_Q$, and hence
\[
\dwt(P^t)\ =\ \wt(D_{P^t})\ =\ \wt(D_Q)\ =\ \dwt(Q)\ =\ \mu
\]
and $(D_{P^t})_{[\mu]}=D_{P_{[\mu]}}$, which implies the last claim.
\end{proof}

\noindent
Now assume that $K_0$ is a grounded asymptotic subfield of~$K$, with asymptotic couple $(\Gamma_0,\psi_0)$, such that $P\in K_0\{Y\}$. Also assume that~$\phi_0\dominatedby 1$ is such that~$v(\phi_0)$ is the largest element of $\Psi_0:=\psi_0(\Gamma_0^{\neq})$, and that $(\ell_n)$ is a sequence in~$K$ such that~$\ell_0'=\phi_0$
and $\ell_{n}'=\ell_{n-1}^\dagger$ for~$n\ge 1$. (These assumptions are met in the situation of Corollary~\ref{cor:newton poly, upward shift} below.) Then $\ell_n\succ \ell_{n+1}\succ 1$ for all $n$, and
if the sequence~$(\ell_n)$ is coinitial in $K^{\succ}$, then $(\ell_n)$ is a logarithmic sequence for~$K$ as defined in~\cite[Section~11.5]{ADAMTT}. 
For each $n$, the element~$\phi_n:=\ell_{n}'=\frac{\phi_0}{\ell_0\cdots\ell_{n-1}}$ is active in $K$.

\begin{prop}\label{prop:bound on comp conj} Suppose $\phi_0\prec 1$, and
set $n_0:=\dwm(P)+1$. Then for~${\phi\dominatedby\phi_{n_0}}$,
\[
D_{P^\phi}\ =\ N_P\in C[Y](Y')^\N,\qquad v(P^\phi)\ =\ v^{\ev}(P)+\nwt(P)v\phi,\qquad R_{P^\phi}\ \prec^\flat_\phi\  P^\phi.
\]
\end{prop}

\noindent
By \cite[Corollary 10.3.2(i) \& Proposition 10.5.15]{ADAMTT},  the $\d$-valued subfield $K_0(C)$ of $K$ has the same value group~$\Gamma_0$ as~$K_0$; hence in order to prove Proposition~\ref{prop:bound on comp conj}, we may replace~$K_0$ by~$K_0(C)$ to arrange $C\subseteq K_0$. 
{\it Thus until the end of the proof of Proposition~\ref{prop:bound on comp conj} we assume $C\subseteq K_0$}\/ (and hence every asymptotic subfield of $K$ containing~$K_0$ is $\d$-valued).
In preparation for this proof, we first analyze the asymptotic couple~$(\Gamma_n,\psi_n)$ 
of the $\d$-valued subfield $K_n:=K_0\langle \ell_{n-1}\rangle=K_0(\ell_0,\ell_1,\dots,\ell_{n-1})$ of~$K$,
with corresponding $\Psi$-set $\Psi_n:=\psi(\Gamma_n^{\neq})$, beginning with the case $n=1$:

\begin{lemma}
The element
$\ell:=\ell_0$ of $K$ is transcendental over $K_0$. Moreover,~$\Gamma_1=\Gamma_0\oplus\Z v\ell$ where  
$\Gamma_0^{<}<m\, v\ell<0$ for all $m\geq 1$, and for all $\gamma_0\in\Gamma_0$, $k\in \Z$:  
\begin{align*}\psi(\gamma_0+k\, v\ell) &= \psi(\gamma_0)\ \text{ if $\gamma_0\ne 0$}, \qquad
\psi(k\, v\ell)=v\phi_0-v\ell\ \text{ if $k\neq 0$,}\\
 \Psi_1&=\Psi_0\cup \big\{\psi(v\ell)\big\}, \quad
\max \Psi_1=\psi(v\ell)=v(\phi_1) > \max \Psi_0.\end{align*}
\end{lemma}
\begin{proof}
By compositional conjugation with $\phi_0$ we reduce to the case where~$\phi_0=1$ (so~${\ell'=1}$) without changing the sequence $(\ell_n)$. Then $\psi(v\ell)=-v\ell>0$.
Suppose~$m\, v\ell\leq\gamma_0$ where $\gamma_0\in\Gamma_0^<$ and $m\geq 1$. Since $(\Gamma,\psi)$ is of $H$-type, this gives~$-v\ell=\psi(v\ell)=\psi(m\, v\ell)\leq \psi(\gamma_0)\leq \max\Psi_0=0$, a
contradiction. This yields the claims in the lemma.
\end{proof}

\noindent
Iteration of the previous lemma yields:

\begin{cor}\label{iter}
We have $\psi(v\ell_n)<\psi(v\ell_{n+1})$ for all $n$, and for $n\geq 1$ the  group~$\Gamma_n$   is an internal direct sum 
\[
\Gamma_n\  =\  \Gamma_0\oplus\Z v(\ell_0)\oplus \Z v(\ell_1)\oplus\cdots\oplus \Z v(\ell_{n-1})
\]
of subgroups of $\Gamma$, with
\[
\Psi_n\ =\ \Psi_0\cup \big\{\psi(v\ell_0),\dots, \psi(v\ell_{n-1})\big\}, \qquad 
\max \Psi_n\ =\ \psi\big(v(\ell_{n-1})\big)\ =\ v(\phi_n).
\]
\end{cor}

\noindent
Next we consider the
$\d$-valued field $K^{\phi_n}$, which is of $H$-type with small derivation and has asymptotic couple $(\Gamma,{\psi^{\phi_n}})$. The coefficients of $P^{\phi_n}$ lie in the $\d$-valued subfield $K_n^{\phi_n}$ 
of $K^{\phi_n}$ with $\max\,\Psi_n^{\phi_n}=0$.  Hence in $K^{\phi_n}\{Y\}$ we have
\[
P^{\phi_n}\ =\ \dd_{P^{\phi_n}} D_{P^{\phi_n}}+R_{P^{\phi_n}}\quad\text{where $R_{P^{\phi_n}}\in K_n^{\phi_n}\{Y\}$, $R_{P^{\phi_n}}\prec^\flat_{\phi_n} P^{\phi_n}$. }
\]

\begin{proof}[Proof of Proposition~\ref{prop:bound on comp conj}]
Put $\phi_{-1}:=1$. Then by~\cite[Corollary 11.1.12]{ADAMTT} we have
\[
\dwm(P^{\phi_{n-1}}) \geq \dwt(P^{\phi_n}) \geq \dwm(P^{\phi_n}) \geq \dwt(P^{\phi_{n+1}}) \geq \dwm(P^{\phi_{n+1}}),
\]
so there is $n$ with $1\leq n\leq n_0=\dwm(P)+1$ such that $\dwm(P^{\phi_{n-1}})=\dwt(P^{\phi_{n}})=\dwm(P^{\phi_{n}})$.  Fix such an $n$, and let $\mu:=\dwm(P^{\phi_n})$.  By Corollary~\ref{cor:compconjval of constant polynomials} we have
\[
\big(D_{P^{\phi_{n-1}}}\big)_{[\mu]}\ =\ D_{P^{\phi_n}}\ =\ D_Q\qquad\text{ for $Q:=\Big(\big(D_{P^{\phi_{n-1}}}\big)_{[\mu]}\Big)^{1/\ell_{n-1}}$,}
\]
and by the last part of Lemma~\ref{compconjval of constant polynomials} applied to $\big(D_{P^{\phi_{n-1}}}\big)_{[\mu]}$ in place of $P$, we obtain~$D_{P^{\phi_n}}=\big(D_{P^{\phi_{n-1}}}\big)_{[\mu]}\in C[Y](Y')^\mu$.
Let $\phi\dominatedby \phi_n$.  Then we apply Lemma~\ref{flatpreservation} to~$R_{P^{\phi_n}}$, $P^{\phi_n}$, $1$, $K^{\phi_n}$ in place of $P$, $Q$, $\phi_0$, $K$ and use the last part of Proposition~\ref{nablamagic} to obtain
\[
D_{P^\phi}\ =\ D_{P^{\phi_n}}\ =\ N_P,\qquad v(P^\phi)-\mu\,v\phi\ =\ v(P^{\phi_n})-\mu\,v(\phi_n),\qquad R_{P^\phi}\ \prec^\flat_{\phi}\  P^\phi
\]
as required.
\end{proof}

\begin{remarks*}
The bound $n_0\leq \dwm(P)+1$ in Proposition~\ref{prop:bound on comp conj} can be improved slightly.  For this, from \cite[Section~4.2]{ADAMTT} recall that the \emph{subdegree of $\bm{i}\in\N^{1+r}$} is defined as~$|\bm{i}|':=i_1+i_2+\cdots+i_r=|\bm{i}|-i_0$, and the \emph{subhomogeneous parts} of $P$ are given by~$P_{|d|'}:=\sum_{|\bm{i}|'=d}P_{\bm{i}}$, so that $P=\sum_dP_{|d|'}$.  The \emph{subdegree of $P$} is $\sdeg(P):=\max\{d:P_{|d|'}\neq0\}$ (with $\max\emptyset:=-\infty$ as usual). If $P=P_{|d|'}$ for some $d$, then~$P$ is \emph{subhomogeneous}.  If $P$ is subhomogeneous of subdegree $d$, then so is $P^\theta$ for any $\theta$.

Now $\dwm(P^{\phi_n})\geq\sdeg(P^{\phi_n})=\sdeg(P)$ for all $n$, hence if $P$ is subhomogeneous of subdegree $d$, then we may take $n_0:=\dwm(P)+1-d$ in the above argument.  A $\d$-polynomial of subdegree $0$ is a polynomial, and an examination of the weight~$1$ terms in Lemma~\ref{compconjval of constant polynomials} shows that if $\sdeg(P)\leq 1$, then $\dwm(P^t)\leq1$ and we can take~${n_0:=2}$.  Applying these bounds separately to each subhomogeneous component shows that we may take
\[
n_0\ :=\ \max\big\{2,\max\{\dwm(P_{|d|'})+1-d:d\geq2\}\big\}.
\]
The examples  $P=Y''$ and $P=Y'Y^{(3)}-2(Y'')^2$ demonstrate that this is best possible for $n_0\leq 3$.  However, consideration of the necessary linear relations between coefficients shows that there is no $\d$-polynomial which gives $n_0=4$ and actually requires $4$ upward shifts.  It may be possible to generalize this argument and further strengthen the bound by a careful combinatorial analysis.
\end{remarks*}

\noindent
We can obtain a similar bound for eventual equalizers, by equalizing every time we do a compositional conjugation:
\begin{prop}\label{prop:weight bound for eventual equalizers}
    Suppose $P,Q\in K_0\{Y\}$ are homogeneous of different degrees.  Then with $n_0:=\max\{\wt P,\wt Q\}$ and $\phi_n$ as in Proposition~\ref{prop:bound on comp conj}, for any $\phi\dominatedby\phi_{n_0}$, the equalizer of $P^{\phi}$ and $Q^\phi$ is the eventual equalizer of $P$ and $Q$.
\end{prop}

\noindent
Since we will be alternating compositional conjugations with equalizing, we need to understand equalizers for $\d$-polynomials of appropriate form.  This is simplest in the case of $\d$-polynomials whose degrees differ by $1$:

\begin{lemma}\label{equalizers for shifts}
    Suppose $P$ is homogeneous of degree $d$ and $R_P\prec^\flat P$.  Then $vP^t=vP+(\dwm P)vt$, and if $a^\dagger\dominatedby t$, then $vP^t_{\times a}=vP^t+dva$.  If $Q$ is homogeneous of degree $d+1$, $R_Q\prec^\flat Q$, and $P\asymp Q$, then the equalizer of $P^t$ and $Q^t$ lies in $\Z vt$.
\end{lemma}
\begin{proof}
    By Lemma~\ref{compconjval of constant polynomials} applied to $D_P$ and Lemma~\ref{flatpreservation} applied to $\dd_PD_P$, $R_P$, we have $vP^t=vP+(\dwm P)vt$.  From $\Ri(P^t_{\times a})=a^d\Ri(P^t)_{+t^{-1}a^\dagger}$ and Lemma~\ref{additiveconjval}, we have $vP^t_{\times a}=vP^t+dva$.  Since the above also holds with $Q$, $d+1$ in place of $P$, $d$, we have that if $va=(\dwm P-\dwm Q)vt$ then \[vP^t_{\times a}\ =\ vP+\big((d+1)\dwm P-d\dwm Q\big)vt\ =\ vQ^t_{\times a}.\qedhere\]
\end{proof}

\noindent
Next, we need to understand how a sequence of combined compositional and multiplicative conjugations of the above form can stabilize:

\begin{lemma}\label{stable comp mult conj}
    Suppose $P\in C\{Y\}$ is homogeneous and isobaric, $k\in\Z^{\neq}$, and $\dwm P^t_{\times t^k}=\dwm P$.  Then $\dwm P=0$.
\end{lemma}
\begin{proof}
    Let $a:=t^k$, $d:=\deg P$, $\mu:=\wt P$.  Then $P^t_{\times a}=a^dt^\mu P$.  Since $P$ is homogeneous and isobaric, $Y^{\mu-dk}P(Y^k)$ is homogeneous and isobaric of degree and weight $\mu$ by~\cite[Corollary 4.3.17]{ADAMTT}, hence super-isobaric of super-weight $2\mu$, so~$Q:=Y^{2\mu}(Y^{\dagger})^{\mu-dk}P\big((Y^\dagger)^k\big)$ lies in $K\{Y\}$ and is isobaric.  We have $[Y^{n}P(Y^k)]^t_{\times t}=(tY)^nP^t_{\times t^k}(Y^k)$, and hence
    \[
    \Ri(Q^t)\ =\ \Ri(Q)^t_{\times t}\ =\ (tY)^{\mu-dk}P^t_{\times a}(Y^k)\ =\ a^dt^{2\mu-dk}\Ri(Q).
    \]
    Thus by~\cite[Lemma 12.8.2]{ADAMTT}, we must have $Q\in C[Y](Y')^\N$ and $Y^{\mu-dk}P(Y^k)=\Ri(Q)\in CY^\N$, so $P\in CY^\N$ and $\mu=0$.
\end{proof}

\begin{cor}\label{cor: stable comp mult conj}
Suppose $P$ is homogeneous and $R_P\prec^\flat P$, $k\in\Z^{\neq}$, and $\dwm P^t_{\times t^k}=\dwm P$.  Then $\dwm P=0$.
\end{cor}
\begin{proof}
    Let $d:=\deg P$, $\mu:=\dwm P$, $a:=t^k$.  Applying Lemma~\ref{flatpreservation} to $P,R_P$ and using $t\asymp^\flat1$, we find that $(R_P)^t_{\times a}\prec^\flat P^t_{\times a}$.  By Lemma~\ref{equalizers for shifts}, $((D_P)_{[>\mu]})^t\dominatedby t^{\mu+1}$ and $v\big(\dd_P(D_P)_{[>\mu]}\big)^t_{\times a}\geq vP+(\mu+1+dk)vt>vP^t_{\times a}$.  Writing $P=\dd_P(D_P)_{[\mu]}+\dd_P(D_P)_{[>\mu]}+R_P$ and applying Lemma~\ref{stable comp mult conj} to $(D_P)_{[\mu]}$ yields the result.
\end{proof}

\noindent
Together with our previous results, this gives three cases:

\begin{cor}\label{conjugate and equalize}
    Suppose $P$ is homogeneous of degree $d$, $R_P\prec^\flat P$, and $a\in t^\Z$.  Let $\mu:=\dwm P$.  Then either
    \begin{enumerate}
        \item[\textup{(1)}] $\dwm P^t<\mu$ and $\dwm P^t_{\times a}<\mu$;
        \item[\textup{(2)}] $D_{P_{[\mu]}}\in CY^{d-\mu}(Y')^\mu$, $a=1$, and $\dwm P^t=\dwm P^t_{\times a}=\mu$; or
        \item[\textup{(3)}] $D_{P_{[\mu]}}\in CY^{d-\mu}(Y')^\mu$, $\dwm P^t=\mu$, and $\dwm P^t_{\times a}=0$.
    \end{enumerate}
\end{cor}
\begin{proof}
    By Corollary~\ref{cor:compconjval of constant polynomials} and~\cite[Lemma~12.8.2]{ADAMTT}, we have $\dwm P^t=\mu$ iff $D_{P_{[\mu]}}\in CY^{d-\mu}(Y')^\mu$.  Thus if $a=1$, then (1) or (2) holds, and if $a\neq1$, then either the first part of (1) or the first part of (3) holds.  If $a\neq1$ and the first part of~(3) holds, then so does the rest of (3): we have $t^{-1}a^\dagger\in\Z^{\neq}$ and $\big(Y^m(Y')^n\big){}^t_{\times a}=a^{m+n}t^nY^m(Y'+t^{-1}a^\dagger Y)^n$ and $a^\dagger\in\Z^{\neq} t$, and as in the proof of Corollary~\ref{cor: stable comp mult conj}, $(D_P)_{[>\mu]}$ and $R_P$ do not contribute to $D_{P^t_{\times a}}$.
    
    By Corollary~\ref{cor: stable comp mult conj}, if $a\neq1$ and $\dwm P^t_{\times a}=\mu$, then $\mu=0$ and $D_{P_{[\mu]}}\in CY^d$.  Thus if $a\neq1$, then either the second part of (1) holds, or (3) holds.
\end{proof}

\noindent
Now we just need to iterate.

\begin{proof}[Proof of Proposition~\ref{prop:weight bound for eventual equalizers}]
    As in the Herbrand's Theorem argument at the beginning of Section~\ref{equalizersection},   replacing $P$, $Q$ with $Y^{n_0+1-d/m}P(Y^{1/m})$ and   $Y^{n_0-e/m}Q(Y^{1/m})$, respectively, we may reduce to the case where $\deg P=\deg Q+1$ without changing~$\wt P$ or~$\wt Q$.
    For each $n$, let $\alpha_n$ be the equalizer of $P^{\phi_n}$ and $Q^{\phi_n}$.  Applying Lem\-ma~\ref{equalizers for shifts}, we see that $\alpha_{n+1}-\alpha_n\in\Z v\ell_n$, so we may take $a_n$ with~$va_n=\alpha_n$ and~$a_{n+1}\in\ell_n^\Z a_n$.
    
    Set $p_n:=\dwm P^{\phi_n}_{\times a_n}$.  Then  $0\leq p_{n+1}\leq p_n\leq n_0$, and by Corollary~\ref{conjugate and equalize},  if~$p_{n+1}=p_n>0$, then $\alpha_{n+1}=\alpha_n$ and $p_{n+2}\in\{0,p_n\}$.  If~$p_{n+1}=p_n$, then as in the proof of Proposition~\ref{prop:bound on comp conj} we have $P^{\phi}_{\times a_n}\asymp(\phi/\phi_n)^{p_n}P^{\phi_n}_{\times a_n}$ for all~$\phi\dominatedby\phi_n$. Likewise with $Q$ and $q_n:=\dwm Q^{\phi_n}_{\times a_n}$ in place of $P$ and $p_n$.
    We claim that $p_{n+1}=q_{n+1}=p_{n}=q_{n}$ for some $n\leq n_0$, which implies the proposition.

    Let $k$ be minimal such that $p_k=p_{k+1}$ or $q_k=q_{k+1}$.  Then $k\leq n_0-p_k$, $p_k=q_k$, and $\alpha_{k+1}=\alpha_k$.  If $p_k=q_k=0$, then we are done.

    Suppose $p_k=q_k=1$.  There is only one monomial of degree $d$ and weight $1$, so we have $\big(P^{\phi_k}_{\times a_k}\big)_{[1]}\in K Y^{d-1}Y'$, and similarly for $Q$.  We are thus in case (2) of Lemma~\ref{conjugate and equalize}, and $p_{k+1}=q_{k+1}=1$.

    Suppose $p_k\geq2$, so $k\leq n_0-2$, and $p_{k+1}=p_k$ but $q_{k+1}<q_k$.  Then $p_{k+1}\neq q_{k+1}$, so $p_{k+2}\in\{0,p_k\}$ but $p_{k+2}\neq p_{k+1}$, i.e., $p_{k+2}=0$.  Then for any $n>k$, if~$q_n>0$ then~$p_n\neq q_n$ and $q_{n+1}<q_n$, so by induction $q_n\leq \max\{0,q_k-(n-k)\}$.  In particular, $q_{n_0}=0$.  The case $q_{k+1}=q_k$, $p_{k+1}<p_k$ being symmetric, this completes the proof.
\end{proof}

\noindent
Next we apply the above to $K:=\T$.
We recall that the \emph{exponential transseries} are those transseries which can be constructed by using only exponentiation and infinite summation. (See \cite[Appendix~A]{ADAMTT}.) 
They form  a grounded asymptotic subfield $\T_{\exp}$ of $K$ with $\max\Psi_{\T_{\exp}}=v(1/x)$.
For any~$f\in\T$, there is an~$m$ such that the $m$th upward shift $f{\uparrow^m}$ of $f$ is exponential; likewise, if~$Q\in\T\{Y\}$, then~$Q{\uparrow^m}\in\T_{\exp}\{Y\}$ for sufficiently large $m$.  Also set $\log_0x:=x$, $\log_nx:=\log(\log_{n-1} x)$ for~$n\geq1$, and~$\upg_n:=(\log_{n+1}x)'$, $\upg_{-1}:=1$.  Using Proposition~\ref{prop:bound on comp conj} we now derive:

\begin{cor}\label{cor:newton poly, upward shift}
Suppose that $P{\uparrow}^m\in\T_{\exp}\{Y\}$.   Then 
 \[
 D_{P{\uparrow^n}}\ =\ N_P\in \R[Y](Y')^\N\quad\text{for all~$n\geq  \dwm(P)+m+2$.}
 \]
\end{cor}
\begin{proof}
First assume $m=0$, so $P\in \T_{\exp}\{Y\}$, and   $n\geq \dwm(P)+2$. Then~$D_{P{\uparrow^n}}=D_{P^{\upg_{n-1}}}$ by \cite[Lemma 13.3.16]{ADAMTT}.  We now apply the above to $K_0:=\T_{\exp}$, $\phi_0:=1/x$, and $\ell_n:= \log_{n+1}x$. 
Then $\upg_{m}=\phi_m$ for all $m$, so $\upg_{n-1}\dominatedby\upg_{n_0}=\phi_{n_0}$ where~$n_0:=\dwm(P)+1$. Hence by Proposition~\ref{prop:bound on comp conj}  we have $D_{P{\uparrow^n}}=D_{P^{\upg_{n-1}}}=N_P\in \R[Y](Y')^\N$ as required.  The general case follows from the special case applied to $Q:=P{\uparrow}^m$ in place of $P$,
using $\dwm(P)\geq\dwm(Q)$ and $P{\uparrow^{n}}=Q{\uparrow^{n-m}}$ for $n\geq m$. 
\end{proof}

\noindent
Similarly, Proposition~\ref{prop:weight bound for eventual equalizers} implies:

\begin{cor}\label{cor:eventual equalizers, upward shift}
Suppose that $P{\uparrow}^m,Q{\uparrow}^m\in\T_{\exp}\{Y\}$, and that $P$ and $Q$ are nonzero and homogeneous of different degrees.  Then for all $n\geq\max\{\wt P,\wt Q\}+m$, the equalizer of $P{\uparrow}^n$ and $Q{\uparrow}^n$ is the eventual equalizer of $P$ and $Q$.
\end{cor}

\noindent
Corollary~\ref{cor:newton poly, upward shift} corrects a statement made in the ``Notes and comments'' of \cite[Section~13.7]{ADAMTT}. 
In the case where $\wt(P)\ge\dwm(P)+2$, it also improves the bound shown  (strictly speaking, for the differential subfield $\mathbb T_{\operatorname{g}}$ of grid-based transseries in place of $\T$) in~\cite[Theorem~8.6]{van2006transseries}.  Corollary~\ref{cor:eventual equalizers, upward shift} is essentially~\cite[Proposition~8.14]{van2006transseries} (for $\T$ rather than $\T_{\operatorname{g}}$).

\section{Hard-to-Find Equalizers}\label{sec:equalizer example}

\noindent
\emph{In this appendix, $K$ is a valued differential field with small derivation.}
In the proof of Theorem~\ref{equalizerthmriccati}, we constructed a certain map $g\colon \Gamma\to\Gamma$ and showed that our desired~$\alpha=va$ was the unique fixed point of $g$ and that $g^n(\alpha_0)=g^{n+1}(\alpha_0)$ for any~$\alpha_0\in\Gamma$ and $n=\deg P+\deg Q+1$.  In this appendix, we show that this is sharp by constructing (ordinary) polynomials~$P$,~$Q$ of specified degrees $d,e$ and $\alpha_0$ such that $g^{d+e}(\alpha_0)\neq g^{d+e+1}(\alpha_0)$.
Fix $d,e\in\N$, $p,q\in\Q^{\neq}$, and consider
\[
P\, :=\, a_0+a_1Y+\cdots+a_dY^d, \ Q\, :=\, b_0+b_1Y+\cdots+b_eY^e \in K[Y] \qquad (a_i, b_j\in K).
\]
If $\ddeg P=0$, then for any $c\dominates1$ in $K$, we have  
$$v(P_{+c})\geq\min\{va_i+ivc\, :\, i=0,\ldots,d\},$$ with equality whenever the $va_i+ivc$ are distinct.  In particular, $vP_{+c}$ is controlled by~$\qddeg_{vc}P$ and $v_{vc}P$. (See Definition~\ref{def:qddeg}.) Our goal is therefore first to arrange for the existence of $c_d\dominates c_{d-1}\dominates\cdots\dominates c_0$ such that $\qddeg_{vc_i}P=i$ for $i=0,\ldots,d$, and then to ensure that the $a_n^\dagger$ generated by the iteration have the same property.
Now suppose that
\[
1\ \prec\ \frac{a_0}{a_1}\ \prec\ \frac{a_1}{a_2}\ \prec\ \cdots\ \prec\ \frac{a_{d-1}}{a_d}.
\]
Let $c\in K^{\times}$, and suppose that $\frac{a_{n-1}}{a_n}\prec c\prec\frac{a_n}{a_{n+1}}$ for some $n\in\{0,\ldots,d\}$, where we set $a_{-1}:=a_0$ and $a_{d+1}:=0$.  Then 
\[
v(a_0)\ >\ v(ca_1)>\ \cdots\ >\ v(c^na_n)\ <\ v(c^{n+1}a_{n+1})\ <\ \cdots\ <\ v(c^da_d),
\]
so for any $0\le i\le d$,
\[
v(P_{+c,i})\ =\ v\bigg(\sum_{j\ge i}\binom{j}{i}c^{j-i}a_j\bigg)\ =\ v(c^{k-i}a_k) \qquad\text{where $k:=\max\{n,i\}$.}
\]
In particular, $v(P_{+c})=v(c^na_n)$.
\medskip

\noindent
Let $m:=d+e$, and suppose that $c_0,\ldots,c_m\in K$ are such that 
\[
c_m\ \succ\ c_m^\dagger\ \succ\ c_{m-1}\ \succ\  c_{m-1}^\dagger\ \succ\ c_{m-2}\ \succ\ \cdots\ \succ\ c_0\ \succ\ 1.
\]
Set $\eta_j:=vc_j$ and $\delta_j:=\nabla\eta_j$ for $j=0,\ldots,m$.  Let 
\[
\begin{aligned}
a_0\ &:=\ 1,& \ a_1\ &:=\  c_0^{-1},& \ a_2\ &:=\  c_0^{-1}c_1^{-1},& \ &\ldots,& \ a_d\ &:=\ c_0^{-1}c_1^{-1}\cdots c_{d-1}^{-1},\\b_0\ &:=\ 1,& \ b_1\ &:=\  c_d^{-1}, &\ b_2\ &:=\  c_d^{-1}c_{d+1}^{-1},& \ &\ldots,& \ b_e\ &:=\ c_d^{-1}c_{d+1}^{-1}\cdots c_{m-1}^{-1}.
\end{aligned}
\]
Then by the above, 
\[
v(P_{+pc_j^\dagger,i})\ =\ \begin{cases}va_j+(j-i)\delta_j&\text{ if }\ i\le j\le d,\\va_d+(d-i)\delta_j&\text{ if }\ i\le d\le j,\\va_i&\text{ if }\ i\ge j\end{cases}\]
and
\[v(Q_{+qc_j^\dagger,i})\ =\ \begin{cases}vb_{j-d}+(j-d-i)\delta_j&\text{ if }\ i\le j-d,\\vb_i&\text{ if }\ i\ge j-d,\end{cases}
\]
where $0\leq i\leq d$ in the first and $0\leq i\leq e$ in the second.  In particular,
\[
v(P_{+pc_j^\dagger})\ =\ \begin{cases}va_j+j\delta_j&\text{ if }\ 0\le j\le d,\\va_d+d\delta_j&\text{ if }\ j\ge d\end{cases}\]
and \[v(Q_{+qc_j^\dagger})\ =\ \begin{cases}vb_0&\text{ if }\ j\le d,\\vb_{j-d}+(j-d)\delta_j&\text{ if }\ j\ge d.\end{cases}
\]
Recall the functions $f,\gamma,g\colon\Gamma\to\Gamma$ defined by
\[
\begin{aligned}
f(\alpha)\ &=\ vP_{+pa^\dagger}-vQ_{+qa^\dagger},\\
\gamma(\alpha)\ &=\ 2\nabla\big(f(\alpha)-\alpha\big),\\
g(\alpha)\ &=\ v_{\gamma(\alpha)}P_{+pa^\dagger}-v_{\gamma(\alpha)}Q_{+qa^\dagger},
\end{aligned}
\]
where $a\in K$ is such that $va=\alpha$.
The above calculation shows that
\[
\begin{split}
f(\eta_j)\ &=\ \begin{cases}va_j-vb_0+j\delta_j&\text{ if }\ j\le d,\\va_d-vb_{j-d}+(2d-j)\delta_j&\text{ if }\ j\ge d\end{cases}\\
&=\ O(\delta_j)
\end{split}
\]
for $j=0,\ldots,m$.  Since $\delta_j=\nabla \eta_j$, Lemma~\ref{nablalemma}(\ref{nablaslow}) yields $\gamma(\eta_j)=2\nabla(f(\eta_j)-\eta_j)=2\delta_j$.  From the calculation of $v(P_{+pc_j^\dagger,i})$ and~$v(Q_{+qc_j^\dagger,i})$, we then obtain $$\qddeg_{\gamma(\eta_j)}P_{+pc_j^\dagger}=\min(j,d),\qquad \qddeg_{\gamma(\eta_j)}Q_{+qc_j^{\dagger}}=\max(0,j-d).$$  
Combining these, 
\[
g(\eta_j)\ =\ \begin{cases}va_j-vb_0&\text{ if }\ j\le d,\\va_d-vb_{j-d}&\text{ if }\ j\ge d.\end{cases}
\]
Since $va_0=vb_0=0$ and 
\begin{align*}
va_j &=-\eta_{j-1}+O(\delta_{j-1})\quad\text{for $1\le j\le d$,} \\
vb_{j-d}&=-\eta_{j-1}+O(\delta_{j-1})\quad\text{for $d+1\le j\le m$,}
\end{align*}
we have found that $g(\eta_j)=\pm\eta_{j-1}+O(\delta_{j-1})$ for $1\le j\le m$.

\medskip
\noindent
Our calculation of $v(P_{+c})$ and $v(Q_{+c})$ only depended on $vc$. Hence, since  for $j\geq1$ we have
$\nabla\big({\pm\eta_{j-1}+O(\delta_{j-1})}\big)=\nabla\eta_{j-1}$, we obtain:
if $vc=g(\eta_j)$, then 
$$vP_{+pc^\dagger}=vP_{+pc_{j-1}^\dagger},\ vQ_{+qc^\dagger}=vQ_{+qc_{j-1}^\dagger},\ f(g(\eta_j))=f(\eta_{j-1})=O(\delta_{j-1}).$$  Then by Lemma~\ref{nablalemma}(\ref{nablaslow}), $$\gamma(g(\eta_j))=2\delta_{j-1}=\gamma(\eta_{j-1}),$$ and by Lemma~\ref{basicndprops}(\ref{ndconj}),  $$g(g(\eta_j))=g(\eta_{j-1})\quad\text{ for $1\le j\le m$.}$$
We therefore have $g^k(\eta_m)=g(\eta_{m-k+1})$ for $1\le k\le m+1$.  The penultimate term in this sequence is $g^m(\eta_m)=g(\eta_1)=\pm\eta_0+O(\delta_0)\ne0$, and the final term is $g(\eta_0)=0$, the fixed point of $f$.  Thus with $\alpha_0:=\eta_m$ and $n:=m+1=\deg P+\deg Q+1$, we obtain $g^{n-1}(\alpha_0)\neq g^n(\alpha_0)$, as desired.

\medskip
\noindent
Concretely, in $\T$ we can take 
$$c_0:=x^{-1}\ex^x,\quad c_1:=\ex^{\ex^x}, \quad\text{and}\quad c_{j+1}:=\ex^{c_j}\text{ for $j\geq1$.}$$  (In fact, we can take $c_0:=\pi\ex^x$; we specified $c_j\prec c_{j+1}^\dagger$ instead of $c_j\dominatedby c_{j+1}^\dagger$ to prevent cancellation, which can also be accomplished with an appropriate coefficient.)

\medskip
\noindent
The above calculation of $f(\eta_j)$ relied only on the inequalities~$vc_j<\nabla\eta_j<vc_{j-1}$, so for any $\eta$ such that $\delta_j\le\nabla\eta<\eta_{j-1}$, we have
\[
f(\eta)\ =\ \begin{cases}va_j-vb_0+j\nabla\eta&\text{ if }\ j\le d,\\va_d-vb_{j-d}+(2d-j)\nabla\eta&\text{ if }\ j\ge d.\end{cases}
\]
If $c_{j-1}\prec c_j^{\langle n\rangle}$ for all $n\in \N$, where $a^{\langle0\rangle}:=a$, $a^{\langle n \rangle}:=(a^{\langle n-1\rangle})^\dagger$ for $n\geq1$, then
$$f^n(\eta_j)=j\nabla^n\eta+O(\nabla^{n+1}\eta)\quad\text{ for all $n$.}$$  Taking $c_{-1}=1$ and $c_{j-1}\prec c_j^{\langle n\rangle}$ for all $j$ and $n$, iterating $f$ would require an iteration of transfinite length at least~$m\omega$.

\bibliographystyle{amsplain}	
\bibliography{refs}

\end{document}